\documentclass[11pt, reqno]{amsart}
\usepackage{amssymb}
\usepackage{amsmath}
\usepackage{enumerate}
\usepackage{mathrsfs}
\usepackage{amsfonts}  
\usepackage{color}
\usepackage{vmargin}
\usepackage{amsthm}
\usepackage{graphicx}  
\usepackage{esvect}
\usepackage{MnSymbol}
\usepackage[center]{caption}

\def\Xint#1{\mathchoice
{\XXint\displaystyle\textstyle{#1}}%
{\XXint\textstyle\scriptstyle{#1}}%
{\XXint\scriptstyle\scriptscriptstyle{#1}}%
{\XXint\scriptscriptstyle\scriptscriptstyle{#1}}%
\!\int}
\def\XXint#1#2#3{{\setbox0=\hbox{$#1{#2#3}{\int}$ }
\vcenter{\hbox{$#2#3$ }}\kern-.6\wd0}}

\def\dashint{\Xint-}

\DeclareMathOperator*{\diam}{diam}

\makeatletter
\@namedef{subjclassname@2020}{%
  \textup{2020} Mathematics Subject Classification}
\makeatother
\long\def\symbolfootnote[#1]#2{\begingroup%
\def\thefootnote{\fnsymbol{footnote}}\footnote[#1]{#2}\endgroup}

\setmarginsrb{20mm}{20mm}{20mm}{20mm}{10mm}{10mm}{10mm}{10mm}

\newtheoremstyle{remark}
  {}{}{}{}{\bfseries}{.}{.5em}{{\thmname{#1 }}{\thmnumber{#2}}{\thmnote{ (#3)}}}
\theoremstyle{remboldstyle}

\RequirePackage{amsthm}
\newcommand{\measurerestr}{%
  \,\raisebox{-.127ex}{\reflectbox{\rotatebox[origin=br]{-90}{$\lnot$}}}\,%
}
\newtheorem{tw}{Theorem}[section]
\newtheorem{defi}[tw]{Definition}

\newtheorem{corollary}[tw]{Corollary}
\newtheorem{lem}[tw]{Lemma}
\newtheorem{cor}[tw]{Corollary}
\newtheorem{prop}[tw]{Proposition}
\newtheorem*{rem}{Remark}
\newtheorem{ex}[tw]{Example}
\newtheorem{conv}[tw]{Convention}
\def\={\hspace{-3mm}&=&\hspace{-3mm}}

\usepackage{tikz}
\usetikzlibrary{arrows,calc,patterns,cd}
\usepackage{wrapfig}

\let\oldmarginpar\marginpar
\renewcommand{\marginpar}[2][rectangle,draw,text width= 1.5cm,rounded corners]{
    \oldmarginpar{
    \scriptsize \tikz \node at (0,0) [#1]{#2};}
    }
\makeatletter
\renewcommand{\tocsection}[3]{%
  \indentlabel{\@ifnotempty{#2}{\bfseries\ignorespaces#1 #2\quad}}\bfseries#3}
\renewcommand{\tocsubsection}[3]{%
  \indentlabel{\@ifnotempty{#2}{\ignorespaces#1 #2\quad}}#3}
\renewcommand{\tocsubsubsection}[3]{%
  \indentlabel{\@ifnotempty{#2}{\ignorespaces#1 #2\quad}}#3}

\newcommand\@dotsep{4.5}
\def\@tocline#1#2#3#4#5#6#7{\relax
  \ifnum #1>\c@tocdepth % then omit
  \else
    \par \addpenalty\@secpenalty\addvspace{#2}%
    \begingroup \hyphenpenalty\@M
    \@ifempty{#4}{%
      \@tempdima\csname r@tocindent\number#1\endcsname\relax
    }{%
      \@tempdima#4\relax
    }%
    \parindent\z@ \leftskip#3\relax \advance\leftskip\@tempdima\relax
    \rightskip\@pnumwidth plus1em \parfillskip-\@pnumwidth
    #5\leavevmode\hskip-\@tempdima{#6}\nobreak
    \leaders\hbox{$\m@th\mkern \@dotsep mu\hbox{.}\mkern \@dotsep mu$}\hfill
    \nobreak
    \hbox to\@pnumwidth{\@tocpagenum{\ifnum#1=1\bfseries\fi#7}}\par% <-- \bfseries for \section page
    \nobreak
    \endgroup
  \fi}
\AtBeginDocument{%
\expandafter\renewcommand\csname r@tocindent0\endcsname{0pt}
}
\def\l@subsection{\@tocline{2}{0pt}{2.5pc}{5pc}{}}
\def\l@subsubsection{\@tocline{2}{0pt}{4.6pc}{1pc}{}}
\makeatother

\begin{document}

\date{}

\title[Compact embeddings of Sobolev, Besov, and Triebel--Lizorkin spaces]{\bf Compact embeddings of Sobolev, Besov, and Triebel--Lizorkin spaces}

\author{Ryan Alvarado, Przemys{\l}aw G\'{o}rka and Artur S{\l}abuszewski}%\\


\maketitle
\begin{abstract}
We establish necessary and sufficient conditions guaranteeing compactness of embeddings of fractional Sobolev spaces, Besov spaces, and Triebel--Lizorkin spaces, in the general context of quasi-metric-measure spaces.  Although stated in the setting of quasi-metric spaces, the main results in this article are new, even in the metric setting. Moreover, by considering the more general category of quasi-metric spaces we are able to obtain these characterizations for optimal ranges of exponents that depend (quantitatively) on the geometric makeup of the underlying space. 
\end{abstract}
%\bigskip

\tableofcontents

\section{Introduction}
\label{sect:intro}

Compact embeddings of certain classes of function spaces, such as Sobolev spaces, have several important applications in analysis and are particularly useful in theory of elliptic differential operators. Some of the most fundamental results in this regard are the Sobolev embedding and Rellich--Kondrachov compactness theorems which state, among other things, that if $\Omega\subseteq\mathbb{R}^n$ is a bounded domain with a sufficiently regular boundary then for   $p\in[1,n)$, the classical first-order Sobolev space $W^{1,p}(\Omega)$ continuously embeds into the Lebesgue space $L^{np/(n-p)}(\Omega)$ and compactly embeds into  $L^{\tilde{p}}(\Omega)$ whenever $\tilde{p}\in[1,np/(n-p))$; see for instance \cite[6.2~Theorem]{ad} and the references therein. 

%Continuous and compact embeddings for various scales of function spaces have been further studied by many individuals in settings much more general than the Euclidean one, including Riemannian manifolds and (quasi-)metric measure spaces \cite{AYY22,AYY21,BK,G,G17,GS23,GS,Hajlasz2,Hajlasz,HK,hhhpl21,hebey,HK21,K,Karak1}.

Continuous and compact embeddings for various scales of function spaces have been further studied by many individuals in settings much more general than the Euclidean one, including Riemannian manifolds and (quasi-)metric-measure spaces; see \cite{agh20,AYY22,AYY21,BK,GGP,G,G18,G17,GK,GS23,GS,Hajlasz2,Hajlasz,HK,hhhpl21,hebey,HK21,IK,K,Karak1,Krotov,R}. In particular, under various assumptions, compact embeddings in the spirit of the classical Rellich--Kondrachov theorem for (fractional) Haj\l{}asz--Sobolev and Poincar\'e spaces defined on metric-measure spaces were established by Haj\l{}asz--Koskela \cite{HK}, Ka\l{}amajska \cite{K},  Bj\"orn--Ka\l{}amajska \cite{BK}, and G\'{o}rka--S{\l}abuszewski \cite{GS} (see also \cite{GS} for results for fractional Slobodeckij--Sobolev spaces on metric spaces,
\cite{GK} for results for Sobolev spaces on metrizable groups, as well as \cite{IK}, \cite{Krotov}, and \cite{R} for results for certain generalized classes of Sobolev spaces on metric spaces).
Recently, Alvarado--Yang--Yuan \cite{AYY22,AYY21} identified necessary and sufficient conditions guaranteeing that the Sobolev embedding theorem holds for the fractional Haj\l{}asz--Sobolev spaces ${M}^{\alpha,p}$, the Haj\l{}asz--Besov spaces ${N}^\alpha_{p,q}$, and the Haj\l{}asz--Triebel--Lizorkin spaces ${M}^\alpha_{p,q}$ (introduced in \cite{Hajlasz,KYZ11,Y03}) in the general context of quasi-metric-measure spaces (see  \cite{Karak1,Karak2} for some partial results and \cite{HK21} for  results concerning Orlicz--Sobolev spaces). 

The primary goal of this article is to extend and refine some of the work presented in \cite{BK,GS,K}. More specifically, for an optimal range of the smoothness parameter, $\alpha$, and under minimal assumptions on the ambient space, we study compact embeddings of ${M}^{\alpha,p}(X)$, ${N}^\alpha_{p,q}(X)$, and ${M}^\alpha_{p,q}(X)$ into $L^{\tilde{p}}(X)$ spaces (equipped with a possibly different measure) in terms of the geometric and measure-theoretic properties of the underlying (quasi-)metric-measure space $(X,d,\mu)$. In particular, improve upon previous results by establishing these characterizations under weaker assumptions than were considered in \cite{BK,GS,K}, such as total boundedness of the metric space and doubling conditions on the metric space or measure. 
%; see Theorems~\ref{drugie}, \ref{drugie_plus} and \ref{snr-3} which separately address the cases $\tilde{p}=p$, $\tilde{p}\leq p$, and $\tilde{p}>p$, respectively. 

We stress here that we do not work with quasi-metrics for sake of generality. In fact, if we worked solely within the class of metrics and the standard triangle inequality
then we would not be able to establish the full strengths of our results. In place of metrics and the triangle inequality, we work with the more inclusive family of quasi-metrics satisfying the following quasi-subadditivity condition\footnote{It is easy to see that \eqref{TR-ineq.2} is equivalent to the so-called quasi-triangle inequality: $d(x,y)\leq C'\left[d(x,z)+d(z,y)\right].$}:
\begin{equation}\label{TR-ineq.2}
d(x,y)\leq C_d\max \left\{d(x,z),d(z,y)\right\},\ \forall\,x,y,z\in X,
\end{equation}
where $C_d\in[1,\infty)$ is an optimal constant\footnote{See (\ref{C-d}) in the next section.} independent of $x$, $y$, and $z$. Note that every metric satisfies \eqref{TR-ineq.2} for some $C_d\in[1,2]$, and a metric $d$ satisfying \eqref{TR-ineq.2} with $C_d=1$  is commonly referred to as an ultrametric. 
At the heart of why we need to work with quasi-metrics lies the fact that  \eqref{TR-ineq.2} (specifically the constant $C_d$) captures the interplay between the (quasi-)metric and the geometry of space $X$ better than the (quasi-)triangle inequality. For instance, the  Euclidean distance $|\,\cdot-\cdot\,|$ and its snowflaked counterpart, $|\,\cdot-\cdot\,|^{s}$,  $s\in(0,1)$, both satisfy the triangle inequality on $\mathbb{R}^n$; however, the constant in \eqref{TR-ineq.2}  for $|\,\cdot-\cdot\,|$ and $|\,\cdot-\cdot\,|^s$ can be taken to be $C_{|\,\cdot-\cdot\,|}=2$ and $C_{|\,\cdot-\cdot\,|^s}=2^s$, respectively. The inability of triangle inequality to detect these differences leads to a suboptimal function space theory, even in the metric setting, and it is for this reason that it is necessary for us to work in the general setting of quasi-metric spaces. As our main results reveal, the optimality of the range for $\alpha$ is directly related to the constant $C_d$ in \eqref{TR-ineq.2} and the nature of the ``best" quasi-metric on $X$ which is bi-Lipschitz equivalent to $d$, where the latter notion is quantified via the lower smoothness index of $(X,d)$, denoted by ${\rm ind}(X,d)$; see Subsection~\ref{subsect:QMS} for the definition and the properties of ${\rm ind}(X,d)$. Although our work is carried out in the context of quasi-metric spaces, many of our main results, to the best of our knowledge, are brand new, even in the metric setting. 

%The consequences of \eqref{TR-ineq.2} as it relates to the optimality of our main results shall be demonstrated below. More details on the underlying spaces are listed in Section \ref{s-set} below.

%As we will demonstrate throughout the optimality of the range for $\alpha$ is directly related to the constant $C_d$ in \eqref{TR-ineq.2} and the nature of the ``best" quasi-metric on $X$ which is bi-Lipschitz equivalent to $d$, where the latter notion is quantified via the \textit{lower smoothness index of $(X,d)$}, denoted by ${\rm ind}(X,d)$ (see \eqref{index} for a precise definition). Although our work is carried out in the context of quasi-metric spaces, many of our main results, to the best of our knowledge, are brand new, even in the metric setting. 
%%In particular, Theorems~\ref{drugie}, \ref{piate}, and \ref{drugie_plus} extend and refine some of the recent work carried out in \cite{BK,GS} for fractional Sobolev spaces ${M}^{\alpha,p}$. 
%We will expand more on this aspect below.%In order to illustrate the optimality of our assumptions, we include several examples and observations throughout this paper. 

Even though all of the main results in this article concern the spaces ${M}^{\alpha,p}$, ${N}^\alpha_{p,q}$, and ${M}^\alpha_{p,q}$, for the sake of exposition and to put our results into context with those in the existing literature, in this introduction we will simply record our main conclusions for ${M}^{\alpha,p}={M}^\alpha_{p,\infty}$ spaces and refer the reader to the corresponding, more general, statements in the body of the article. 

Turning to specifics, an interesting facet of our main results is how the geometric nature of the optimal conditions guaranteeing compactness of the aforementioned embeddings varies depending on if $\tilde{p}=p$, $\tilde{p}\leq p$, or $\tilde{p}>p$.
Of the three cases, compactness of embeddings when  $\tilde{p}=p$, that is, for embeddings of the form
\begin{equation}
\label{embedintro}
{M}^{\alpha,p}(X)\hookrightarrow L^p(X)
\end{equation} 
turns out to be the most delicate as the geometric and measure-theoretic properties of the underlying (quasi-)metric-measure space intervene in these embedding theorems in a rather subtle manner. Indeed, for metric-measure spaces, Ka\l{}amajska \cite[Theorem~2]{K} proved that the embedding ${M}^{1,p}(X)\hookrightarrow L^p(X)$ is compact for $p\in[1,\infty)$, provided the metric-measure space is totally bounded\footnote{Compare the assumptions in \cite[Theorem~2]{K} with Proposition~\ref{tot} in this work.} and satisfies either a weakened geometric doubling or a weakened measure doubling type condition. Maintaining a very similar set of assumptions, Bj\"orn--Ka\l{}amajska in \cite[Proposition~3.9]{BK} extended \cite[Theorem~2]{K} to compact embeddings of ${M}^{\alpha,p}(X)$ into $L^p(E)$ (defined on a measurable and totally bounded set $E\subseteq X$) for certain ranges of $\alpha$ and $p$ that depend on the weakened geometric doubling and measure doubling type conditions. Compactness of the embedding \eqref{embedintro}
for all $\alpha\in(0,\infty)$ and $p\in[1,\infty)$ was provided in \cite[Proposition~1.2]{BK} under the stronger assumptions that $X$ is bounded and geometrically doubling, in the sense that every ball of radius $r$ can be covered by at most a fixed number of balls of radius $r/2$.\footnote{Every bounded and geometrically doubling quasi-metric space is totally bounded.} In \cite[Theorem~4.1]{GS}, G\'{o}rka--S{\l}abuszewski extended \cite[Propositions~1.2]{BK} (and \cite[Propositions~3.9]{BK} in the case $E=X$) by illustrating that total boundedness of the metric-measure space alone, without any additional geometric doubling or measure doubling type conditions, is enough to ensure that the embedding \eqref{embedintro} is compact for all $\alpha,p\in(0,\infty)$. In our first main result, we refine this line of work by further relaxing the total boundedness assumption on the underlying space using a new integrability condition on the measure (see Subsection~\ref{subsect: integrable measures} for specific definitions, examples, and properties).
\begin{tw}\label{remarkoint-intro}
Let $(X,d,\mu)$ be a quasi-metric-measure space and suppose that $\mu$ is integrable, in the sense that for each fixed $r\in(0,\infty)$, the mapping $x\mapsto\mu(B_d(x,r))$ is measurable\footnote{If $(X,d)$ is a metric space then the  mapping $x\mapsto\mu(B_d(x,r))$ is always measurable.} and
\begin{align*}
\int_{X} \frac{1}{\mu(B_{d}(x,r))} \, d\mu(x)<\infty,
\end{align*}
where $B_d(x,r)\subseteq X$ denotes the (quasi-)metric ball with center $x\in X$ and radius $r\in(0,\infty)$.
Then, for any $\alpha,p\in(0,\infty)$, the embedding ${M}^{\alpha,p}(X)\hookrightarrow L^p(X)$ is compact. 
\end{tw}
The reader is directed to Theorem~\ref{drugie} and Corollary~\ref{remarkoint} for a stronger, more complete account of Theorem~\ref{remarkoint-intro} which includes embeddings of ${N}^\alpha_{p,q}(X)$ and ${M}^\alpha_{p,q}(X)$ into $L^p(X)$, equipped with a different measure $\nu$.

Theorem~\ref{remarkoint-intro} fully extends and generalizes the compactness results in \cite{GS,K} and \cite{BK} (in the case $E=X$), even in the metric setting. Indeed, we prove in Proposition~\ref{TB-int} that if $(X, d,\mu)$ is a totally bounded quasi-metric-measure space, then $\mu$ is integrable. We go further and provide examples of integrable measures, including Gaussian-type measures, that are defined on unbounded metric spaces (see Examples~\ref{exint1} and \ref{expbeta}) as well as an example (see Example~\ref{infinitecomb}) of an integrable measure on a bounded metric space with the property that every ball is not totally bounded (hence the metric space is not totally bounded or geometrically doubling). In particular, Theorem~\ref{remarkoint-intro} is valid in these metric-measure spaces, but one cannot appeal to \cite{BK,GS,K} to conclude compactness of the embedding in \eqref{embedintro} as these results are not applicable in these settings.

To establish Theorem~\ref{remarkoint-intro}, we prove a new compactness criterion for sets in $L^{\tilde{p}}(X)$, where $X$ is separable quasi-metric space of finite measure; see Theorems~\ref{compLp} and \ref{compL0}. We show that this criterion is actually equivalent to compactness in $L^{\tilde{p}}(X)$ under minimal assumptions on the underlying measure space. This generalizes the results of \cite[Theorem~2]{Krotov} (established in the setting of bounded metric spaces equipped with a doubling measure) and \cite[Theorem~1.3]{BG} (which dealt with the case of totally bounded metric-measure spaces). This new criterion enables us to circumvent the need to assume that the quasi-metric space is totally bounded or geometrically doubling, or that the measure is doubling. 

Regarding necessary conditions under which the embedding \eqref{embedintro} is compact, it follows from \cite[Proposition~7.4]{BK} that if $\mu$ is a doubling measure\footnote{A quasi-metric-measure space with doubling measure is geometrically doubling.} on a metric space $(X,d)$, in the sense that there exists a positive  constant $C$ such that
\begin{equation}\label{doublingmeasure}
\mu(2B)\leq C\mu(B)
\end{equation}
for all balls $B\subseteq X$, where $2B$ denotes the ball with the same center as $B$ having twice its radius,  then compactness of the embedding \eqref{embedintro}, with $\alpha=1$, implies that $(X,d)$ is totally bounded. One of the main goals of this article is to clarify the degree to which the doubling condition \eqref{doublingmeasure} is necessary in this context. In a step towards this goal, we establish the following characterization which extends and refines \cite[Proposition~7.4]{BK} (in the case $E=X$) to ${M}^{\alpha,p}$ spaces defined on quasi-metric-measure spaces satisfying a doubling-type measure condition which is strictly weaker than \eqref{doublingmeasure}. To state this result, we introduce the following notional convention: Given a quasi-metric space $(X,d)$ and a number $\alpha\in(0,\infty)$, we will understand by $\alpha\preceq{\rm ind}(X,d)$ that $\alpha\leq{\rm ind}(X,d)$ and that the value $\alpha={\rm ind}(X,d)$ is only permissible when  the supremum defining ${\rm ind}(X,d)$ in \eqref{index} is attained.

%Under a weakened notion of doubling measure, where the factor 2 is replaced by $C_d$ and the constant $C$ appearing in \eqref{doublingmeasure} is allowed to depend on the radius of the ball, we manage to extend this result to ${M}^{\alpha,p}$ spaces defined on a quasi-metric-measure space, $(X,d,\mu)$, for an optimal range of $\alpha$. More specifically, we prove the following characterization. 

%\begin{tw} \label{czwarte-inftro}
%Let $(X,d,\mu)$ be a quasi-metric-measure space and suppose there exists $\delta_0\in(0,\infty)$ such that
%$\mu$ is $(C_d,\delta)$-doubling for every $\delta\in(0,\delta_0]$, where $C_d\in[1,\infty)$ is as in \eqref{TR-ineq.2}.
%Fix $\alpha,p\in(0,\infty)$  and suppose that $\alpha\leq{\rm ind}(X,d)$. If the embedding
%${M}^{\alpha,p}(X)\hookrightarrow L^p(X)$ is compact, then $(X,d)$ is totally bounded.
%\end{tw}

\begin{tw} \label{piate-intro}
Let $(X,d,\mu)$ be a quasi-metric-measure space and suppose there exists $\delta_0\in(0,\infty)$ such that $\mu$ is $(C_d,\delta)$-doubling for every $\delta\in(0,\delta_0]$, in the sense that for each $\delta\in(0,\delta_0]$, there exists a constant $C(\delta)\in[1,\infty)$ such that
\begin{equation}
\label{deltadoubling}
\mu(B_d(x,C_d\delta))\leq C(\delta)\mu(B_d(x,\delta))\quad\,\mbox{for all $x\in X$,}
\end{equation}
where $C_d\in[1,\infty)$ is as in \eqref{TR-ineq.2}. Then the following statements are equivalent.
\begin{enumerate}
\item $(X,d)$ is totally bounded.
\item $\mu$ is integrable, in the sense described in Theorem~\ref{remarkoint-intro}.
\item For all exponents $\alpha,p\in(0,\infty)$ with $\alpha\preceq{\rm ind}(X,d)$, the embedding ${M}^{\alpha,p}(X)\hookrightarrow L^p(X)$ is compact.
\item There exist exponents $\alpha,p\in(0,\infty)$ with $\alpha\preceq{\rm ind}(X,d)$, the embedding ${M}^{\alpha,p}(X)\hookrightarrow L^p(X)$ is compact.
\end{enumerate}
\end{tw}
Theorem~\ref{piate-intro} is a direct consequence of Proposition~\ref{TB-int}, Corollary~\ref{remarkoint}, and  Theorem~\ref{czwarte}, which are stated for the spaces ${N}^\alpha_{p,q}$ and ${M}^\alpha_{p,q}$ (see also Theorem~\ref{piate}). Note that every doubling measure on a metric space (as in \eqref{doublingmeasure}) trivially satisfies \eqref{deltadoubling} since $C_d\in[1,2]$ whenever $d$ is a metric. However, measures satisfying \eqref{deltadoubling} need not even be locally doubling (see Example~\ref{nonlocaldoub}). 
Importantly, the $(C_d,\delta)$-doubling assumption in Theorem~\ref{piate-intro} is minimal in that the conclusion of Theorem~\ref{piate-intro} may fail to hold true in the absence of this property.
Indeed, Example~\ref{infinitecomb} provides an example of a bounded, but not totally bounded  metric-measure space (in fact every ball is not totally bounded) where the measure is not $(C_d,\delta)$-doubling, but the embedding in \eqref{embedintro} is compact for all $\alpha,p \in (0,\infty)$. The reader is referred to Subsection~\ref{subsect:doubling measures} for specific definitions, examples, and properties of $(c,\delta)$-doubling measures. See also \cite{BK,Hyt,K} wherein related doubling-type conditions have appeared.

The range of $\alpha$ in  Theorem~\ref{piate-intro} is the largest range of this type to be identified and it turns out to be in the nature of best possible in the following sense: In the Euclidean setting, the Lebesgue measure is doubling (as in \eqref{doublingmeasure}) and ${\rm ind}\,(\mathbb{R}^n, |\,\cdot-\cdot\,|)=1$  and so, Theorem~\ref{piate-intro} is valid for ${M}^{\alpha,p}$ with $\alpha\in(0,1]$ and $p\in(0,\infty)$. Theorem~\ref{piate-intro} also holds for the same range of $\alpha$ and $p$ whenever the underlying space is a domain $\Omega\subseteq\mathbb{R}^n$  or, more generally, a metric measure space. If $(X,d,\mu)$ is an ultrametric-measure space (such as a Cantor-type set) then $C_d=1$ and so, ${\rm ind}(X,d)=\infty$ and  $\mu$ is trivially $(C_d,\delta)$-doubling for all $\delta\in(0,\infty)$. Therefore, for ultrametric-measure spaces, the equivalences between the statements in Theorem~\ref{piate-intro} hold with $\alpha,p\in(0,\infty)$ without any assumptions on the quasi-metric-measure space. To provide yet another example, for the metric space $(\mathbb{R}^n,|\cdot-\cdot|^{s})$, $s\in(0,1)$ one has that ${\rm ind}\,(\mathbb{R}^n,|\cdot-\cdot|^{s})=1/s>1$ (see Subsection~\ref{subsect:QMS}) and hence, Theorem~\ref{piate-intro} is
valid with $\alpha\in(0,1/s]$ and $p\in(0,\infty)$. Thus, in certain metric spaces, Theorem~\ref{piate-intro} is valid for values of $\alpha>1$. In view of these examples, one can see from the lower smoothness index, ${\rm ind}(X,d)$, and the $(C_d,\delta)$-doubling properties in Theorem~\ref{piate-intro},  how the geometry of $(X,d,\mu)$ intervenes in the compactness of the embeddings in \eqref{embedintro}. 

Our strategy for proving Theorem~\ref{piate-intro} (or, more specifically, Theorem~\ref{czwarte}) is to construct a suitable family of maximally smooth H\"older-continuous ``bump" functions in the spirit of the classical Urysohn lemma (see Lemma~\ref{GVa2}). In the metric setting, it is always possible to construct such bump functions that are Lipschitz (i.e., H\"older-continuous of order 1); however, in a general quasi-metric space, there is no guarantee that nonconstant Lipschitz functions exist: For example, every Lipschitz function in the quasi-metric space $(\mathbb{R}^n,|\cdot-\cdot|^s)$, where $s>1$ and $|\cdot-\cdot|$ denotes the standard Euclidean distance, is constant. In fact, in this context the Haj\l{}asz--Sobolev space $M^{1,p}$ is trivial (only contains constant functions) for large values of $p$ (see \cite{AYY21}).  As such, we need construct our ``bump" functions within the more general class of H\"older continuous functions.  Interestingly, the maximal amount of the smoothness (measured on the H\"older scale) that such bump functions can possess is intimately linked to the geometry of the underlying space via the lower smoothness index of $(X,d)$, and it is the maximality of this smoothness parameter that allows us to characterize compact embeddings of the spaces ${M}^{\alpha,p}$, ${N}^\alpha_{p,q}$, and  ${M}^\alpha_{p,q}$ for an optimal range of $\alpha$.

%\vskip1in

For compact embeddings into $L^{\tilde{p}}(X)$ with $\tilde{p}<p$, we prove the following result which, to the best of our knowledge, is new even in Euclidean setting. 

\begin{tw} \label{drugie_plus-intro}
Let $(X,d,\mu)$ be a quasi-metric-measure space. If $\mu(X) <\infty$, then for any $\alpha, p\in (0,\infty)$ and  $\tilde{p}\in(0,p)$, the embedding
\begin{align} \label{emb_in_lower-intro}
{M}^{\alpha,p}(X)\hookrightarrow L^{\tilde{p}}(X)
\end{align}
is compact. On the other hand, if the embedding in \eqref{emb_in_lower-intro} is continuous for some $p\in (0,\infty)$, $\tilde{p}\in(0,p)$, and  $\alpha\in (0,\infty)$ with $\alpha\preceq{\rm ind}(X,d)$, then  $\mu(X)<\infty$.
\end{tw}

The reader is directed to Theorem~\ref{drugie_plus} for the corresponding version of Theorem~\ref{drugie_plus-intro} adapted to the spaces ${M}^\alpha_{p,q}(X)$ and ${N}^\alpha_{p,q}(X)$.
Note that in any \textit{metric}-measure space, Theorem~\ref{drugie_plus-intro} is always valid for all $\alpha\in(0,1]$ and $p\in(0,\infty)$.

In contrast to Theorem~\ref{piate-intro}, embeddings into $L^{\tilde{p}}(X)$ with $\tilde{p}\in(0,p)$ are fully characterized in terms of the finiteness of the measure. Our approach to proving Theorem~\ref{drugie_plus-intro} follows along lines similar to the proof of Theorems~\ref{remarkoint-intro} and \ref{piate-intro}.

By considering a strengthening of the assumptions in Theorem~\ref{remarkoint-intro}\footnote{The assumptions in Theorem~\ref{snr-3-intro} imply that the quasi-metric-measure space is totally bounded and hence, the measure $\mu$ is integrable.}, we are able to prove the following Sobolev embedding and Rellich--Kondrachov compactness theorems which give continuous and compact embeddings into $L^{\tilde{p}}(X)$ with $\tilde{p}>p$.

\begin{tw}\label{snr-3-intro}
Let $(X,d,\mu)$ be a quasi-metric-measure space such that $\mu(X)<\infty$ and the measure $\mu$ is locally lower $s$-Ahlfors-regular for some $s\in(0,\infty)$, in the sense that there exists a constant $b\in(0,\infty)$ such that $\mu(B_d(x,r))\geq br^s$ for all $x\in X$ and $r\in(0,1]$. Then the following statements are valid for all $\alpha,p\in(0,\infty)$:
\begin{enumerate}
\item If $\alpha p<s$, then ${M}^{\alpha,p}(X)\hookrightarrow L^{sp/(s-\alpha p)}(X)$, and
for each fixed ${\tilde{p}}\in(0,sp/(s-\alpha p))$, the embedding ${M}^{\alpha,p}(X)\hookrightarrow L^{\tilde{p}}(X)$ is compact.

\item If $\alpha p>s$, then ${M}^{\alpha,p}(X)\hookrightarrow \mathcal{C}^{\alpha-s/p}(X)$, and
for each fixed $\gamma\in(0,\alpha-s/p)$ the embedding ${M}^{\alpha,p}(X)\hookrightarrow \mathcal{C}^{\gamma}(X)$ is compact, where, for any $\beta\in(0,\infty)$, $\mathcal{C}^{\beta}(X)$ denotes the inhomogeneous H\"older space of order $\beta$ (see Subsection~\ref{subsect:fnctspaces}).  In particular, the embedding ${M}^{\alpha,p}(X)\hookrightarrow L^{\tilde{p}}(X)$ is compact for all $\tilde{p}\in(0,\infty)$.
\end{enumerate}
\end{tw}

The assumptions in Theorem~\ref{snr-3-intro} arise naturally in analysis on (quasi-)metric spaces. Indeed, in any  quasi-metric-measure space, if $\mu$ is doubling and $X$ is bounded then $\mu(X)<\infty$ and $\mu$ is locally lower $s$-Ahlfors-regular for any $s\in(0,\infty)$ satisfying \eqref{doubdim}; see Section~\ref{subsect:selfimprove}.

The approach to proving Theorem~\ref{snr-3-intro} is different from the case of embeddings into $L^{\tilde{p}}(X)$ with $\tilde{p}\in(0,p]$ and is based on continuous embeddings obtained in \cite{AYY21}.

It follows from \cite[Theorem~4.13]{AYY21} that if $\alpha,p\in(0,\infty)$ and $\alpha\preceq{\rm ind}(X,d)$
then continuity of the embeddings ${M}^{\alpha,p}(X)\hookrightarrow L^{sp/(s-\alpha p)}(X)$, when $\alpha p<s$, and ${M}^{\alpha,p}(X)\hookrightarrow \mathcal{C}^{\alpha-s/p}(X)$, when $\alpha p>s$ and $(X,d)$ is uniformly perfect, imply the local $s$-lower Ahlfors regular condition in Theorem~\ref{snr-3-intro}. Therefore, for (uniformly perfect) quasi-metric-measure spaces of finite measure, the local $s$-lower Ahlfors regular condition fully characterizes the embeddings in Theorem~\ref{snr-3-intro}.

By \cite[Theorem~1]{hajlaszkt1}, a domain $\Omega\subseteq\mathbb{R}^n$ is a $W^{1,p}$-extension domain for $p\in(1,\infty)$, in the sense that there exists a bounded and linear operator 
$\mathscr{E}: W^{1,p}(\Omega) \to W^{1,p}(\mathbb{R}^n)$
satisfying $\mathscr{E}u|_\Omega=u$ for any $u\in W^{1,p}(\Omega)$, if and only if $W^{1,p}(\Omega)=M^{1,p}(\Omega)$ and there exists a positive constant $C$ such that, for any $x\in\Omega$ and $r\in(0,1]$, $|B(x,r)\cap \Omega|\ge C r^n$,
where $|B(x,r)\cap \Omega|$ denotes the $n$-dimensional Lebesgue measure of $B(x,r)\cap \Omega$
and $B(x,r)$ denotes the open ball in $\mathbb{R}^n$ centered at $x$ with the radius $r$. In light of this, if $\Omega\subseteq\mathbb{R}^n$ is a $W^{1,p}$-extension domain with $p\in(1,\infty)$  and $|\Omega|<\infty$ then Theorem~\ref{snr-3-intro} implies the classical Rellich--Kondrachov compactness theorem about compactness of the embeddings
$W^{1,p}(\Omega)\hookrightarrow L^{\tilde{p}}(\Omega)$, when ${\tilde{p}}\in(0,np/(n-p))$, and
$W^{1,p}(\Omega)\hookrightarrow \mathcal{C}^{\gamma}(\Omega)$, when $\gamma\in(0,1-n/p)$.

As a corollary of Theorems~\ref{piate-intro}, \ref{drugie_plus-intro}, and \ref{snr-3-intro}, we prove that under the  assumption that the measure is doubling, compactness of embeddings of the form ${M}^{\alpha,p}(X)\hookrightarrow L^{\tilde{p}}(X)$, $\alpha,p\in(0,\infty)$  and $\tilde{p}\in(0,p]$ with $\alpha\preceq{\rm ind}(X,d)$, self-improve in the sense that all of the embeddings given by the Rellich--Kondrachov compactness theorem (Theorem~\ref{snr-3-intro}) hold. More specifically, we prove the following; see Theorem~\ref{improve} for the corresponding result stated for the spaces ${M}^\alpha_{p,q}(X)$ and ${N}^\alpha_{p,q}(X)$.

\begin{tw}\label{improve-intro}
Let $(X,d,\mu)$ be a quasi-metric-measure space where the measure $\mu$ is doubling as in \eqref{doublingmeasure}, and suppose that the embedding ${M}^{\alpha,p}(X,d,\mu)\hookrightarrow L^{\tilde{p}}(X,\mu)$ is compact for some $\alpha,p\in(0,\infty)$ and $\tilde{p}\in(0,p]$ with $\alpha\preceq{\rm ind}(X,d)$. Then, the following statements are valid for all $\beta,u,s\in(0,\infty)$ where $s$ is the dimensional exponent given by the doubling measure $($see \eqref{doubdim}$)$:
\begin{enumerate}
\item If $\beta u<s$, then ${M}^{\beta,u}(X)\hookrightarrow L^{su/(s-\beta u)}(X)$, and
for each fixed $\tilde{q}\in(0,su/(s-\beta u))$, the embedding ${M}^{\beta,u}(X)\hookrightarrow L^{\tilde{q}}(X)$ is compact.

\item If $\beta u>s$, then ${M}^{\beta,u}(X)\hookrightarrow \mathcal{C}^{\beta-s/u}(X)$, and
for each fixed $\gamma\in(0,\beta-s/u)$ the embedding ${M}^{\beta,u}(X)\hookrightarrow \mathcal{C}^{\gamma}(X)$ is compact.  Moreover, the embedding ${M}^{\beta,u}(X)\hookrightarrow L^{\tilde{q}}(X)$ is compact for all $\tilde{q}\in(0,\infty)$.
\end{enumerate}
\end{tw}

Note that for \textit{metric}-measure spaces equipped with a doubling measure, the conclusions of Theorem~\ref{improve-intro} are valid whenever the embedding ${M}^{\alpha,p}(X,d,\mu)\hookrightarrow L^{\tilde{p}}(X,\mu)$ is compact with $\alpha\in(0,1]$, $p\in(0,\infty)$, and $\tilde{p}\in(0,p]$.

Our final main results concern compact embeddings between the spaces ${N}^\alpha_{p,q}(X)$ and ${M}^\alpha_{p,q}(X)$ (and hence, also between ${M}^{\alpha,p}(X)$ spaces) with different choices of exponents, $\alpha$, $p$, and $q$; see Theorems~\ref{Triebel_comp} and \ref{Bes_comp} as well as  Corollaries~\ref{cor1} and \ref{beta}. As with our other main results, we do not require that the underlying quasi-metric space is totally bounded nor do we assume that the measures are doubling.

The remainder of this paper is organized as follows. In Section~\ref{sect:prelim}, we review some basic terminology and results pertaining to quasi-metric spaces and the main classes of function spaces considered in this article,
including the fractional Haj\l{}asz--Sobolev spaces ${M}^{\alpha,p}$, the Haj\l{}asz--Triebel--Lizorkin spaces ${M}^\alpha_{p,q}$, and the Haj\l{}asz--Besov spaces  ${N}^\alpha_{p,q}$. We also establish a Urysohn-type lemma (Lemma~\ref{GVa2}) that guarantees the existence of maximally smooth H\"older-continuous ``bump" functions.

The main aim of Section~\ref{sect:meas} is to introduce and study a number of topological and measure theoretic notions that are utilized in the statements of our main results, including the $(c,\delta)$-doubling and integrability conditions on the measure. Section~\ref{sect:lusin} is devoted to recording a version of Lusin's theorem (Theorem~\ref{lusin}) under minimal assumptions on the ambient topological measure space. Characterizations of totally bounded sets in $L^p$ with $p\geq0$ are recorded in Section~\ref{sect:totalbdd}. In particular, the aforementioned Lusin's theorem is used to establish necessary conditions for totally boundedness of sets in space of measurable functions, $L^0$, defined on separable quasi-metric space of finite measure, which generalizes \cite[Theorem~1.3]{BG} and \cite[Theorem~2]{Krotov}.

In Section~\ref{sect:cptembedd} we employ the tools developed in Sections~\ref{sect:prelim}-\ref{sect:totalbdd} to prove the main results in this paper concerning necessary and sufficient conditions under which the embeddings ${M}^{\alpha,p}(X)\hookrightarrow L^{\tilde{p}}(X)$, ${N}^\alpha_{p,q}(X)\hookrightarrow L^{\tilde{p}}(X)$, and ${M}^\alpha_{p,q}(X)\hookrightarrow L^{\tilde{p}}(X)$ are compact. We consider separately the cases $\tilde{p}=p$, $\tilde{p}<p$, and $\tilde{p}>p$ in Subsections~\ref{subsect:p}, \ref{subsect:p-less}, and \ref{subsect:p-greater}, respectively. The main result in Subsection~\ref{subsect:p-greater} is Theorem~\ref{snr-3}, which extends the classical Rellich--Kondrachov compactness theorem to the setting of quasi-metric-measure spaces for the class of fractional Haj\l{}asz--Sobolev, Haj\l{}asz--Triebel--Lizorkin, and Haj\l{}asz--Besov spaces. In Subsection~\ref{subsect:selfimprove} we state and prove a self-improvement property for compact embeddings into $L^{\tilde{p}}(X)$ with $\tilde{p}\leq p$ under the assumption that the measure is doubling; see Theorem~\ref{improve}.
Finally, we study compact embeddings between the spaces ${N}^\alpha_{p,q}(X)$ and ${M}^\alpha_{p,q}(X)$ (and hence, also between ${M}^{\alpha,p}(X)$ spaces) with different choices of exponents in Section~\ref{sect:betweenbesov}.

\section{Preliminaries}
\label{sect:prelim}

Let $\mathbb{N}$ denote all (strictly) positive integers and $\mathbb{N}_0:=\mathbb{N}\cup\{0\}$. Also, let $\mathbb{R}_{+}$ be the set of all (strictly) positive real numbers and set $\overline{\mathbb{R}}:=\mathbb{R}\cup\{\pm\infty\}$. By $C$ we denote a generic constant whose actual value may change from line to line. Given a set $A$, we let $\#A$ stand for its cardinality.

\subsection{The class of quasi-metric spaces}
\label{subsect:QMS}
A pair $(X,d)$ shall be called a {\it quasi-metric} {\it space} if
$X$ is a nonempty set (of cardinality $\geq2$) and $d$ is a quasi-metric on $X$, i.e., 
$d: X\times X\to[0,\infty)$ is a function for which there exist
constants $C_0,C_1\in(0,\infty)$ such  that, for any $x$, $y$, $z\in X$, 
\begin{enumerate}
\item $d(x,y)=0$ if and only if $x=y$;
\item $d(y,x)\leq C_0\,d(x,y)$;
\item $d(x,y)\leq C_1\max\{d(x,z),d(z,y)\}$.
\end{enumerate}
In this context, we let
\begin{eqnarray}\label{C-d}
C_d:=\sup_{\substack{x,y,z\in X\\\mbox{\scriptsize{{not all equal}}}}}
\frac{d(x,y)}{\max\{d(x,z),d(z,y)\}}\in[1,\infty).
\end{eqnarray}
and
\begin{eqnarray}\label{C-d-tilde}
\widetilde{C}_d:=\sup_{x,y\in X,\ x\not=y}
\frac{d(y,x)}{d(x,y)}\in[1,\infty).
\end{eqnarray}
When $C_d=\widetilde{C}_d=1$, $d$ is a genuine metric that is typically referred to as an \textit{ultrametric}.
Note that if the underlying quasi-metric space is $\mathbb{R}^n$, $n\in\mathbb{N}$, equipped with the Euclidean distance $d$, then $\widetilde{C}_d=1$ and $C_d=2$.

Two quasi-metrics $d$ and $\rho$ on $X$ are said to be {\it equivalent},
denoted by $d\approx\rho$, if there exists a positive constant $\kappa$ such that for all $x,y\in X$,
$$
\kappa^{-1}\rho(x,y)\leq d(x,y)\leq \kappa\rho(x,y).
$$
The \textit{lower smoothness index} of a quasi-metric space $(X,d)$ is defined as
\begin{equation}
\label{index}
{\rm ind}(X,d):=\sup_{\rho\approx d}\left(\log_2C_\rho\right)^{-1}\in(0,\infty],
\end{equation}
where the supremum is taken over all quasi-metrics $\rho$ on $X$ which are
equivalent to $d$. Here, we adopt  the convention $1/0:=\infty$.
The lower smoothness index was originally introduced in \cite[Definition 4.26]{MMMM13}
and its value is reflects certain geometrical properties of
the ambient quasi-metric space, as evidenced by the  following examples:
\begin{itemize}%[itemsep=1pt]
\item {${\rm ind}(X,d)\geq 1$ whenever there is a genuine metric on $X$ which is equivalent to $d$;}

\item {${\rm ind}\,(\mathbb{R}^n,|\cdot-\cdot|)=1$ and ${\rm ind}\,([0,1]^n,|\cdot-\cdot|)=1$, where $|\cdot-\cdot|$ denotes the Euclidean distance; in fact, ${\rm ind}\,(Y,\|\cdot-\cdot\|)=1$ whenever $Y$ is a subset of a normed vector space $(X,\|\cdot\|)$ containing an open line segment;}

\item {${\rm ind}(X,d)\leq 1$ whenever the interval $[0,1]$ can be
bi-Lipschitzly embedded into $(X,d)$;}

\item {$(X,d)$ cannot be bi-Lipschitzly embedded into some ${\mathbb{R}}^n$ with
$n\in{\mathbb{N}}$, whenever ${\rm ind}(X,d)<1$;}

\item {${\rm ind}(X,d)=\infty$ whenever there is an ultrametric 
on $X$ which is equivalent to $d$; in particular, ${\rm ind}(X,d)=\infty$ whenever the set $X$ has finite cardinality.}

\item  {${\rm ind}(X,d)=1$ if $(X,d)$ is a metric space that is equipped with a doubling measure  and supports a weak $(1, p)$-Poincar\'e inequality with $p>1$;}

\item ${\rm ind}(X,d^s)=\frac{1}{s}\,{\rm ind}(X,d)$ for all $s\in(0,\infty)$;

\item ${\rm ind}(Y,d)\geq{\rm ind}(X,d)$ for any subset $Y$ of $X$.
\end{itemize}
See \cite[Remark~4.6]{AYY21} and \cite[Section 4.7]{MMMM13} or \cite[Section 2.5]{AM15} for more details.

Given a nonempty set $E\subseteq X$, we let ${\rm diam}_d(E):=\sup\{d(x,y): x,y\in E\}$ and if $F\subseteq X$ is another nonempty set then we define ${\rm dist}_d(E,F):=\inf\{d(x,y): x\in E,\,\,y\in F\}$. When $E=\{x\}$ for some $x\in X$,  we will abbreviate  ${\rm dist}_d(x,F):={\rm dist}_d(\{x\},F)$.

Balls (with respect to $d$) shall be denoted by $B_d(x,r):=\{y\in X: d(x,y)<r\}$, where $x\in X$ and $r\in(0,\infty)$. %We will sometimes simply write $B(x,r)$ in place of $B_d(x,r)$ if the choice of quasi-metric is clear from context. 
Since $d$ is not necessarily symmetric, one needs to pay particular attention to the order of $x$ and $y$ in the definition of $B_d(x,r)$. While the center and the radius of a ball in a quasi-metric space is not necessarily uniquely determined, our balls will always have specified centers and radii so formally a ball is a triplet: a set, a center and a radius. 

The quasi-metric $d$ naturally induces a topology on $X$, denoted by $\tau_d$, where a set $\mathcal{O}\subseteq X$ is open in $\tau_d$ provided, for any $x\in\mathcal{O}$, there exists $r\in(0,\infty)$ such that quasi-metric ball $B_d(x,r)\subseteq\mathcal{O}$. Note that if $\rho$ is another quasi-metric on $X$ that is equivalent to $d$, then there exists a positive constant $\kappa$ such that, for any $x\in X$ and $r\in(0,\infty)$,
$$
B_{\rho}\big(x,\kappa^{-1}r\big)\subseteq B_d(x,r)\subseteq B_{\rho}(x,\kappa r),
$$
from which it follows that $\tau_{\rho}=\tau_d$.

In contrast to what is true in the setting of metric spaces, 
balls with respect to a quasi-metric may not be open sets in the topology induced by the quasi-metric. Nevertheless, it is well-known that the topology induced by a quasi-metric is metrizable. The following result is a sharp quantitative version of this fact  which was established in \cite[Theorem 3.46]{MMMM13}, improving an earlier result with similar aims from \cite{MaSe79}.

\begin{lem}
\label{DST1}
Let $(X,d)$ be a quasi-metric space
and assume that  $C_d, \widetilde{C}_d\in[1,\infty)$ are as in
\eqref{C-d} and \eqref{C-d-tilde}, respectively.
Then there exists a symmetric quasi-metric
$\rho$ on $X$ that satisfies $C_\rho\leq C_d$,
$$
(C_d)^{-2}d(x,y)\leq\rho(x,y)\leq\widetilde{C}_d\,d(x,y)\quad\mbox{for all }\, x,\, y\in X,
$$
and which has the following property: for each fixed finite number $\beta\in(0,(\log_2C_d)^{-1}]$,
\begin{equation*}
[\rho(x,y)]^\beta\leq [\rho(x,z)]^\beta+[\rho(z,y)]^\beta
\quad\mbox{for all }\,x,y,z\in X.
\end{equation*}
Moreover, the function $\rho:(X,\tau_d)\times (X,\tau_d)\longrightarrow [0,\infty)$
is continuous and hence, all $\rho$-balls are open in the topology $\tau_\rho=\tau_d$.
\end{lem}
Although balls in quasi-metric spaces need not be open, it follows from Lemma~\ref{DST1} that for all $x\in X$ and $r\in(0,\infty)$ the ball $B_\rho(x,C_d^2r)$ is an open set satisfying $B_d(x,r)\subseteq B_\rho(x,C_d^2r)\subseteq B_d(x,C_d^2\widetilde{C}_dr)$. We slightly improve this observation in the following proposition.

\begin{prop}\label{otwarty}
Let $(X, d)$ be a quasi-metric space. Then for every $x\in X$ and $r\in(0,\infty)$, there exists a set $E(x,r)\subseteq X$ which is open in $\tau_d$ and satisfies
\begin{align*}
B_d(x,r) \subseteq E(x,r) \subseteq B_d(x,C_{d}r).
\end{align*}
\end{prop}
\begin{proof}
Fix $x\in X$ and $r\in(0,\infty)$. We define $E_0 := B_d(x,r)$ and for $k\geq 0$
\begin{align*}
E_{k+1} := \bigcup_{y\in E_{k}} B_d(y,r_{k}),
\end{align*}
where $r_{k} := r/C_{d}^{k+1}$.
Let $E(x,r) := \bigcup_{k\geq 0} E_k$. Clearly $E(x,r)\subseteq X$ is an open set and $B_d(x,r) \subseteq E(x,r)$. 

We will now prove that for any $k \geq 0$ we have $E_k \subseteq B_d(x,C_{d}r)$. Since $C_d\geq1$ we immediately have that $E_0\subseteq B_d(x,C_{d}r)$. Assume that $k\geq 1$ and let $x_{k}\in E_{k}$. By definition of $E_{k}$ we can find $x_{k-1} \in E_{k-1}$ such that $d(x_{k-1},x_{k}) < r_{k-1}$.  We repeat the argument and obtain a sequence of points $\{x_j\}_{j=0}^{k-1}$ such that for any $j\in \{0,1,\dots,k-1 \}$ we have $x_j \in E_j$ and $d(x_{j},x_{j+1}) < r_{j}$. 

We claim that for any $i \in\{0,1,\dots,k-1 \}$,
\begin{align} \label{k-i}
d(x_{k-1-i},x_{k}) < C_{d} r_{k-1-i}.
\end{align}
For $i=0$, the above inequality is clear since $C_d\geq1$.  For $0<i\leq k-1$, by an induction argument we obtain
\begin{align*}
d(x_{k-1-i},x_{k})\leq C_{d} \max\{d(x_{k-1-i},x_{k-i}), d(x_{k - i},x_{k})  \} < C_{d} \max\{ r_{k-i-1},C_d r_{k-i}\} = C_{d} r_{k-1-i}.
\end{align*}
This finishes the proof of \eqref{k-i}.

Now, we apply \eqref{k-i} with $i = k-1$ and we get
\begin{align*}
d(x,x_{k}) \leq C_{d} \max\{d(x,x_{0}) , d(x_{0},x_{k}) \} < C_{d} r.
\end{align*}
Therefore, we proved for any $k\geq 0$ that $E_{k}\subseteq B_d(x,C_d r)$, which implies that $E(x,r) \subseteq B_d(x,C_{d}r)$. This finishes the proof of the proposition.
\end{proof}

Given a quasi-metric space $(X,d)$ and a set $E\subseteq X$, we denote by $\overline{E}$ the closure of $E$ in $\tau_d$. Note that for balls in $(X,d)$, the topological closure of $B_d(x,r)$ may not be equal to the closed ball $\overline{B}_d(x,r):=\{y\in X: d(x,y)\leq r\}$. We will estimate the closure of balls in Proposition~\ref{zero} and in its proof we will make use of the following result.
\begin{prop}\label{closure-prop}
Let $(X,d)$ be a quasi-metric space and $E\subseteq X$. Then $x\in \overline{E}$ if and only if for all $r\in(0,\infty)$, $B_d(x,r) \cap E \neq \emptyset$.
\end{prop}
\begin{proof}
First, let $x\in \overline{E}$ and suppose, for sake of contradiction, that for some $r\in(0,\infty)$ we have 
\begin{align}\label{EcapE}
B_d(x,C_d  r) \cap E = \emptyset.
\end{align}
By Proposition~\ref{otwarty}, there is a set $E(x,r)\subseteq X$ which is open in $\tau_d$ and satisfies
\begin{align*}
B_d(x,r) \subseteq E(x, r) \subseteq B_d(x,C_dr).
\end{align*}
From this and \eqref{EcapE} we get  $E\subseteq X\setminus E(x,r).$ Since $X\setminus E(x,r)$ is closed in $\tau_d$, the set $F := (X\setminus E(x,r)) \cap \overline{E}$ is also closed in $\tau_d$. On the one hand, given that $E\subseteq F \subseteq \overline{E}$, the definition of the closure implies $\overline{E}=F$. On the other  $x\not\in F$ which contradicts assumption $x\in \overline{E}$.

Now, let us suppose that $x\not\in \overline{E}$, i.e., $x\in X\setminus \overline{E}$.
Since $X\setminus \overline{E}$ is open in $\tau_d$, we can find $r\in(0,\infty)$ such that
$B_d(x,r) \subseteq X\setminus \overline{E},$
and hence
\begin{align*}
B_d(x,r)\cap E \subseteq (X\setminus \overline{E}) \cap E = \emptyset.
\end{align*}
Therefore, $B_d(x,r)\cap E =\emptyset$ and the claim follows. This finishes the proof of the proposition.
\end{proof}

Given a quasi-metric space $(X,d)$, a nonempty set $E\subseteq X$, and a threshold $\delta\in(0,\infty)$ we define $(E)_\delta:=\{x\in X: {\rm dist}_d(x,E)<\delta\}$.

\begin{prop}\label{zero}
Let $(X,d)$ be a quasi-metric space. Then, for any $x\in X$ and any $\delta,r\in(0,\infty)$ satisfying $\widetilde{C}_d\delta\leq r$, one has $\overline{B_d(x,r)} \subseteq (B_d(x,r))_{\delta} \subseteq B_d(x,C_dr)$.
\end{prop}
\begin{proof}
Fix $x\in X$ and  $\delta,r\in(0,\infty)$ satisfying $\widetilde{C}_d\delta\leq r$. By Proposition~\ref{closure-prop}, if $y\in \overline{B_d(x,r)}$, then we have $B_d(y,\delta) \cap B_d(x,r) \neq \emptyset$ and it is easy to see from this that $y\in(B_d(x,r))_{\delta}$. Hence, $\overline{B_d(x,r)} \subseteq (B_d(x,r))_{\delta}$.

Now, let us take $y\in (B_d(x,r))_{\delta}$. Then there is some $z\in B_d(x,r)$ such that $d(y,z) < \delta\leq r/\widetilde{C}_d$. Then,
\begin{align*}
d(x,y) \leq C_{d} \max\{d(x,z), d(z,y) \}
\leq C_{d} \max\{d(x,z), \widetilde{C}_dd(y,z) \} < C_{d} r.
\end{align*}
Hence, $y\in B_d(x,C_dr)$ and the claim follows. This finishes the proof of the proposition.
\end{proof}

\subsection{Fractional Sobolev, Triebel--Lizorkin, and Besov spaces}
\label{subsect:fnctspaces}
In this section we will recall the definitions and some basic facts of function spaces considered in this work.

Let $(X, d, \mu)$ be a quasi-metric space equipped with a nonnegative Borel measure, and fix exponents $\alpha\in(0,\infty)$ and $p,q\in (0,\infty]$. A sequence of nonnegative measurable functions $\overrightarrow{g} := \{g_k\}_{k\in\mathbb{Z}}$, which are defined on $X$, is called an \emph{$\alpha$-fractional gradient} of a measurable function $u: X\to\mathbb{R}$ if there is a measurable set $E \subseteq X$  such that $\mu(E) = 0$ and
\begin{align}
\label{g-frac}
|u(x) - u(y)| \leq [d(x,y)]^{\alpha}\big(g_k(x) + g_k(y)\big),
\end{align}
for all $k\in \mathbb{Z}$ and $x,y\in X\setminus  E$ satisfying $2^{-k-1}\leq d(x,y)<2^{-k}$. The set of all fractional $\alpha$-gradients of $u$ (with respect to $(X, d, \mu)$) is denoted by $\mathbb{D}^{\alpha}_{d, \mu}(u)$. If we deal with the fixed measure $\mu$ on $(X, d)$ we shall write $\mathbb{D}^{\alpha}_{d}(u)$ instead of  $\mathbb{D}^{\alpha}_{d, \mu}(u)$.

Given a sequence  $\overrightarrow{g}:=\{g_k\}_{k\in\mathbb{Z}}$
of measurable functions defined on $X$, we set
\begin{equation*}
\Vert \overrightarrow{g}\Vert_{L^p(X;\ell^{q}(\mathbb{Z}))}:=
\left\Vert\, \Vert \{g_k\}_{k\in\mathbb{Z}}\Vert_{\ell^q} \right\Vert_{L^p(X,\mu)}
\end{equation*}
and
\begin{equation*}
\Vert \overrightarrow{g}\Vert_{\ell^q(\mathbb{Z};L^{p}(X,\mu))}:=\left\Vert \left\{\Vert g_k\Vert_{L^p(X,\mu)}\right\}_{k\in\mathbb{Z}} \right\Vert_{\ell^q},
\end{equation*}
where
\begin{equation*}
\Vert \{g_k\}_{k\in\mathbb{Z}}\Vert_{\ell^q}:=
\begin{cases}
 \displaystyle\left(\sum_{k\in\mathbb{Z}}\vert g_k\vert^q\right)^{1/q}& ~\text{if}~q\in(0,\infty),\\
 \displaystyle \sup_{k\in\mathbb{Z}}\vert g_k\vert& ~\text{if}~q=\infty.
\end{cases}
\end{equation*}

A measurable function $u: X\to\mathbb{R}$ is said to belong to the \emph{homogeneous Haj{\l}asz--Triebel--Lizorkin space} $\dot{M}^{\alpha}_{p,q}(X,d,\mu)$ if the following semi-norm
\begin{align*}
\|u\|_{\dot{M}^{\alpha}_{p,q}(X,d,\mu)} := \inf_{\overrightarrow{g} \in \mathbb{D}^{\alpha}_d(u)} \|\overrightarrow{g}\|_{L^p(X;\ell^{q}(\mathbb{Z}))}
\end{align*}
is finite. The \emph{inhomogeneous Haj{\l}asz--Triebel--Lizorkin space} $M^{\alpha}_{p,q}(X,d,\mu)$ consists of all measurable functions $u \in L^p(X,\mu) \cap\dot{M}^{\alpha}_{p,q}(X,d,\mu)$ endowed with the (quasi-)norm
\begin{align*}
\|u\|_{M^{\alpha}_{p,q}(X,d,\mu)} := \|u\|_{L^p(X,\mu)} + \|u\|_{\dot{M}^{\alpha}_{p,q}(X,d,\mu)}.
\end{align*}

A measurable function $u: X\to\mathbb{R}$ is said to belong to the \emph{homogeneous Haj{\l}asz--Besov space} $\dot{N}^{\alpha}_{p,q}(X,d,\mu)$ if the following semi-norm
\begin{align*}
\|u\|_{\dot{N}^{\alpha}_{p,q}(X,d,\mu)} := \inf_{\overrightarrow{g} \in \mathbb{D}^{\alpha}_d(u)} \|\overrightarrow{g}\|_{\ell^q(\mathbb{Z};L^{p}(X,\mu))}
\end{align*}
is finite. The \emph{inhomogeneous Haj{\l}asz--Besov space} $N^{\alpha}_{p,q}(X,d,\mu)$ consists of all measurable functions $u \in L^p(X,\mu) \cap\dot{N}^{\alpha}_{p,q}(X,d,\mu)$ endowed with the (quasi-)norm
\begin{align*}
\|u\|_{N^{\alpha}_{p,q}(X,d,\mu)} := \|u\|_{L^p(X,\mu)} + \|u\|_{\dot{N}^{\alpha}_{p,q}(X,d,\mu)}.
\end{align*}

The above Haj{\l}asz--Triebel--Lizorkin and Haj{\l}asz--Besov spaces were introduced in \cite{KYZ11}.
It is easy to see that when  $p\in[1,\infty)$ and $q\in[1,\infty]$, the spaces
$M^s_{p,q}(X,d,\mu)$ and $N^s_{p,q}(X,d,\mu)$ are Banach. 
Otherwise, they are quasi-Banach spaces. It was shown in \cite{KYZ11} that $M^s_{p,q}(\mathbb{R}^n)$ coincides with the classical Triebel--Lizorkin space $F^s_{p,q}(\mathbb{R}^n)$ for any $s\in(0,1)$, $p\in(\frac{n}{n+s},\infty)$, and $q\in(\frac{n}{n+s},\infty]$,  and $N^s_{p,q}(\mathbb{R}^n)$ coincides with the classical Besov space $B^s_{p,q}(\mathbb{R}^n)$ for any $s\in(0,1)$, $p\in(\frac{n}{n+s},\infty)$, and $q\in(0,\infty]$.
We also refer the reader to \cite{AWYY21,AYY22,AYY21,GKZ13,Karak1,Karak2} for more information on Triebel--Lizorkin and Besov spaces in general geometric settings.

The following result is from \cite[Proposition~2.1]{AYY22}.

\begin{lem}
\label{equivspaces}
Let $(X,d,\mu)$ be a quasi-metric space equipped with a nonnegative Borel measure, and let
$\alpha,p\in(0,\infty)$ and $q\in(0,\infty]$.
Suppose that $\rho$ is a quasi-metric on $X$ such that $\rho\approx d$ and for which all $\rho$-balls are measurable.
Then there exist constants $c_1, c_2\in(0,\infty)$ such that
$$
c_1\|u\|_{\dot{M}^{\alpha}_{p,q}(X,\rho,\mu)}\leq\|u\|_{\dot{M}^{\alpha}_{p,q}(X,d,\mu)}
\leq c_2\|u\|_{\dot{M}^{\alpha}_{p,q}(X,\rho,\mu)}\quad\mbox{for all } u\in\dot{M}^{\alpha}_{p,q}(X,d,\mu),
$$
and
$$
c_1\|u\|_{\dot{N}^{\alpha}_{p,q}(X,\rho,\mu)}\leq\|u\|_{\dot{N}^{\alpha}_{p,q}(X,d,\mu)}
\leq c_2\|u\|_{\dot{N}^{\alpha}_{p,q}(X,\rho,\mu)}
\quad\mbox{for all } u\in\dot{N}^{\alpha}_{p,q}(X,d,\mu).
$$
\end{lem}

Now, we recall the definition of the (fractional) Haj\l{}asz--Sobolev space. Let $(X, d, \mu)$ be a quasi-metric space equipped with a nonnegative Borel measure, and fix exponents $\alpha\in(0,\infty)$ and $p\in (0,\infty]$.
A nonnegative measurable function  $g$ defined on $X$  is called
an \textit{$\alpha$-gradient} of a measurable function $u: X\to\mathbb{R}$ if
there is a measurable set $E\subseteq X$ with $\mu(E)=0$ such that
\begin{equation}
\label{g}
\vert u(x)-u(y)\vert\leq [d(x,y)]^\alpha\big(g(x)+g(y)\big)\quad\mbox{for all}\,\,x,y\in X\setminus E.
\end{equation}
The collection of all the $\alpha$-gradients of $u$ (with respect to $(X, d, \mu)$) is denoted by $\mathcal{D}_d^\alpha(u)$.

A measurable function $u: X\to\mathbb{R}$ is said to belong to the \emph{homogeneous  Haj{\l}asz--Sobolev space} $\dot{M}^{\alpha,p}(X,d,\mu)$ if the following semi-norm
\begin{align*}
\|u\|_{\dot{M}^{\alpha,p}(X,d,\mu)} := \inf_{g \in \mathcal{D}_d^\alpha(u)} \|g\|_{L^p(X,\mu)}
\end{align*}
is finite. The \emph{inhomogeneous Haj{\l}asz--Sobolev space} $M^{\alpha,p}(X,d,\mu)$ consists of all measurable functions $u \in L^p(X,\mu) \cap\dot{M}^{\alpha,p}(X,d,\mu)$ endowed with the (quasi-)norm
\begin{align*}
\|u\|_{M^{\alpha,p}(X,d,\mu)} := \|u\|_{L^p(X,\mu)} + \|u\|_{\dot{M}^{\alpha,p}(X,d,\mu)}.
\end{align*}

The $M^{\alpha,p}$ Sobolev spaces were introduced by Haj\l{}asz in \cite{Hajlasz} for $\alpha=1$ and by Yang in \cite{Y03} for $\alpha\neq1$. It was shown in \cite{Hajlasz2} that $M^{1,p}(\Omega)$ coincides with the classical Sobolev space $W^{1,p}(\Omega)$ provided $p>1$ and either $\Omega=\mathbb{R}^n$ or $\Omega\subseteq\mathbb{R}^n$ is a domain with a sufficiently regular boundary.

The following proposition highlights a relationship among these function spaces depending on the values of their exponents.

\begin{prop} \label{embs}
Let $(X, d, \mu)$ be a quasi-metric space equipped with a nonnegative Borel measure. Then, for any $\alpha, \sigma, p, q, r\in(0,\infty)$, one has
\begin{enumerate}[i)]
\item $\dot{N}^{\alpha}_{p,q}(X,d,\mu)\hookrightarrow \dot{N}^{\alpha}_{p,\infty}(X,d,\mu)$,
\item $\dot{M}^{\alpha}_{p,q}(X,d,\mu)\hookrightarrow \dot{M}^{\alpha}_{p,\infty}(X,d,\mu)$,
\item $\dot{M}^{\alpha}_{p,\infty}(X,d,\mu) = \dot{M}^{\alpha,p}(X,d,\mu)$, as sets, with equal semi-norms,
\item $\dot{M}^{\alpha}_{p,p}(X,d,\mu) =\dot{N}^{\alpha}_{p,p}(X,d,\mu)$, as sets, with equal semi-norms,
\item $N^{\alpha+\sigma}_{p,\infty}(X,d,\mu)\hookrightarrow N^{\alpha}_{p,r}(X,d,\mu)$,
\item $M^{\alpha+\sigma}_{p,\infty}(X,d,\mu)\hookrightarrow M^{\alpha}_{p,r}(X,d,\mu)$.
\end{enumerate}
\end{prop}
\begin{proof}
Fix  $\alpha, \sigma, p, q, r\in(0,\infty)$. As a preamble, observe that for any sequence $\{a_k\}_{k\in\mathbb{Z}}$ of numbers, we have
\begin{align} \label{linf<lq}
\|\{a_k\}_{k\in\mathbb{Z}}\|_{\ell^{\infty}(\mathbb{Z})} \leq \|\{a_k\}_{k\in\mathbb{Z}}\|_{\ell^{q}(\mathbb{Z})}.
\end{align}

To prove $i)$, fix $u\in \dot{N}^{\alpha}_{p,q}(X,d,\mu)$ and let $\overrightarrow{g} := \{g_k\}_{k\in\mathbb{Z}} \in \mathbb{D}^{\alpha}_d(u)$ be such that $\|\overrightarrow{g}\|_{\ell^{q}(\mathbb{Z};L^p(X,\mu))}<\infty$. Then taking $a_k := \|g_{k}\|_{L^p(X,\mu)}$ in \eqref{linf<lq} for all $k\in\mathbb{Z}$, we get
\begin{align*}
\|u\|_{\dot{N}^{\alpha}_{p,\infty}(X,d,\mu)} \leq \|\overrightarrow{g}\|_{\ell^{\infty}(\mathbb{Z};L^p(X,\mu))} \leq \|\overrightarrow{g}\|_{\ell^{q}(\mathbb{Z};L^p(X,\mu))}.
\end{align*}
Taking the infimum over all $\overrightarrow{g}\in \mathbb{D}^{\alpha}_d(u)$ yields
\begin{align*}
\|u\|_{\dot{N}^{\alpha}_{p,\infty}(X,d,\mu)} \leq \|u\|_{\dot{N}^{\alpha}_{p,q}(X,d,\mu)}.
\end{align*}

For $ii)$, fix $u\in\dot{M}^{\alpha}_{p,q}(X,d,\mu)$ and let $\overrightarrow{g} := \{g_k\}_{k\in\mathbb{Z}} \in \mathbb{D}^{\alpha}_d(u)$ be such that $\|\overrightarrow{g}\|_{L^p(X;\ell^{q}(\mathbb{Z}))}<\infty$. In this case, there is a measurable set $N\subseteq X$ such that $g_k(x)<\infty$ for all $k\in\mathbb{Z}$ and $x\in X\setminus N$. Therefore, if $k\in\mathbb{Z}$ and $x\in X\setminus N$ then taking $a_k := g_k(x)$ in \eqref{linf<lq}, we obtain
\begin{align*}
\|\{g_k(x)\}_{k\in\mathbb{Z}}\|^p_{\ell^{\infty}(\mathbb{Z})} \leq \|\{g_k(x)\}_{k\in\mathbb{Z}}\|^p_{\ell^{q}(\mathbb{Z})},
\end{align*}
from which it follows that
\begin{align*}
\|u\|_{\dot{M}^{\alpha}_{p,\infty}(X,d,\mu)} \leq \|\overrightarrow{g}\|_{L^p(X;\ell^{\infty}(\mathbb{Z}))}\leq \|\overrightarrow{g}\|_{L^p(X;\ell^{q}(\mathbb{Z}))}.
\end{align*}
Hence, we have
$
\|u\|_{\dot{M}^{\alpha}_{p,\infty}(X,d,\mu)} \leq \|u\|_{\dot{M}^{\alpha}_{p,q}(X,d,\mu)}.
$

Regarding $iii)$, let $u\in\dot{M}^{\alpha}_{p,\infty}(X,d,\mu)$ and consider $\overrightarrow{g} := \{g_k\}_{k\in\mathbb{Z}} \in \mathbb{D}^{\alpha}_d(u)$. %such that $\|\overrightarrow{g}\|_{L^p(X;\ell^{\infty}(\mathbb{Z}))}<\infty$.
We define $g:= \sup_{k\in \mathbb{Z}} g_{k}$. It is clear that $g$ is a nonnegative measurable function that satisfies inequality \eqref{g} and
\begin{align*}
 \|\overrightarrow{g}\|_{L^p(X;\ell^{\infty}(\mathbb{Z}))} = \|g\|_{L^p(X,\mu)} \geq \|u\|_{\dot{M}^{\alpha,p}(X,d,\mu)}.
\end{align*}
Taking the infimum over all $\overrightarrow{g}\in \mathbb{D}^{\alpha}_d(u)$ gives
$
\|u\|_{\dot{M}^{\alpha}_{p,\infty}(X,d,\mu)}\geq \|u\|_{\dot{M}^{\alpha,p}(X,d,\mu)}.
$

On the other hand, given any $u\in M^{\alpha,p}(X,d,\mu)$ and any $g\in \mathcal{D}_d^\alpha(u)$, we define $g_k: = g$ for all $k\in \mathbb{Z}$. It is easy to see that $\overrightarrow{g} := \{g_k\}_{k\in\mathbb{Z}} \in \mathbb{D}^{\alpha}_d(u)$ and
\begin{align*}
 \|g\|_{L^p(X,\mu)} = \|\overrightarrow{g}\|_{L^p(X;\ell^{\infty}(\mathbb{Z}))} \geq  \|u\|_{\dot{M}^{\alpha}_{p,\infty}(X,d,\mu)}.
\end{align*}
Thus, we finally get
$
\|u\|_{\dot{M}^{\alpha,p}(X,d,\mu)}\geq  \|u\|_{\dot{M}^{\alpha}_{p,\infty}(X,d,\mu)}.
$

Note that $iv)$ follows from the observation that for any sequence  $\overrightarrow{g}:=\{g_k\}_{k\in\mathbb{Z}}$
of measurable functions defined on $X$, there holds
\begin{align*}
\|\overrightarrow{g} \|^p_{L^p(X;l^p(\mathbb{Z}))} = \int\limits_{X} \sum_{k\in\mathbb{Z}} |g_k(x)|^p\, d\mu(x) =   \sum_{k\in\mathbb{Z}} \int\limits_{X} |g_k(x)|^p\,d\mu(x) = \|\overrightarrow{g} \|^p_{l^p(\mathbb{Z};L^p(X,\mu))}.
\end{align*}

Turning to $v)$, fix  $u\in N^{\alpha+\sigma}_{p,\infty}(X,d,\mu)$ and $\overrightarrow{g}\in \mathbb{D}^{\alpha+\sigma}_d(u)$. We define $\overrightarrow{h}:=\{h_k\}_{k\in\mathbb{Z}}$ by setting for all $k\in\mathbb{Z}$ and all $x\in X$,
\[h_{k}(x) := 
\begin{cases}
2^{-k\sigma} g_{k}(x) \mbox{ if } k\geq 0,\\
2^{k\alpha + \alpha} |u(x)| \mbox{ if } k< 0.
\end{cases}
\]
Then it is straightforward to check that $\overrightarrow{h}\in\mathbb{D}^{\alpha}_d(u)$. Moreover, we can estimate
\begin{align*}
\|\overrightarrow{h}\|^r_{l^r(\mathbb{Z};L^p(X,\mu))} &= \sum_{k\in\mathbb{Z}} \bigg(\int\limits_{X} [h_k(x)]^{p}\, d\mu(x) \bigg)^{r/p} \\
&= \sum_{k\geq 0} 2^{-kr\sigma} \bigg(\int\limits_{X} [g_k(x)]^{p}\, d\mu(x) \bigg)^{r/p} + 2^{r\alpha}\sum_{k<0} 2^{kr\alpha} \bigg(\int\limits_{X} |u(x)|^{p}\, d\mu(x) \bigg)^{r/p} \\
&\leq \dfrac{1}{1-2^{-r\sigma}}\|\overrightarrow{g}\|_{l^{\infty}(\mathbb{Z};L^p(X,\mu))}^{r}    +   \dfrac{1}{1-2^{-r\alpha}}  \|u\|_{L^p(X,\mu)}^{r},
\end{align*}
and the claim now follows.

Finally, as concerns $vi)$, the proof is similar  to the proof of $v)$. Let $u\in M^{\alpha+\sigma}_{p,\infty}(X,d,\mu)$. Since $M^{\alpha+\sigma}_{p,\infty}(X,d,\mu)= M^{\alpha + \sigma,p}(X,d,\mu)$ by part $iii)$ in this proposition, we can find a gradient $g \in \mathcal{D}^{\alpha + \sigma}_d(u)$.
We define  a new gradient $\overrightarrow{h}:=\{h_k\}_{k\in\mathbb{Z}}\in \mathbb{D}^{\alpha}_d(u)$  by setting for all $k\in\mathbb{Z}$ and all $x\in X$,
\[
h_k(x) := 
\begin{cases}
2^{-k\sigma} g(x) \mbox{ if } k\geq 0,\\
2^{k\alpha + \alpha} |u(x)| \mbox{ if } k< 0.
\end{cases}
\]
 Then,
\begin{align*}
\|\overrightarrow{h}\|^p_{L^p(X;l^{r}(\mathbb{Z}))} &= \int\limits_{X} \Big(\sum\limits_{k\in\mathbb{Z}}[h_k(x)]^{r}\Big)^{p/r}\, d\mu(x) = \int\limits_{X} \Big(\sum\limits_{k\geq 0}2^{-kr\sigma}[g(x)]^r + \sum\limits_{k<0} 2^{k\alpha r + \alpha r} |u(x)|^{r}\Big)^{p/r}\, d\mu(x)  \\
&= \int\limits_{X} \Big( \frac{1}{1-2^{-r\sigma}} [g(x)]^r + \frac{1}{1-2^{-r\alpha}}|u(x)|^{r}  \Big)^{p/r}\, d\mu(x) \\
&\leq 2^{p/r} \Big(\frac{1}{1-2^{-r\sigma}}\Big)^{p/r} \|g\|^{p}_{L^p(X,d,\mu)} + 2^{p/r}\Big(\frac{1}{1-2^{-r\alpha}}\Big)^{p/r}\|u\|^{p}_{L^p(X,d,\mu)},
\end{align*}
and the claim now follows. This finishes the proof of the proposition.
\end{proof}

\begin{lem}\label{Fakt}
Let $(X,d)$ be a quasi-metric space, $p,\tilde{p} \in (0,\infty)$ and $\mu$, $\nu$ be a nonnegative Borel measures such that the embedding
\begin{align} \label{LpwLfal}
L^p(X,\mu) \hookrightarrow L^{\tilde{p}}(X,\nu)
\end{align}
is continuous. Then, for any $\alpha \in (0,\infty)$, $q\in (0,\infty]$ one has 
\begin{align*}
\dot{M}^{\alpha}_{p,q}(X,d,\mu)&\hookrightarrow \dot{M}^{\alpha}_{\tilde{p},q}(X,d,\nu),\quad \dot{N}^{\alpha}_{p,q}(X,d,\mu)\hookrightarrow \dot{N}^{\alpha}_{\tilde{p},q}(X,d,\nu), \\
M^{\alpha}_{p,q}(X,d,\mu)&\hookrightarrow M^{\alpha}_{\tilde{p},q}(X,d,\nu),\quad N^{\alpha}_{p,q}(X,d,\mu)\hookrightarrow N^{\alpha}_{\tilde{p},q}(X,d,\nu).
\end{align*}
\end{lem}

\begin{proof}
It is enough consider only the homogeneous spaces. 
By \eqref{LpwLfal} there exists $C\in (0,\infty)$ such that for every $f\in L^p(X,\mu)$ we have
 \begin{align} \label{zholdera}
 \|f\|_{L^{\tilde{p}}(X,\nu)} \leq  C \|f\|_{L^{p}(X,\mu)}.
 \end{align}
Putting $f = \chi_{E}$ in the above inequality, where $E\subseteq X$ is an arbitrary measurable set such $\mu(E) = 0$, we get that $\nu$ is absolutely continuous with respect to $\mu$.
Fix $u\in L^p(X,\mu)$ and let $\overrightarrow{g} \in   \mathbb{D}^{\alpha}_{d, \mu}(u)$ be such that $\|\overrightarrow{g}\|_{L^p(X;\ell^{q}(\mathbb{Z}),\mu)} < \infty$ . Since $\nu<<\mu$, we have also $\overrightarrow{g} \in\mathbb{D}^{\alpha}_{d, \nu}(u)$. By \eqref{zholdera} we get
\begin{align*}
\|u\|_{\dot{M}^{\alpha}_{\tilde{p},q}(X,d,\nu)} \leq \|\overrightarrow{g}\|_{L^{\tilde{p}}(X;\ell^{q}(\mathbb{Z}),\nu)} \leq  C \|\overrightarrow{g}\|_{L^p(X;\ell^{q}(\mathbb{Z}),\mu)}. 
\end{align*}
Taking the infimum over $\overrightarrow{g} \in\mathbb{D}^{\alpha}_{d, \mu}(u)$ yields
\begin{align*}
\|u\|_{\dot{M}^{\alpha}_{\tilde{p},q}(X,d,\nu)} \leq C\|u\|_{\dot{M}^{\alpha}_{p,q}(X,d,\mu)}.
\end{align*}
In the same manner we get
\begin{align*}
\|u\|_{\dot{N}^{\alpha}_{\tilde{p},q}(X,d,\nu)} \leq C\|u\|_{\dot{N}^{\alpha}_{p,q}(X,d,\mu)}.
\end{align*}
\end{proof}

Let $(X,d)$ be a quasi-metric space and fix $\beta\in(0,\infty)$. Recall that the \emph{homogeneous H\"older space of order $\beta$},
$\dot{\mathcal{C}}^\beta(X,d)$ consists of all functions $f: X\to\mathbb{R}$ with the property that the following semi-norm
$$
{\lVert f\rVert}_{\dot{\mathcal{C}}^\beta(X,d)}
:=\sup_{\substack{x,\,y\in X,\, x\ne y}}
\frac{|f(x)-f(y)|}{[d(x,y)]^\beta}
$$
is finite. The \emph{inhomogeneous H\"older space of order $\beta$}, ${\mathcal{C}}^\beta(X,d)$
is defined as
$$
{\mathcal{C}}^\beta(X,d):=\big\{f\in\dot{\mathcal{C}}^\beta(X,d): \mbox{$f$ is a bounded function in $X$}\big\},
$$
and is equipped with the following norm
$$
{\lVert f\rVert}_{{\mathcal{C}}^\beta(X,d)}
:=\sup_{x\in X}|f(x)|+{\lVert f\rVert}_{\dot{\mathcal{C}}^\beta(X,d)}.
$$

\begin{lem}
\label{holdcpt}
Let $(X,d)$ be a quasi-metric space. Then $(X,d)$ is totally bounded if and only if the
embedding 
$$
\mathcal{C}^{\beta}(X,d)\hookrightarrow\mathcal{C}^{\alpha}(X,d)
$$
is well defined and compact for every $0<\alpha<\beta<\infty$.
\end{lem}
\begin{proof}
This result was established in \cite[Theorem~4.16]{MMMM13} under the assumption that the quasi-distance $d$ is symmetric (i.e., when $\widetilde{C}_d= 1$). The version of this result stated here may be proved either by proceeding along similar lines, or by observing that the result for symmetric quasi-distances in \cite{MMMM13} self-improves to the current version as follows. Consider the symmetrization, $d_{\rm sym}$, of $d$ which is defined by setting $d_{\rm sym}(x,y):=\max\{d(x,y),d(y,x)\}$ for all $x,y\in X$. It is easy to verify that $d_{\rm sym}$ is a symmetric quasi-metric on $X$ satisfying $d_{\rm sym}\approx d$. Hence, in particular, we have that $\mathcal{C}^{\alpha}(X,d)=\mathcal{C}^{\alpha}(X,d_{\rm sym})$ for all $\alpha\in(0,\infty)$, and that $(X,d)$ is totally bounded if and only if $(X,d_{\rm sym})$ is totally bounded. Granted these observations, this lemma now follows immediately from the corresponding result in \cite[Theorem~4.16]{MMMM13} for symmetric quasi-metrics.
\end{proof}

The next result is a Urysohn-type lemma which is a slight refinement of \cite[Lemma~4.4]{AYY21}. 

\begin{lem}
\label{GVa2} 
Let $(X, d, \mu)$ be a quasi-metric space equipped with a nonnegative Borel measure $\mu$, and fix parameters $\alpha,p\in(0,\infty)$ and $q\in(0,\infty]$. Let $C_d\in[1,\infty)$ be as in \eqref{C-d} and suppose that $\alpha\leq(\log_{2}C_d)^{-1}$, where $\alpha=(\log_{2}C_d)^{-1}$ is only permitted when $q=\infty$. Also assume that $E_0,E_1\subseteq X$ are two nonempty measurable sets such that ${\rm dist}_d(E_0,E_1)>0$. Then there exists a continuous (hence, measurable) function $\Phi: X\to[0,1]$ such that $\Phi\equiv 1$  on $E_1$ and $\Phi\equiv 0$ on $E_0$. Moreover, there exists  $C\in(0,\infty)$,  depending only on $\alpha$, $q$, and $d$, such that
\begin{equation}
\label{gradest1}
\Vert\Phi\Vert_{\dot{M}^\alpha_{p,q}(X,d,\mu)}\leq C[{\rm dist}_d(E_0,E_1)]^{-\alpha}[\mu(X\setminus E_0)]^{1/p}
\end{equation}
and
\begin{equation}
\label{gradest2}
\Vert\Phi\Vert_{\dot{N}^\alpha_{p,q}(X,d,\mu)}\leq C
[{\rm dist}_d(E_0,E_1)]^{-\alpha}[\mu(X\setminus E_0)]^{1/p}.
\end{equation}
\end{lem}
\begin{proof}
Given our assumptions on $\alpha$, we can choose a finite number $\beta\in[\alpha,(\log_{2}C_d)^{-1}]$, where $\beta\neq\alpha$ unless $\alpha=(\log_{2}C_d)^{-1}$. By \cite[Theorem~4.1]{AMM13} (see, also \cite[Theorem~2.6]{AM15}), there exists a continuous function $\Phi\in\dot{\mathcal{C}}^\beta(X,d)$ such that $0\leq\Phi(x)\leq1$ for all $x\in X$, $\Phi\equiv 1$  on $E_1$, $\Phi\equiv 0$ on $E_0$, and for which there exists $C\in(0,\infty)$ satisfying
$$
{\lVert \Phi\rVert}_{\dot{\mathcal{C}}^\beta(X,d)}\leq C[{\rm dist}_d(E_0,E_1)]^{-\beta}.
$$
The estimates in \eqref{gradest1} and \eqref{gradest2} now follow from applying \cite[Lemma~3.10]{AYY22} with $f:=1$, $\vec{h}:=\vec{0}$, $V:=X\setminus E_0$, and $\Psi:=\Phi$.
\end{proof}

\section{Topological and measure-theoretic considerations}
\label{sect:meas}

We begin with a few definitions.
Given an arbitrary 
set $X$ and a topology $\tau$ on $X$, we denote by
${\it Borel}_{\tau}(X)$, the smallest sigma-algebra of 
$X$ containing $\tau$.  Recall that a measure $\mu:\mathfrak{M}\to[0,\infty]$, where  $\mathfrak{M}$ is a sigma-algebra of subsets of $X$, is said to be
a \emph{Borel measure (on $X$)} provided ${\it Borel}_{\tau}(X)\subseteq\mathfrak{M}$.

The triplet $(X,d,\mu)$ will be called  a \textit{quasi-metric-measure space} if $(X,d)$ is a quasi-metric space and $\mu$ is a nonnegative Borel measure on $X$ such that all balls (with respect to the quasi-metric $d$) are measurable and $\mu(B_d(x,r))\in(0,\infty)$ for any $x\in X$ and any $r\in(0,\infty)$.

\subsection{Separability and total boundedness in quasi-metric-measure spaces}

\begin{lem} {\bf (Ulam-like Theorem)} \label{Ulam}
Let $(X, d, \mu)$ be a separable quasi-metric space equipped with a nonnegative Borel measure and suppose that $\mu(X)<\infty$. %and that every ball with respect to $d$ is measurable. 
Then, for any $\varepsilon\in(0,\infty)$ there exists closed and totally bounded set $G_{\varepsilon}\subseteq X$ such that $\mu(X\setminus G_{\varepsilon}) < \varepsilon$.
\end{lem}
In the setting of Polish metric spaces very similar result has been proven by Ulam.\footnote{The Ulam Theorem says, that each probability measure $\mu$ defined on Borel sets of 
complete metric space $X$ satisfies the following condition: for every  $\varepsilon\in(0,\infty)$ there exists a compact set $G_{\epsilon} \subseteq X$ such that $\mu (X \setminus G_{\epsilon})< \varepsilon$.}
\begin{proof}
Let $\{x_j\}_{j=1}^{\infty}$ be a dense subset in $(X,d)$. For $n,k \in \mathbb{N}$ we define 
\begin{align*} 
A_{n,k} := \bigcup_{j=1}^k \overline{ B_d(x_j, 1/n)}.
\end{align*}

For any given $m \in \mathbb{N}$, since $\{x_j\}_{j=1}^{\infty}$ is dense in $(X,d)$, we have $\bigcup_{k=1}^{\infty}A_{m,k}=X$ and since $A_{m,k}\subseteq A_{m,k+1}$ for every $k\in\mathbb{N}$, we have %Hence, for any $n \in \mathbb{N}$ 
\[
 \lim_{k \rightarrow \infty} \mu(A_{m,k})=\mu(X).
\]
Therefore, we can find $k_m\in\mathbb{N}$ such that
\[
\mu(X\setminus A_{m,k_m}) = \mu(X) -\mu(A_{m,k_m}) \leq 1/2^m.
\]
Now, let $n\in\mathbb{N}$ be arbitrary and set 
\[
 G_n :=\bigcap_{m=n}^{\infty}A_{m,k_m}.
\]
Then, $G_n$ is a closed subset of $(X,d)$ and  
\[
\mu(X\setminus G_n) \leq\sum_{m=n}^\infty\mu(X\setminus A_{m,k_m})\leq 2^{1-n}.
\] 
Moreover, since Proposition~\ref{zero} implies that $A_{n,k} \subseteq \bigcup_{j=1}^k  B_d(x_j, C_d/n)$, for all $n,k \in \mathbb{N}$, we can see that $G_n$ is also totally bounded. Given that $n\in\mathbb{N}$ was arbitrary, this finishes the proof of the lemma.
\end{proof}

\begin{prop}\label{sepequiv}
A quasi-metric space $(X, d)$ is separable if and only if there exists a nonnegative Borel measure on $(X, d)$ such that all balls (with respect to $d$) are measurable and have positive and finite measure.
\end{prop}
In the setting of metric spaces the theorem has been proven in \cite{G}. Let us mention that \cite{DG} contains relation between this theorem and the axiom of choice.
\begin{proof}
Suppose first that $(X, d)$ is separable and let $\{x_j\}_{j=1}^{\infty}$ be a countable and dense subset of $(X, d)$. We define a measure $\mu$ by setting
\[
\mu:=\sum_{i=j}^{\infty}\frac{1}{2^j}\delta_{x_j},
\] 
where $\delta_{x_j}$ is the Dirac measure, i.e., for all sets $A\subseteq X$, $\delta_{x_j}(A)=1$ if $x_j\in A$ and $\delta_{x_j}(A)=0$ otherwise. Clearly, $\mu$ is a nonnegative Borel measure on $X$. Moreover, every ball $B\subseteq X$ with respect to $d$ is measurable and we have
\[
 \mu(B) \leq \mu(X)=1.
\]
On the other hand, since $\{x_j\}_{j=1}^{\infty}$ is dense, there exists $x_k \in B$ and therefore,
\[
\mu(B) \geq \frac{1}{2^k}\delta_{x_k}(B)=\frac{1}{2^k}>0.
\]

Now, let $\mu$ be a nonnegative Borel measure on $(X,d)$ such that all balls (with respect to $d$) are measurable and have positive and finite measure. Let $x_0 \in X$ and for each $k\in \mathbb{N}$, let $B_k:=B_d(x_0,k)$. For $n \in \mathbb{N}$ we 
denote by $A^k_n$ the maximal $1/n$-separated subset of $B_k$. Thus, we have
\[
B_k \subseteq \bigcup_{x \in A^k_n}B_d(x,1/n)
\]
and
\[
B_d\left(x,\frac{1}{n C_d\widetilde{C}_d }\right)\cap B_d\left(y,\frac{1}{n C_d\widetilde{C}_d }\right)=\emptyset\,\, \text{ for all}\,\, x, y \in A^k_n\text{ with}\,\, x\neq y.
\]
For $\ell \in \mathbb{N}$ we define
\[
A^k_n(\ell):=\left\{x \in A^k_n: \mu\left(B_d\left(x,\frac{1}{n C_d\widetilde{C}_d}\right)\right)\geq \frac{1}{\ell}\right\}.
\]
Since
\[
\bigcup_{x \in A^k_n (\ell)}B_d\left(x,\frac{1}{n C_d\widetilde{C}_d}\right) \subseteq B_d(x_0, C_d k)
\]
and $\mu( B_d(x_0, C_d k)) < \infty$, we have $\# A^k_n (\ell) < \infty$. Furthermore, since balls (with respect to $d$) have positive measure, we get

\[
	A^k_n= \bigcup_{l=1}^{\infty}A^k_n(l).
\] 
Finally, we get that the set
\[
C:=\bigcup_{k=1}^{\infty}\bigcup_{n=1}^{\infty}A^k_n 
\]
is countable and dense in $(X,d)$. This finishes the proof of the proposition.
\end{proof}

In a quasi-metric-measure space, totally boundedness of the entire space is equivalent to the measure possessing a ``non-collapsing" property. More specifically, we have the following proposition.
\begin{prop}\label{tot}
Let $(X,d,\mu)$ be a quasi-metric-measure space. Then $(X, d)$ is totally bounded if and only if $\mu(X)< \infty$ and for all $r\in(0,\infty)$,
\[
h(r):=\inf_{x\in X} \mu (B_d(x,r))>0.
\]
\end{prop}
We remark that the conditions $\mu(X)< \infty$ and $\inf_{x\in X} \mu (B_d(x,r))>0$ for all $r\in(0,\infty)$ were assumed in \cite[Theorem~2]{K} to establish a compact embedding theorem for $M^{1,p}$ Sobolev spaces on metric spaces.
\begin{proof}
Since the proof of this result has been given in \cite{GS} in the setting of metric-measure spaces, we only provide a sketch of the argument here. Suppose that $(X, d)$ is totally bounded. Then $\mu(X)< \infty$ since we are assuming that all balls (with respect to $d$) have finite measure. Moreover, given $r\in(0,\infty)$, there exist $x_1,\dots,x_N \in X$, where $N\in\mathbb{N}$, such that
\[
X=\bigcup_{i=1}^N B_d(x_i, r/(C_d\widetilde{C}_d)).
\]
%Let $\widetilde{C}(r):=\inf_{i \in \{1,\dots,N\}}\mu(B(x_i,r/C_d))>0$. 
It follows that for all $x\in X$
there exists $j\in\{1,\dots,N\}$ such that $B_d(x_j,r/C_d)\subseteq B_d(x,r)$ and so we immediately have 
\[
\mu(B_d(x,r))\geq\inf_{i \in \{1,\dots,N\}}\mu(B_d(x_i,r/C_d))>0.
\]
Hence, $h(r)\geq\inf_{i \in \{1,\dots,N\}}\mu(B_d(x_i,r/C_d))>0$.

Next, we shall prove another implication. Let $r\in(0,\infty)$ and suppose that $A$ is a maximal $r$-separated set in $(X,d)$. Then $X= \bigcup_{x\in A} B_d(x,r)$. Since balls $\{B_d(x,r/(C_d\widetilde{C}_d))\}_{x\in A}$ are pairwise disjoint, we have
\begin{equation*}
 \#A \leq\frac{\mu (X)}{h\big(r/(C_d\widetilde{C}_d)\big)}<\infty.
\end{equation*}
Since $r\in(0,\infty)$ was arbitrary, it follows that $(X, d)$ is totally bounded, and the proof of the proposition is now complete.
\end{proof}

\subsection{Doubling-type conditions for measures}
\label{subsect:doubling measures}
Let $(X, d, \mu)$ be a quasi-metric-measure space. Recall that the measure $\mu$ is said to be \emph{doubling (with respect to $d$)} if there exists a constant $C_{\mu}\in[1,\infty)$ such that
\begin{equation}\label{doub}
\mu(B_d(x,2r))\leq C_{\mu}\mu(B_d(x,r))\quad\,\mbox{for all $x\in X$ and $r\in(0,\infty)$.}
\end{equation}
The smallest constant $C$ for which \eqref{doub} is satisfied will be denoted by $C_{\mu}$. 
The factor 2 in \eqref{doub} is immaterial and can be replaced by any $\theta\in(1,\infty)$ (with a possibly different choice of $C$).

%It is well known that the doubling condition in \eqref{doub} implies that
%%
%\begin{equation}\label{doub2}
%\mu(B_d(x,\lambda r))\leq C_\mu \lambda^{\log_2(C_\mu)} \mu(B_d(x,r)),
%\end{equation}
%whenever $x\in X$, $r\in(0,\infty)$, and $\lambda\in[1,\infty)$.

\begin{lem}\label{doubbdd}
Let $(X,d,\mu)$ be a quasi-metric-measure space with doubling measure. Then $\mu(X)<\infty$ if and only if $\diam_d(X) <\infty$.
\end{lem}
Lemma~\ref{doubbdd} is well known for symmetric quasi-metrics. The proof for nonsymmetric quasi-metrics is similar; however we include a short proof for the sake of completeness.
\begin{proof}
If $\diam_d(X) <\infty$ then obviously measure of $X$ is finite since we are assuming that balls with respect to $d$ have finite measure. 

Assume $\mu(X)<\infty$. Seeking a contradiction, let us suppose that $\diam_d(X)= \infty$. For fixed $x_0\in X$, we can find a sequence $\{x_n\}_{n\in\mathbb{N}}$ of points in $X\setminus \{x_0\}$ such that $r_n := d(x_n,x_0)\to \infty$ as $n\to\infty$. Observe that $B_d(x_0,r_n)\subseteq B_d(x_n, 2C_d r_n)$ for all $n\in\mathbb{N}$. Therefore, if we choose $k\in\mathbb{N}_0$ such that $2^{k}\geq C_d$,
%\[
%B(x_0,r_n)\subseteq \overline{B(x_n,C_d r_n)} \subseteq B(x_n, C_d ^2 r_n).
%\]
we can appeal to the doubling condition in \eqref{doub} to estimate
\begin{align*}
\mu(B_d(x_0,r_n)) &\leq \mu(B_d(x_n,2^{k+1} r_n)) \leq C_{\mu}^{2k+1} \mu(B_d(x_n,r_n/2^k)
\leq C_{\mu}^{2k+1} \mu(B_d(x_n,r_n/C_d))\\
&\leq C_{\mu}^{2k+1} \mu(X\setminus  B_d(x_0,r_n/(C_d\widetilde{C}_d))) = C_{\mu}^{2k+1} [\mu(X) - \mu(B_d(x_0,r_n/(C_d\widetilde{C}_d)))].
\end{align*}
After passing with $n\to \infty$ we obtain
\begin{align*}
\mu(X) \leq 0,
\end{align*}
which is a clear contradiction to the fact that all balls have positive measure. This completes the proof of the lemma.
\end{proof}

Now, we introduce a weak doubling-type property where the constant in the doubling condition \eqref{doub} is allowed to depend on the radius of the ball.
\begin{defi}\label{deltadoubdef}
Let $(X, d, \mu)$ be a quasi-metric-measure space and let $c\in[1,\infty)$ and $\delta\in(0,\infty)$. The measure $\mu$ is said to be \emph{$(c,\delta)$-doubling (with respect to $d$)} if there exists $C(\delta)\in[1,\infty)$ such that 
\begin{equation}
\label{deltadoub}
\mu(B_d(x,c\delta))\leq C(\delta)\mu(B_d(x,\delta))\quad\,\mbox{for all $x\in X$.}
\end{equation}
The smallest constant $C(\delta)$ for which \eqref{deltadoub} is satisfied will be denoted by $\Delta_c(\delta)$. 
\end{defi}

The $(c,\delta)$-doubling condition will play a central role in the sequel when characterizing $M^\alpha_{p,q}$ and $N^{\alpha}_{p,q}$ embeddings and so we take a moment to study this condition. Note that similar doubling-type properties have been considered in \cite{BK,Hyt,K}.

Clearly every doubling measure is $(c,\delta)$-doubling for all $c\in[1,\infty)$ and $\delta\in(0,\infty)$. However, as the next examples illustrate, the notion of $(c,\delta)$-doubling is strictly weaker than that of doubling.

\begin{ex}\label{exdis}
Let $(\mathbb{N}, d, \mu)$, where, $d$ is the discrete metric, i.e., for each $n,m\in\mathbb{N}$, $d(n,m) :=0$ if $n=m$ and $d(n,m):=1$ otherwise, and $\mu$ is defined by setting $\mu(\{n\}):=1/2^n$ for each $n\in\mathbb{N}$. Then, for every $c\in(1,\infty)$, the measure $\mu$ is $(c,\delta)$-doubling with $\Delta_c(\delta)=1$ for every $\delta\in(0,1/c]$. Moreover, $\mu$ is not doubling and the ball $B_d(1,2)=\mathbb{N}$ is not totally bounded.
\end{ex}
\begin{ex} \label{exp0}
Let $(\mathbb{R}_{+}, d, \mu)$, where $d:=|\cdot-\cdot|$ is the standard Euclidean distance and $\mu$ is defined by $d\mu:=e^{x^{\beta}}dx$ for any fixed $\beta\in (0,1]$. Then $\mu(\mathbb{R}_{+})=\infty$ and $\mu$ is not doubling. Moreover, for every $c\in(1,\infty)$, the measure $\mu$ is $(c,\delta)$-doubling with $\Delta_c(\delta)\leq2c e^{(c\delta)^{\beta}}$ for every $\delta\in(0,\infty)$.
\end{ex} 
\begin{proof}
Clearly, $\mu(\mathbb{R}_{+})=\infty$ and $\mu$ is not doubling since 
\[
\lim_{x \rightarrow \infty} \frac{\mu(B_d(x,2x))}{\mu(B_d(x,x))}=\lim_{x \rightarrow \infty} \frac{\int\limits_{0}^{3x}e^{y^{\beta}}dy}{\int\limits_{0}^{2x}e^{y^{\beta}}dy}=
\lim_{x \rightarrow \infty} \frac{3e^{(3x)^{\beta}}}{2e^{(2x)^{\beta}}}=\infty.
\]
 If $c\in(1,\infty)$, then for every $x\in\mathbb{R}_{+}$ and $\delta\in(0,\infty)$, we can estimate
\[
\frac{\mu(B_d(x,c\delta))}{\mu(B_d(x,\delta))}= 
\frac{\int\limits_{\max\{0,x-c\delta\}}^{x+c\delta}e^{y^{\beta}}dy}{\int\limits_{\max\{0,x-\delta\}}^{x+\delta}e^{y^{\beta}}dy}
\leq\frac{\int\limits_{x-c\delta}^{x+c\delta}e^{y^{\beta}}dy}{\int\limits_{x}^{x+\delta}e^{y^{\beta}}dy}    
\leq \frac{2c\delta e^{(x+c\delta)^{\beta}}}{\delta e^{x^\beta}} \leq 2c e^{(c\delta)^{\beta}}.
\]
 Since $\beta$-H\"{o}lder continuity of function $f(t) = t^{\beta}$ gives us $(x+c\delta)^{\beta} \leq x^{\beta} + (c\delta)^{\beta}$, and the claim follows. 
\end{proof}

\begin{ex} \label{exp}
Let $(\mathbb{R}_{+}, d, \mu)$, where $d:=|\cdot-\cdot|$ is the standard Euclidean distance and $\mu$ is defined by $d\mu:=e^{-x^{\beta}}dx$ for any fixed $\beta\in (0,1]$. Then $\mu(\mathbb{R}_{+})<\infty$  and $\mu$ is not doubling. Moreover, for every $c\in(1,\infty)$, the measure $\mu$ is $(c,\delta)$-doubling with $\Delta_c(\delta)\leq2c e^{[(c+1)\delta]^{\beta}}$ for every $\delta\in(0,\infty)$.
\end{ex}
\begin{proof}
Clearly, $\mu(\mathbb{R}_{+})<\infty$ which, in view of Lemma~\ref{doubbdd}, immediately implies that $\mu$ is not doubling. Now, let $c\in(1,\infty)$, $x\in\mathbb{R}_{+}$ and $\delta\in(0,\infty)$. If $x>c\delta$, then we have
\begin{align*}
\dfrac{\mu(B_d(x,c\delta))}{\mu(B_d(x,\delta))} = \dfrac{ \int\limits_{x-c\delta}^{x+c\delta} e^{-y^{\beta}}dy}{\int\limits_{x-\delta}^{x+\delta} e^{-y^{\beta}}dy} 
\leq\dfrac{ \int\limits_{x-c\delta}^{x+c\delta} e^{-y^{\beta}}dy}{\int\limits_{x-\delta}^{x} e^{-y^{\beta}}dy}
\leq \dfrac{2c\delta e^{-(x-c\delta)^{\beta}}}{\delta e^{-x^{\beta}}} = 2c e^{-(x-c\delta)^{\beta} + x^{\beta}}.
\end{align*}
Since 
\begin{align*}
x^{\beta} - (c\delta)^{\beta}\leq (x-c\delta)^{\beta},
\end{align*}
we obtain
\begin{align*}
\dfrac{\mu(B_d(x,c\delta))}{\mu(B_d(x,\delta))} \leq 2c e^{(c\delta)^{\beta}}.
\end{align*}
On the other hand, if $x\leq c\delta$ then we have
\begin{align*}
\dfrac{\mu(B_d(x,c\delta))}{\mu(B_d(x,\delta))} = \frac{ \int\limits_{0}^{x+c\delta}e^{-y^{\beta}}dy}{\int\limits_{\max\{0,x-\delta\}}^{x+\delta}e^{-y^{\beta}}dy} \leq \dfrac{x+c\delta}{\int\limits_{x}^{x+\delta}e^{-y^{\beta}}dy} \leq \dfrac{2c\delta}{\delta e^{-(x+\delta)^{\beta}}} = 2c e^{(x+\delta)^{\beta}} \leq 2c e^{[(c+1)\delta]^{\beta}}.
\end{align*}
The claim now follows.
\end{proof}

In Examples~\ref{exdis}-\ref{exp}, the measures are $(c,\delta)$-doubling for all $\delta\in(0,\infty)$ but not doubling; however, since the constant $\Delta_c(\delta)$ in each example is bounded when $\delta$ is bounded, these measures are \textit{locally doubling} in the sense that there exists some $r_0\in(0,\infty)$ such that the doubling condition \eqref{doub} holds for all radii $r\in(0,r_0]$. The following example illustrates that $(c,\delta)$-doubling measures need not even be locally doubling.
\begin{ex} \label{nonlocaldoub}
Let $([0,\infty), d, \mu)$, where $d:=|\cdot-\cdot|$ is the standard Euclidean distance and $\mu$ is defined by $d\mu:=e^{-\frac{1}{x^\beta}}dx$ for any fixed $\beta\in (0,\infty)$. 
Then $\mu$ is not locally doubling. Moreover, for every $c\in(1,\infty)$, the measure $\mu$ is $(c,\delta)$-doubling with $\Delta_c(\delta)\leq 4c e^{(\frac{2}{\delta})^\beta}$ for every $\delta\in(0,\infty)$.
\end{ex} 
\begin{proof}
The measure $\mu$ is not locally doubling since 
\[
\lim_{x \rightarrow 0^+} \frac{\mu(B_d(0,2x))}{\mu(B_d(0,x))}=\lim_{x \rightarrow 0^+} \frac{\int\limits_{0}^{2x}e^{-\frac{1}{y^\beta}}dy}{\int\limits_{0}^{x}e^{-\frac{1}{y^\beta}}dy}=
\lim_{x \rightarrow 0^+} \frac{2e^{-\frac{1}{(2x)^\beta}}}{e^{-\frac{1}{x^\beta}}}=\infty.
\]
If $c\in(1,\infty)$, then for every $x\in[0,\infty)$ and $\delta\in(0,\infty)$, we can estimate
\[
\frac{\mu(B_d(x,c\delta))}{\mu(B_d(x,\delta))}= 
\frac{\int\limits_{\max\{0,x-c\delta\}}^{x+c\delta}e^{-\frac{1}{y^\beta}}dy}{\int\limits_{\max\{0,x-\delta\}}^{x+\delta}e^{-\frac{1}{y^\beta}}dy}
\leq\frac{\int\limits_{x-c\delta}^{x+c\delta}e^{-\frac{1}{y^\beta}}dy}{\int\limits_{x+\delta/2}^{x+\delta}e^{-\frac{1}{y^\beta}}dy}    
\leq \frac{2c\delta e^{-\frac{1}{(x+c\delta)^\beta}}}{\frac{\delta}{2} e^{-\frac{1}{(x+\delta/2)^\beta}}}
%=4ce^{\frac{c\delta-\delta/2}{(x+\delta/2)(x+c\delta)}}
\leq 4ce^{\frac{1}{(x+\delta/2)^\beta}}
\leq 4ce^{(\frac{2}{\delta})^\beta}.
\]
The claim now follows. 
\end{proof}

In fact, the measure in Example~\ref{nonlocaldoub} fails to even satisfy the following weaker formulation of the locally doubling condition: For all $x\in X$, there exist $C_x,r_x\in(0,\infty)$ such that $\mu(B_d(x,2r))\leq C_x\mu(B_d(x,r))$ for all $r\in(0,r_x]$.
\begin{prop} \label{tot1}
Let $(X,d,\mu)$ be a quasi-metric-measure space and let $c\in(1,\infty)$. Then, $(X,d)$ is totally bounded if and only if $(X,d)$ is bounded and $\mu$ is $(c,\delta)$-doubling for all $\delta \in (0,\diam_d(X)]$. 
\end{prop}
In view of Example~\ref{exdis}, Proposition~\ref{tot1} can fail if $\mu$ is only $(c,\delta)$-doubling for  $\delta\leq\delta_0<\diam_d(X)$.

\begin{proof}
Suppose first that $(X,d)$ is bounded and that $\mu$ is $(c,\delta)$-doubling for all $\delta \in (0,\diam_d(X)]$. It is clear from \eqref{deltadoub} that $\mu$ is $(c,\delta)$-doubling for $\delta \in (0,\infty)$.
To show that $(X,d)$ is totally bounded, fix $\varepsilon \in (0,\diam_d(X)]$ and suppose that $A$ is a maximal $\varepsilon$-separated set in $X$. Since $A$ is $\varepsilon$-net in $X$, it is enough for us to show that $A$ is finite. To this end, consider the smallest $k\in \mathbb{N}$ such that $c^{k}  \varepsilon > \widetilde{C}_d C_d\diam_d(X)$. Given that the collection $\{B_d(x,\varepsilon/(\widetilde{C}_d C_d)) \}_{x\in A}$ of balls are pairwise disjoint, we have
\begin{align*}
\mu(X) &\geq \sum_{x\in A} \mu(B_d(x,\varepsilon/(\widetilde{C}_d C_d))) \geq [\Delta_c(\varepsilon/(\widetilde{C}_dC_d))]^{-1} \sum_{x\in A} \mu(B_d(x,c\varepsilon/(\widetilde{C}_dC_d))) \\
&\geq [\Delta_c(\varepsilon/(\widetilde{C}_dC_d))]^{-1} [\Delta_c(c\varepsilon/(\widetilde{C}_dC_d))]^{-1} \sum_{x\in A} \mu(B_d(x,c^2\varepsilon/(\widetilde{C}_dC_d))) \\ &\geq \dots \geq \prod_{\ell=0}^{k-1} [\Delta_c(c^\ell\varepsilon/(\widetilde{C}_dC_d))]^{-1} \sum_{x\in A} \mu(B_d(x,c^k\varepsilon/(\widetilde{C}_dC_d))) \\
&\geq \prod_{\ell=0}^{k-1} [\Delta_c(c^\ell\varepsilon/(\widetilde{C}_dC_d))]^{-1} \mu(X)\cdot \#A
\end{align*}
Hence $\#A \leq \prod_{\ell=0}^{k-1} \Delta_c(c^\ell\varepsilon/C_d)<\infty$, as wanted.

Now, assume that $(X,d)$ is totally bounded. Since we clearly have that $(X,d)$ is bounded, we focus on proving that $\mu$ is $(c,\delta)$-doubling for all $\delta \in (0,\diam_d(X)]$. First, note that $\mu(X)<\infty$ since $(X, d, \mu)$ is a bounded quasi-metric-measure space. Now, having Proposition~\ref{tot} in mind, let $\delta \in (0,\diam_d(X)]$ and take $h(\delta) := \inf_{y\in X} \mu(B_d(y,\delta))>0$. Then for $x\in X$, we have
\begin{align*}
\mu(B_d(x,c \delta)) \leq\frac{\mu(X)}{h(\delta)} \mu(B_d(x,\delta)).
\end{align*}
Hence $\mu$ satisfies \eqref{deltadoub} with $C(\delta):= \mu(X)/h(\delta)\in[1,\infty)$ and the proof of the proposition is now complete.
\end{proof}
Gathering Proposition~\ref{tot} with Proposition~\ref{tot1} we get the following theorem.
\begin{tw}\label{allin}
Let $(X,d,\mu)$ be a quasi-metric-measure space and let $c\in(1,\infty)$. Then, the following conditions are equivalent:
\begin{enumerate}
\item $(X, d)$ is totally bounded;
\item $\mu(X)< \infty$ and for all $r\in(0,\infty)$,  $\inf_{x\in X} \mu (B_d(x,r))>0$;
\item  $(X,d)$ is bounded and $\mu$ is $(c,\delta)$-doubling for all $\delta \in (0,\diam_d(X)]$.
\end{enumerate}
\end{tw}
\begin{prop}\label{doub-diff}
Let $(X, d, \mu)$ be a quasi-metric-measure space and suppose that there exist $\delta_0\in(0,\infty]$ and $c\in(1,\infty)$ such that $\mu$ is $(c,\delta)$-doubling for every
finite $\delta\in(0,\delta_0]$. Then for each fixed $c'\in[1,\infty)$, the measure $\mu$ is $(c',\delta)$-doubling for every finite $\delta\in(0,c\delta_0/c']$.
\end{prop}
\begin{proof}
Let $c'\in[1,\infty)$ and suppose that $\delta\in(0,\infty)$ satisfies $\delta\leq c\delta_0/c'$. Consider the smallest $k\in \mathbb{N}$ such that $c^k\geq c'$. Then $c'\delta/c^\ell\leq\delta_0$ for every $\ell\in\mathbb{N}$ and so we can iterate the $(c,\delta)$-doubling property to estimate for every $x\in X$, 
\begin{align*}
\mu(B_d(x,c'\delta))&=\mu(B_d(x,c(c'\delta/c)))\leq \Delta_c(c'\delta/c) \mu(B_d(x,c'\delta/c)) \\
&\leq \Delta_c(c'\delta/c)\Delta_c(c'\delta/c^2) \mu(B_d(x,c'\delta/c^2)) \\ 
&\leq \dots \leq \prod_{\ell=1}^{k} \Delta_c(c'\delta/c^\ell) \mu(B_d(x,c'\delta/c^k)) \\
&\leq \prod_{\ell=1}^{k} \Delta_c(c'\delta/c^\ell) \mu(B_d(x,\delta)).
\end{align*}
Hence, $\mu$ is $(c',\delta)$-doubling for every finite $\delta\in(0,c\delta_0/c']$ with
$\Delta_{c'}(\delta)\leq\prod_{\ell=1}^{k} \Delta_c(c'\delta/c^\ell)$, and the proof of the proposition is complete.
\end{proof}

The following doubling-type condition will be used later to show that there are quasi-metric-measure spaces where $M^{\alpha}_{p,q}$ and $N^{\alpha}_{p,q}$ never embed compactly into $L^p$ (see Theorem~\ref{thm-doubinf}).

\begin{defi}\label{doub-inf}
Let $(X,d,\mu)$ be an unbounded quasi-metric-measure space such that $\mu(X) <\infty$. With $C_d\in[1,\infty)$ as in \eqref{C-d}, we say that $(X,d,\mu)$ is \emph{doubling at infinity} if there is a point $x_{0}\in X$ and constant $C\in(0,\infty)$ such that 
\begin{align*}
\liminf_{R \to \infty}\dfrac{\mu(X\setminus  B_d(x_0,R))}{\mu(X\setminus  B_d(x_0,C_d R))} < \infty.
\end{align*} 
\end{defi}

\begin{ex}
Let $([0, \infty), d, \mu)$,  where $d:=|\cdot-\cdot|$ is the standard Euclidean distance and $\mu$ is defined by $d\mu := f(x) dx$, where $f \equiv 1$ on $[0,1)$ and for $k\geq 0$ 
\begin{align*}
f(x) := a_{k} \exp(2^{k}x - 2\cdot 4^{k}) \textit{ if } x\in [2^{k},2^{k+1}),
\end{align*}
where $a_k = \frac{\exp(4^{k})}{\exp(4^k)-1}$. Then, $\mu$ is doubling at infinity, but not $(2,\delta)$-doubling for any $\delta\in(0,\infty)$. 
\end{ex}
\begin{proof}
Note that $C_d=2$. Let  $N\in \mathbb{N}$ and $k\in \mathbb{N}$ be such that $2^{k} \leq N < 2^{k+1}$. Then 
\begin{align*}
\dfrac{\mu(X\setminus  B_d(0,N))}{\mu(X\setminus  B_d(0,2N))} \leq \dfrac{\mu(X\setminus  B_d(0,2^{k}))}{\mu(X\setminus  B_d(0,2^{k+2}))} = \dfrac{\sum_{l=k}^{\infty}2^{-l}}{\sum_{l=k+2}^{\infty}2^{-l}} = 4,
\end{align*} 
and hence, $\mu$ is doubling at infinity.

Next, we will shall show that $\mu$ is not $(2,\delta)$-doubling for any $\delta\in(0,\infty)$. 
Let $\delta\in(0,\infty)$  and consider a sufficiently large $k\in \mathbb{N}\cap [2, \infty)$ for which $\delta < 2^{k-2}$. Since $2^{k} + 2\delta < \frac{3}{2}2^{k}$, we have
\begin{align*}
\mu(B_d(2^{k} + \delta,\delta)) \leq a_{k} \int\limits_{(2^{k}, \frac{3}{2}2^{k})} \exp(2^{k}x - 2\cdot 4^{k})dx = a_k 2^{-k} \Big(\exp\big(-\frac{1}{2}4^k\big) - \exp\big(-4^k\big)\Big).
\end{align*}
On the other hand, we have 
\begin{align*}
\mu(B_d(2^{k} + \delta,2\delta)) \geq \mu((2^{k}-\delta,2^{k})) = a_{k-1}\int\limits_{(2^{k}-\delta,2^{k})} \exp(2^{k-1}x - 2\cdot 4^{k-1}) dx = 2^{-k + 1} a_{k-1} (1 - \exp(-\delta 2^{k-1})).
\end{align*}
Combining  the above inequalities we get
\begin{align*}
\dfrac{\mu(B_d(2^{k} + \delta,2\delta))}{\mu(B_d(2^{k} + \delta,\delta))} \geq \dfrac{2a_{k-1}}{a_k} \dfrac{1 - \exp(-\delta 2^{k-1})}{\exp\big(-\frac{1}{2}4^k\big) - \exp\big(-4^k\big)} \to \infty,
\end{align*}
when $k\to \infty$. Hence,  $\mu$ is not $(2,\delta)$-doubling, as wanted.
\end{proof}

\subsection{Integrable measures}
\label{subsect: integrable measures}

We now introduce an integrability condition for a measure which, as we will see in Subsection~\ref{subsect:p}, turns out to be closely related to totally boundedness and compact embeddings of the spaces $M^\alpha_{p,q}$ and $N^\alpha_{p,q}$ into $L^p$.
\begin{defi}
Let $(X, d, \mu)$ be a quasi-metric-measure space. The measure $\mu$ is said to be \emph{integrable} if there exists a quasi-metric $\rho$ on $X$ which is equivalent to $d$ and satisfies the following condition: for each fixed $r\in(0,\infty)$, the mapping $x\mapsto\mu(B_\rho(x,r))$ is well defined, measurable, and
\begin{eqnarray}\label{warunek}
\int_X \frac{1}{\mu(B_\rho(x,r))}\, d\mu (x)< \infty.
\end{eqnarray}
\end{defi}
If $d$ is a genuine \textit{metric} then the mapping $x\mapsto\mu(B_d(x,r))$ is always well defined and measurable. For a general quasi-metric, this mapping may not be measurable; however, it follows from Lemma~\ref{DST1} that there is always a quasi-metric $\rho$ on $X$ which is equivalent to $d$ and is continuous on $X\times X$. Hence, the mapping $x\mapsto\mu(B_\rho(x,r))$ is well defined and measurable\footnote{It is straightforward to show that this mapping is lower semicontinuous using the continuity of $\rho$.} for each fixed $r\in(0,\infty)$.
\begin{ex}\label{exint1}
Let $(\mathbb{R}_{+}, d, \mu)$, where $d:=|\cdot-\cdot|$ is the standard Euclidean distance and $\mu$ is defined by $d\mu:=e^{-x^{\beta}}dx$ for any fixed $\beta\in (1,\infty)$. Then  $\mu (\mathbb{R}_{+})<\infty$ and $\mu$ is integrable.
\end{ex}
\begin{proof}
Fix $r\in(0,\infty)$. Then
\begin{align*}
\int\limits_{\mathbb{R}_{+}} \dfrac{1}{\mu(B_d(x,r))}\, d\mu(x) &= \int\limits_{(r,\infty)}\dfrac{e^{-x^{\beta}}}{\int\limits_{x-r}^{x+r} e^{-y^{\beta}}dy} dx + \int\limits_{(0,r)} \dfrac{e^{-x^{\beta}}}{\int\limits_{0}^{x+r} e^{-y^{\beta}}dy} dx \\
&\leq \int\limits_{(r,\infty)}\dfrac{e^{-x^{\beta}}}{\int\limits_{x-r}^{x-r/2} e^{-y^{\beta}}dy} dx +   \int\limits_{(0,r)} \dfrac{e^{-x^{\beta}}}{(x+r)e^{-(x+r)^{\beta}}} dx \\
&\leq \int\limits_{(r,\infty)} \dfrac{e^{-x^{\beta}}}{\frac{r}{2}e^{-(x-\frac{r}{2})^{\beta}}} dx + \frac{1}{r}\int\limits_{(0,r)} e^{-x^{\beta} + (x+r)^{\beta}} dx \\
&\leq \frac{2}{r}\int\limits_{(r,\infty)} e^{-x^{\beta}+(x-\frac{r}{2})^{\beta}}dx + \frac{1}{r}\int\limits_{(0,r)} e^{(2r)^{\beta}} dx \\
&\leq \frac{2}{r}\int\limits_{(r,\infty)} e^{-\beta \frac{r}{2}(x-\frac{r}{2})^{\beta-1}}dx   +e^{(2r)^{\beta}} <\infty,
\end{align*}
where we have made use of the  mean value theorem in obtaining the fourth inequality. Therefore, $\mu$ is integrable.
\end{proof}

\begin{ex} \label{expbeta}
Let $(\mathbb{R}_{+}, d, \mu)$, where $d:=|\cdot-\cdot|$ is the standard Euclidean distance and $\mu$ is defined by $d\mu:=e^{x^{\beta}}dx$ for any fixed $\beta\in (1,\infty)$. Then $\mu (\mathbb{R}_{+})=\infty$ and $\mu$ is integrable.
\end{ex}
\begin{proof}
Clearly, $\mu(\mathbb{R}_{+})=\infty$. Observe that for all $x\in\mathbb{R}_{+}$ and $r\in(0,\infty)$, we have
\[
	\mu(B_d(x,r))=\int_{\max\{0,x-r\}}^{x+r}e^{y^{\beta}}dy \geq \int_{x+r/2}^{x+r}e^{y^{\beta}}dy\geq \frac{r}{2}e^{(x+r/2)^{\beta}}.
\]
Combining this with the mean value theorem (applied to the function $f(t)=t^\beta$), gives
\begin{align*}
\int\limits_{\mathbb{R}_{+}} \dfrac{1}{\mu(B_d(x,r))}\, d\mu(x) =\int_{\mathbb{R}_{+}}  \frac{e^{x^{\beta}}}{\mu(B_d(x,r))}dx \leq \frac{2}{r} \int_{\mathbb{R}_{+}}  \frac{dx}{e^{(x+r/2)^{\beta}-x^{\beta}}}
\leq \frac{2}{r} \int_{\mathbb{R}_{+}}  \frac{dx}{e^{\frac{\beta r}{2} x^{\beta-1}}}<\infty.
\end{align*}
Hence, $\mu$ is integrable.
\end{proof}
\begin{prop}\label{TB-int}
Let $(X,d,\mu)$ be a quasi-metric-measure space. If $(X, d)$ is totally bounded, then $\mu$ is integrable.
\end{prop}
Before proving Proposition~\ref{TB-int}, we make a few remarks. First, in light of Lemma~\ref{doubbdd}, Proposition~\ref{TB-int}, and Theorem~\ref{allin}, on a quasi-metric-measure spaces of finite measure, every doubling measure and every locally lower-Ahlfors regular\footnote{see \eqref{llAr} in Subsection~\ref{subsect:p-greater} for the definition of a locally lower-Ahlfors regular measure.} measure is integrable. Moreover, one can see from Examples~\ref{exint1}, \ref{expbeta}, and \ref{infinitecomb} that there are integrable measures on metric-measure spaces which are not totally bounded. Under certain assumptions, integrability of the measure fully characterizes totally bounded quasi-metric-measure spaces (see Theorems~\ref{piate} and \ref{piate-ult}). Finally, the total boundedness assumption in Proposition~\ref{TB-int} cannot be relaxed to simply boundedness, in general (see, for instance, Example~\ref{exdis}).
\begin{proof}[Proof of Proposition~\ref{TB-int}]
Suppose $(X, d)$ is totally bounded. By Lemma~\ref{DST1}, there exists a symmetric quasi-metric $\rho$ on $X$ which is equivalent to $d$ and for which the mapping $x\mapsto\mu(B_\rho(x,r))$ is well defined and measurable for each fixed $r\in(0,\infty)$. Since $\rho\approx d$, we have that the quasi-metric space $(X,\rho)$ is totally bounded. In particular, by Proposition~\ref{tot} we have that $h(r):=\inf_{x\in X} \mu (B_\rho(x,r))>0$ and therefore,
\[
	\int_X \frac{1}{\mu(B_\rho(x,r))}\, d\mu (x) \leq\frac{\mu(X)}{h(r)} < \infty.
\]
Hence, $\mu$ is integrable.
\end{proof}
\begin{prop}
Let $(X,d,\mu)$ be an unbounded quasi-metric-measure space with the property that $\mu$ is $(C_d,\delta)$-doubling for some fixed $\delta\in(0,\infty)$, where $C_d\in[1,\infty)$ is as in \eqref{C-d}. Suppose that the mapping $x\mapsto \mu(B_d(x,\delta))$ is measurable. Then $\mu$ is not integrable.
\end{prop}
\begin{proof}
Let $A \subseteq X$ be a maximal $C_d\widetilde{C}_d \delta$-separated set, where $\widetilde{C}_d \in[1,\infty)$  is as in \eqref{C-d-tilde}. Since $X$ is unbounded, we have $\#A =\infty$. Moreover, for any $x, y \in A$ with $x\neq y$ we have $B_d(x,\delta)\cap B_d(y,\delta)=\emptyset$. Hence,
\begin{align*}
\int_X \frac{1}{\mu(B_d(x,\delta))}\, d\mu (x) &\geq \sum_{y\in A} \int_{B_d(y,\delta)} \frac{1}{\mu(B_d(x,\delta))}\, d\mu (x) \\
& \geq \sum_{y\in A} \int_{B_d(y,\delta)} \frac{1}{\mu(B_d(y,C_d \delta))}\, d\mu (x)\\
& = \sum_{y\in A}  \frac{\mu(B_d(y,\delta))}{\mu(B_d(y,C_d\delta))}\geq \sum_{y\in A}  \frac{1}{\Delta_{C_d}(\delta)}=\infty.
\end{align*}
Now suppose that $\rho$ is any quasi-metric on $X$ which is equivalent to $d$ and has the property that the mapping $x\mapsto\mu(B_\rho(x,r))$ is well defined and measurable for each fixed $r\in(0,\infty)$. Since $\rho\approx d$, we
can find a positive constant $\kappa$ such that $B_\rho(x,\kappa^{-1}\delta)\subseteq B_d(x, \delta)$ for any $x\in X$. Therefore,
\[\infty=\int_X \frac{1}{\mu(B_d(x,\delta))}\, d\mu (x)\leq\int_X \frac{1}{\mu(B_\rho(x, \kappa^{-1}\delta))}\, d\mu (x),
\]
from which it follows that $\mu$ is not integrable.
\end{proof}

\section{Borel regularity and Lusin's theorem}
\label{sect:lusin}
In this section we record a version of Lusin's theorem (Theorem~\ref{lusin}) for topological spaces equipped with a  finite Borel measure under minimal assumptions.
We begin by recording a few measure theoretic notions.
Given an arbitrary 
set $X$ and a topology $\tau$ on $X$, recall that
${\it Borel}_{\tau}(X)$ denotes the Borel sigma-algebra of 
$X$, and that $\mu:\mathfrak{M}\to[0,\infty]$ is a Borel measure if the sigma-algebra, $\mathfrak{M}$, of subsets of $X$ satisfies ${\it Borel}_{\tau}(X)\subseteq\mathfrak{M}$. In this context, $\mu$ is said to be \emph{Borel-regular} if $\mu$ is a Borel measure on $X$ with the property that every measurable set is contained in a Borel set of equal measure.
The reader 
is alerted to the fact that for a Borel measure, we merely
demand that $\mathfrak{M}$ contains ${\textit Borel}_{\tau}(X)$
and not necessarily that 
$\mathfrak{M}={\textit Borel}_{\tau}(X)$. In the latter
case the measure $\mu$ would automatically be Borel-regular.

\begin{lem}
\label{Borel-in-out-reg}
Let $(X, \tau)$ be a topological space and $\mu:\mathfrak{M}\to[0,\infty]$ be a Borel measure on $X$ such that $\mu(X) < \infty$. Then, the following conditions are equivalent:
\begin{enumerate}[i)]
\item for every open set $U\subseteq X$ there exist an $F_{\sigma}$-set $F$ and a null set $Z$ such that $U=F \cup Z$,
\item for  every Borel set $B\subseteq X$ we have 
\[
\mu(B) =\mathop{\sup}_{\substack{C- \text{closed},\\ C\subseteq B}} \mu(C).
\]
\end{enumerate}
\end{lem}
\begin{rem}
From the proof of Lemma~\ref{Borel-in-out-reg}, it is easy to see that \textit{ii)} implies the following stronger version of \textit{i)}: for every Borel set $B\subseteq X$ there exist an $F_{\sigma}$-set $F$ and a null set $Z$ such that $B=F \cup Z$.
\end{rem}

\begin{proof}
Let us suppose that $i)$ holds. We define the set
\[
	\mathcal{F}:=\{A \in \mathfrak{M}: \forall_{\varepsilon >0} \, \text{there exists } \, C\subseteq A \, \text{closed set such that}\, \mu(A\setminus C) <\varepsilon \}.
\]
Let us note that all closed sets and null sets belong to $\mathcal{F}$. Moreover, by arguing as in \cite[Lemma~3.11]{AM15} we have that if $\{A_i\}_{i=1}^{\infty} \subseteq \mathcal{F}$, then the sets $\bigcup_{i=1}^{\infty}A_i, \bigcap_{i=1}^{\infty}A_i \in \mathcal{F}$. Therefore, by assumption, all open sets in $X$ belong to $\mathcal{F}$. Granted this, it is straightforward to check that the set
\[
	\widetilde{\mathcal{F}}:=\{A\in\mathcal{F}:X\setminus A \in \mathcal{F}\}
\]
is a $\sigma$-algebra containing open sets of $X$, and therefore $\widetilde{\mathcal{F}}$ contains $Borel_{\tau}(X)$. Thus, in particular, every Borel set $B\subseteq X$ belongs to $\mathcal{F}$ and this finishes the proof of $ii)$.

Next, let us suppose that $ii)$ holds and let $U\subseteq X$ be an open set. Then, for every $n\in\mathbb{N}$ there exists a closed set $C_n\subseteq X$ such that $C_n \subseteq U$ and $\mu(U \setminus C_n)<1/n$. We can assume that $C_n \subseteq C_{n+1}$ for all $n\in\mathbb{N}$. Let $C:=\bigcup_{n=1}^{\infty} C_n$. Then $C$ is an $F_{\sigma}$-set and we can write $U=C\cup(U\setminus C)$, where $\mu(U\setminus C)=0$. This finishes the proof of $i)$ and hence, the proof of the lemma.
\end{proof}
\begin{tw}\label{in-out-reg}
Let $(X, \tau)$ be a topological space and $\mu:\mathfrak{M}\to[0,\infty]$ be a Borel measure on $X$ such that $\mu(X) < \infty$. Then, the following conditions are equivalent:
\begin{enumerate}[i)]
\item $\mu$ is Borel regular and for every open set $U\subseteq X$ there exist a $F_{\sigma}$-set $F$ and a null set $Z$ such that $U=F \cup Z$,
\item for  every $E \in \mathfrak{M}$ and every $\varepsilon\in(0,\infty)$, there exist an open set $U_{\varepsilon}\subseteq X$ and a closed set $C_{\varepsilon}\subseteq X$  such that $C_{\varepsilon} \subseteq E \subseteq U_{\varepsilon}$ and 
\[
		\mu(U_{\varepsilon} \setminus E),\,\, \mu(E \setminus C_{\varepsilon}) < \varepsilon.
\]
\end{enumerate}
\end{tw}
\begin{proof}
Let us suppose that $i)$ holds and let us fix $E \in \mathfrak{M}$, $\varepsilon\in(0,\infty)$. Then there exists a Borel set $B$ such that $E\subseteq B$ and $\mu(E)=\mu(B)$. By Lemma~\ref{Borel-in-out-reg} there exists a closed set $D_{\varepsilon}\subseteq X$ such that $D_{\varepsilon}\subseteq X \setminus B$ and $\mu((X\setminus B)\setminus D_{\varepsilon}) < \varepsilon$. Therefore, $U_{\varepsilon}:=X\setminus D_{\varepsilon}$ is an open set in $\tau$ satisfying $E\subseteq B \subseteq U_{\varepsilon}$, and since $\mu(B)=\mu(E)<\infty$, we can estimate
\[
	\mu(U_{\varepsilon} \setminus E) = \mu(U_{\varepsilon} \setminus B) = \mu((X\setminus B)\setminus D_{\varepsilon}) < \varepsilon.
\] 
Now, we shall prove inner-regularity. From the above considerations there exists an open set $O\subseteq X$ such that $B\setminus E \subseteq O$ and
\[
	\mu(O) < \mu(B \setminus E) + \varepsilon /2
\]
and by Lemma~\ref{Borel-in-out-reg} there exists a closed set $C \subseteq B$ such that 
\[
	\mu(B)<\mu(C) +\varepsilon /2.
\]
Therefore, $C_{\varepsilon}:=C \setminus O$ is a closed set contained in $E$ and 
\[
	\mu(E\setminus C_{\varepsilon}) \leq \mu(B \setminus C)+  \mu(O \setminus (B \setminus E) )< \varepsilon.
\] 

Next, let us assume that $ii)$ holds. Let $E \in \mathfrak{M}$, then for every $n\in\mathbb{N}$ there exists an open set $U_n\subseteq X$ such that $E \subseteq U_n$ and $\mu (U_n \setminus E) < 1/n$. We can assume that $U_{n+1} \subseteq U_n$ for all $n\in\mathbb{N}$. Then $B:=\bigcap_{n=1}^{\infty} U_n$ is a Borel set such that $E \subseteq B$ and $\mu(B)=\mu(E)$.  Moreover, by Lemma~\ref{Borel-in-out-reg} we get that every open set is a sum of a null set and $F_{\sigma}$-set and this completes the proof of the theorem.
\end{proof}

\begin{tw}[Lusin Theorem]
\label{lusin}
Let $(X,\tau,\mu)$ be a topological space equipped with a nonnegative Borel regular measure $\mu$ such that $\mu(X) < \infty$, and suppose that every open set $U\subseteq X$ can be written as $U=F \cup Z$, where $F$ is an $F_{\sigma}$-set and  $Z$ is a  null set. Then, for any measurable, finite a.e. function $u:X\to\overline{\mathbb{R}}$ and any $\varepsilon\in(0,\infty)$ there exists closed set $G_{\varepsilon}\subseteq X$ such that 
$\mu(X\setminus G_{\varepsilon}) < \varepsilon$ and restriction of $u$ to $G_{\varepsilon}$ is continuous.
\end{tw}
It was shown in \cite{AGS-reg} that Borel regularity is actually necessary for the conclusion of Theorem~\ref{lusin} to hold, even without the assumptions that $\mu(X) < \infty$ and every open set $U\subseteq X$ can be written as $U=F \cup Z$, where $F$ is an $F_{\sigma}$-set and  $Z$ is a  null set.
\begin{proof}
Let $\{V_{n}\}_{n=1}^{\infty}$ be a countable basis of  Euclidean topology of $\mathbb{R}$ and let $\varepsilon\in(0,\infty)$. For each $n\in\mathbb{N}$ let $E_{n} := u^{-1}(V_{n})$, and set $E_{\infty} := u^{-1}(\{-\infty, \infty\})$. Granted our current assumptions on $(X,\tau,\mu)$, by Theorem~\ref{in-out-reg}, we can find open set $U_n\subseteq X$ and closed set $C_{n}\subseteq X$ such that $C_{n}\subseteq E_{n} \subseteq U_{n}$ and
\begin{align*}
\mu(U_{n}\setminus C_{n}) < \varepsilon/2^{n+1}.
\end{align*}
Moreover, since $\mu(E_{\infty})=0$, there exists an open set $U_{\infty}$ such that $E_{\infty} \subseteq U_{\infty}$ and $\mu(U_{\infty}) <\varepsilon/2$.
Furthermore, we define 
\begin{align*}
F_{\varepsilon} := \left(\bigcup_{n=1}^{\infty} U_{n}\setminus C_{n}\right) \cup U_{\infty} \quad \text{and} \quad G_{\varepsilon} := X\setminus F_{\varepsilon}.
\end{align*}
It is easy to see that $F_{\varepsilon}$ is an open set satisfying $\mu(F_{\varepsilon}) < \varepsilon$. We shall now show that the function $g := u\vert_{G_{\varepsilon}}$ is continuous. Observe that for any $n\in \mathbb{N}$ we have 
\begin{align*}
U_{n}\cap G_{\varepsilon} =  E_{n}\cap G_{\varepsilon} = g^{-1}(V_n).
\end{align*}
Thus, $g^{-1}(V_n) = U_{n}\cap G_{\varepsilon}$ is an open set in the subspace topology on $G_{\varepsilon}$. Finally, since $\{V_{n} \}_{n=1}^{\infty}$ is the basis of $\mathbb{R}$, it follows that $g$ is continuous, and the proof of the theorem is now complete.
\end{proof}

\section{Total boundedness in $L^p$ spaces with $p\geq0$}
\label{sect:totalbdd}
In this section we study totally bounded sets in $L^p$ spaces with $p\geq0$ which will be useful in Section~\ref{sect:cptembedd} when establishing compactness of certain embeddings of the spaces $M^\alpha_{p,q}$ and $N^\alpha_{p,q}$. 

We begin by introducing the space of measurable functions: Let $(X,\mu)$ be a measure space such that $\mu(X)<\infty$. We denote by $L^0(X,\mu)$ the space of all (equivalence classes of) measurable functions on $X$ which are finite almost everywhere. The space $L^0(X,\mu)$ is a complete metric space with respect to the metric
\[
d_{L^0}(f,g):=\int_X\frac{|f(x)-g(x)|}{1+|f(x)-g(x)|}\,d\mu(x).
\]
It is well known that convergence in this metric is equivalent to convergence in measure. Moreover, we clearly have $L^p(X,\mu)\subseteq L^0(X,\mu)$ for all $p\in(0,\infty)$.

The first result of this section is a Lebesgue--Vitali compactness-type theorem from \cite[Theorem~1]{Krotov} (see also \cite[Theorem~2.1]{BG} for version of this statement in the variable exponent case).

\begin{tw}\label{compLp}
Let $(X,\mu)$ be a measure space such that $\mu(X)<\infty$ and $0<p<\infty$. Then, a subset $\mathcal{F}$ of $L^p(X,\mu)$ is totally bounded if and only if the following conditions are satisfied:
\begin{enumerate}[i)]
\item $\mathcal{F}$ is totally bounded in $L^{0}(X,\mu)$, 
\item $\mathcal{F}$ is $p$-equi-integrable, i.e.,
\begin{equation*}
\forall_{\varepsilon >0} \exists_{\delta>0} \forall_{A\subseteq X}\,\,\,\,\mu(A) < \delta \implies \sup_{f\in \mathcal{F}} \int_{A} |f(x)|^p \, d\mu(x) < \varepsilon.
\end{equation*}
\end{enumerate}
\end{tw}

The following Fr\'echet-type criterion  provides a means of verifying condition $i)$ in Theorem~\ref{compLp}.

\begin{tw} \label{Frechet}
Let $(X,\mu)$ be a measure space such that $\mu(X)<\infty$. Then, a subset $\mathcal{F}$ of $L^0(X,\mu)$ is totally bounded if and only if for any $\varepsilon\in(0,\infty)$ there exist $\lambda\in(0,\infty)$ and a family of measurable, pairwise disjoint subsets $X_1,X_2,\dots,X_n$, of $X$   such that $X = \bigcup_{i=1}^{n} X_i $ and which have the property that for any $u\in \mathcal{F}$, there exists a measurable set $E(u) \subseteq X$, such that
\begin{enumerate}[i)]
\item $\mu(E(u)) < \varepsilon$,
\item $|u(x) - u(y)| < \varepsilon$ whenever $x,y \in X_{i}\setminus  E(u)$,
\item $|u(x)| \leq \lambda $ for $x\in X \setminus  E(u)$.
\end{enumerate}
\end{tw}
For a sake of completeness we include the proof.
\begin{proof}
Let $\mathcal{F}$ be a subset of $L^0(X,\mu)$.

``$\Leftarrow$" We fix $\varepsilon\in(0,\infty)$ and let $\delta := \varepsilon/(1+2\mu(X))$. By assumption, we can find $\lambda\in(0,\infty)$ and a family of measurable, pairwise disjoint subsets $X_1,X_2,\dots,X_n$, of $X$  such that $X = \bigcup_{i=1}^{n} X_i $
and which have the property that  for any $u\in \mathcal{F}$, there exists a measurable set $E(u) \subseteq X$, such that
$\mu(E(u)) < \delta$, $|u(x)| \leq \lambda $ for $x\in X \setminus  E(u)$, and for $x,y\in X_{i} \setminus E(u)$ we have $|u(x) - u(y)| < \delta$.

Let $M$ be the least integer satisfying $M\geq \lambda/\delta$ and note that $M\in\mathbb{N}$. Consider the family of functions
\begin{align*}
S:=\bigg\{\sum_{i=1}^n \delta m_{i} \chi_{X_{i}}:  m_{i} \in \{-M,\dots,M\}  \text{ for } i \in \{1,2,\dots,n \}  \bigg\}\subseteq L^0(X,\mu).
\end{align*}
We claim that $S$ is an $\varepsilon$-net of $\mathcal{F}$. To see this, fix $u\in \mathcal{F}$, and for integers satisfying $-M\leq m \leq M$, let
\begin{align*}
A_{m} :=  u^{-1}\big([\delta m,\delta (m+1))\big).
\end{align*}
Since $\lambda \leq \delta M$, we have 
\begin{align} \label{zawieranie}
X\setminus E(u) \subseteq \bigcup_{m=-M}^{M} A_m.
\end{align}
Let us consider a function of the form
\[
\tilde{u}:=\sum_{i=1}^n \delta m_{i} \chi_{X_{i}},
\]
where for $i \in \{1, 2,\dots,n\}$ we choose $m_i\in \{-M,\dots,M\}$ as follows: If $X_{i} \setminus E(u) = \emptyset$, then we set $m_{i}=0$. If $X_{i} \setminus E(u) \neq \emptyset$, then we take any $x_i\in X_{i} \setminus E(u)$ and let $m_i$ be any integer in $\{-M,\dots,M\}$ satisfying $x_i \in A_{m_i}$. Note that such a choice of $m_i$ is guaranteed by \eqref{zawieranie}. Then $\tilde{u} \in S$, and  for  $y\in X\setminus E(u)$ we have
\[
|u(y) - \tilde{u}(y)| < 2\delta.
\]
Indeed, since $X = \bigcup_{i=1}^{n} X_i $ and $X_i$ are pairwise disjoint, there exists an unique $i\in \{1, 2,\dots,n\}$ such $y\in X_{i} \setminus E(u)$. Therefore, since $x_i \in A_{m_i}$ implies that $0\leq u(x_i)-\delta m_i < \delta$, we can estimate
\begin{align*}
|u(y)-\tilde{u}(y)|= |u(y) - \delta m_{i}| \leq |u(y) - u(x_i)| + |u(x_i) - \delta m_{i}| < 2\delta,
\end{align*}
where we have used the fact that $x_i,y\in X_i\setminus E(u)$ in obtaining the last inequality.
 Finally,  since 
\begin{align*}
d_{L^0}(u,\tilde{u}) &= \int_{E(u)} \dfrac{|u(x)-\tilde{u}(x)|}{1+ |u(x)-\tilde{u}(x)|}\, d\mu(x) + \int_{X\setminus E(u)} \dfrac{|u(x)-\tilde{u}(x)|}{1+ |u(x)-\tilde{u}(x)|}\, d\mu(x) \\
&< \delta + 2\delta\mu(X) =\varepsilon,
\end{align*}
we have that $S$ is $\varepsilon$-net of $\mathcal{F}$, as wanted. Since $\varepsilon\in(0,\infty)$ was arbitrary, it follows that $\mathcal{F}$ is totally bounded in $L^0(X,\mu)$.

$``\Rightarrow"$ Suppose that $\mathcal{F}$ is totally bounded in $L^0(X,\mu)$ and fix $\varepsilon\in(0,\infty)$. Let $\delta := \varepsilon/3 $ and suppose that $\{u_{i}\}_{i=1}^{N} \subseteq L^0(X,\mu)$ is a ${2\delta^2}/{(1+\delta)}$-net of $\mathcal{F}$. Since functions in $L^0(X,\mu)$ are finite a.e. and $\mu(X)< \infty$, we can find $\lambda\in(\delta,\infty)$ such that
\begin{align*}
\mu\Big(\big\{x\in X: \max_{i\in\{1,\dots,N\}}|u_i(x)| > \lambda - \delta \big\}\Big) \leq \delta.
\end{align*} 
Let $Y := \{x\in X: \max_{i\in\{1,\dots,N\}}|u_i(x)| \leq  \lambda - \delta \}$ and let $M$ be the least integer satisfying $M\geq (\lambda-\delta)/\delta$. Note that $M\in\mathbb{N}$ and $\mu(X\setminus Y)\leq \delta$. For $j\in \{1,2,\dots,N \}$ and $m\in \{-M, \dots, M \}$ we consider the following sets
\begin{align*}
A^{j}_{m} := \{x\in Y: u_{j}(x) \in [\delta m,\delta (m+1)) \}.
\end{align*}
For any $m_{1},\dots, m_{N} \in \{-M, \dots, M \}$ we define
\begin{align*}
X_{m_{1},m_{2},\dots m_{N}} := \bigcap_{j=1}^{N} A^{j}_{m_j}.
\end{align*}
The family $\{X_{m_{1},m_{2},\dots m_{N}}\}_{m_1,\dots,m_N \in \{-M,\dots,M\}}$ consists of pairwise disjoint sets. 
Since $\lambda-\delta \leq \delta M$, we have 
\[
	Y=\bigcup_{m_{1}= - M }^{ M } \dots \bigcup_{m_{N}= - M }^{ M } X_{m_{1},m_{2},\dots m_{N}}.
\]
Hence,
\begin{align*}
X = (X\setminus Y) \cup  \bigcup_{m_{1}= - M }^{ M } \dots \bigcup_{m_{N}= - M }^{ M } X_{m_{1},m_{2},\dots m_{N}}.
\end{align*}
Let $u\in \mathcal{F}$ and $j\in \{1,2,\dots,N \}$ be such that $d_{L^0}(u,u_{j}) < 2\delta^{2}/(1+\delta)$ and let
$E(u) := (X\setminus Y) \cup Z$, where $Z:=\big\{x\in X: |u(x) - u_{j}(x)| > \delta \big\}$. Observe that,
\[
\mu(Z)\leq\frac{1+\delta}{\delta}\int_Z\frac{|u(x)-u_j(x)|}{1+|u(x)-u_j(x)|}\,d\mu(x)
\leq\frac{1+\delta}{\delta}\,d_{L^0}(u,u_{j})<2\delta.
\] 
Therefore,
\begin{align*}
\mu(E(u))\leq\mu(X\setminus Y)+\mu(Z) < 3\delta = \varepsilon,
\end{align*}
and for any $x,y\in X_{m_{1},m_{2},\dots m_{N}} \setminus E(u)$ we have
\begin{align*}
|u(x) - u(y)| \leq |u(x) - u_{j}(x)| +  |u_{j}(x) - u_{j}(y)| + |u_{j}(y) - u(y)| < 3\delta = \varepsilon.
\end{align*}
Moreover, for $x\in X\setminus  E(u)$ we have
\begin{align*}
|u(x)| \leq |u(x) -u_{j}(x)| + |u_{j}(x)| \leq \lambda.
\end{align*}
This completes the proof of the theorem.
\end{proof}

The last result in this section is a version of Theorem~\ref{Frechet} for separable quasi-metric spaces.
\begin{tw}\label{compL0}
Let $(X,d)$ be a separable quasi-metric space and $\mu$ be a Borel measure such that $\mu(X)<\infty$. Then, a subset $\mathcal{F}$ of $L^0(X,\mu)$ is totally bounded if for any $\varepsilon\in(0,\infty)$ there exist $\delta\in(0,\infty)$ and  $\lambda\in(0,\infty)$ such that for any $u \in \mathcal{F}$ there is a set $E(u)\subseteq X$, satisfying the following conditions:
\begin{enumerate}[i)]
\item $\mu(E(u)) < \varepsilon$,
\item $|u(x) - u(y)| < \varepsilon$ for any $x,y \in X\setminus  E(u)$, $d(x,y)<\delta$,
\item $|u(x)| \leq \lambda $ for $x\in X \setminus  E(u)$.
\end{enumerate}
On the  other hand, if $\mu$ is a Borel regular measure and every open set $U\subseteq X$ can be written as $U=F \cup Z$, where $F$ is an $F_{\sigma}$-set and  $Z$ is a  null set, then the above condition is also necessary.
\end{tw}
Let us mention that the sufficiency part of the above theorem has been proven in the case of bounded metric spaces equipped with a doubling measure \cite[Theorem~2]{Krotov} and in the case of totally bounded metric-measure spaces \cite[Theorem~1.3]{BG}.

\begin{proof}
Let $\mathcal{F}$ be a subset of $L^0(X,\mu)$. 
\smallskip

{\tt Sufficiency:} Let us  fix $\varepsilon\in(0,\infty)$. Then, there exists $\delta,\lambda\in(0,\infty)$ which satisfy the assumptions of the theorem with $\varepsilon/2$. As such, we have that for any $u\in \mathcal{F}$ there is a set $\hat{E}(u)\subseteq X$, such that $\mu(\hat{E}(u)) < \varepsilon/2$, $|u(x)| \leq \lambda$ for $x\in X\setminus \hat{E}(u)$, and  $|u(x) - u(y)| < \varepsilon/2$, for $x,y \in  X\setminus \hat{E}(u)$ with $d(x,y)<\delta$. 

Since we are assuming that $(X,d)$ is separable, we can appeal to Lemma~\ref{Ulam} to write $X = G_{\varepsilon} \cup B_{\varepsilon}$, where $B_\varepsilon$ is a measurable set satisfying $\mu(B_{\varepsilon}) < \varepsilon/2$ and $G_{\varepsilon}$ is totally bounded.  Let $\rho$ be the regularized quasi-metric given by Lemma~\ref{DST1} and let $\{x_{i} \}_{i=1}^{n}$, be $\delta/C_d^3$-net of $G_{\varepsilon}$ with respect to this new quasi-metric $\rho$, where $C_d\in[1,\infty)$ is as in \eqref{C-d}. Now, we define
\begin{align*}
X_{1} &:= B_\rho\big(x_1,\delta/C_d^3\big), \\
X_{i} &:= B_\rho\big(x_i,\delta/C_d^3\big) \setminus \bigcup_{j=1}^{i-1} B_\rho\big(x_j,\delta/C_d^3\big)\,\, \text{ for }\, i\in \{2,3,\dots,n \},\\
X_{n+1} &:= B_{\varepsilon} \setminus \bigcup_{j=1}^{n} B_\rho\big(x_j,\delta/C_d^3\big).
\end{align*}
By Lemma~\ref{DST1}, every ball with respect to $\rho$ is open (in the quasi-metric topology $\tau_d$) and so it is easy to see from this construction that the above sets are measurable, pairwise disjoint, and $X = \bigcup_{i=1}^{n+1} X_i $.

We shall check that assumptions of Theorem~\ref{Frechet} are satisfied. To this end, suppose $u\in \mathcal{F}$ and let $\hat{E}(u)\subseteq X$ be the set given as above. We define $E(u) := \hat{E}(u) \cup X_{n+1}$. Then we easily conclude that $\mu(E(u))<\varepsilon$ and  $|u(x)|\leq \lambda$ on $X\setminus E(u)$. Next, fix $i=1,2,\dots,n$ and suppose that $x,y\in X_{i}\setminus E(u)$. By Lemma~\ref{DST1}, we can estimate (keeping in mind that $\rho$ is symmetric and $C_\rho\leq C_d$)
\[
d(x,y)\leq C_d^2\rho(x,y)\leq C_d^2C_\rho\max\{\rho(x,x_i),\rho(x_i,y)\}
<C_d^3\max\Big\{\delta/C_d^3,\delta/C_d^3\Big\}=\delta.
\]  
Thus, $d(x,y)<\delta$ and so we have $|u(x) - u(y)| <\varepsilon$. Finally, this also holds for $i=n+1$, since $X_{n+1} \setminus E(u)$ is an empty set. Therefore, the assumptions of Theorem~\ref{Frechet} are satisfied and it follows that  $\mathcal{F}$ is totally bounded in $L^0(X,\mu)$.
\smallskip

{\tt Necessity:} For this, we will employ an appropriate version of the Lusin Theorem (Theorem~\ref{lusin}). To proceed, suppose that $\mathcal{F}$ is totally bounded in $L^{0}(X,\mu)$. For a fixed $\varepsilon\in(0,\infty)$, let $\{u_i\}_{i=1}^{N}$ be a finite {$\varepsilon^2/(12+3\varepsilon)$}-net of $\mathcal{F}$.
By the Lusin Theorem (Theorem~\ref{lusin}), for each $i=1,\dots,N$, there exists a closed set $E_{i} \subseteq X$ such that {$\mu(X\setminus  E_{i})<\varepsilon/(6N)$} and $u_{i}\vert_{E_{i}}$ is continuous. Let $E := \bigcap_{i=1}^{N} E_i$, then  $\mu(X\setminus E) < { \varepsilon/6}$.  For $k\in \mathbb{N}$, we define the set

\begin{align*}
A_{k} := \big\{x\in E: \forall z\in D  \mbox{ if } \rho(x,z)<1/k \implies \forall 1\leq i\leq N, \,\, \, |u_{i}(x) - u_{i}(z)| < { \varepsilon/4} \big\},
\end{align*}
where $\rho$ is the regularized quasi-metric given by Lemma~\ref{DST1} and   $D$ is a countable  dense subset of $E$.\footnote{Since $A_{k} = \bigcap_{y\in D} \{x\in E: \chi_{B_{\rho}(y,1/k)}(x) \leq \min_{i=1,\dots,N} \chi_{u^{-1}_{i}((u_i(y) - \varepsilon/4,u_i(y) + \varepsilon/4))} (x)\}$, we see that $A_k$ is a measurable set.} Since each $u_{i}$ is continuous on $E$ and $A_k \subseteq A_{k+1}$, we can find a large enough $K\in\mathbb{N}$ so that $\mu(E\setminus  A_{K})< {\varepsilon/6}$. Let $\delta := 1/(C_{d}\widetilde{C}_dK)$, where $C_d,\widetilde{C}_d\in[1,\infty)$ are as in \eqref{C-d}-\eqref{C-d-tilde}, and let us take $\lambda\in(0,\infty)$ such that the measure of the set
\[
	F_{\varepsilon} := \Big\{x\in X: \max\limits_{i=1,\dots,N} |u_i(x)| > \lambda - \varepsilon/4\Big\}
\]
is at most $\varepsilon/3$. Note that such a choice of $\lambda$ is possible since functions in $L^0(X,\mu)$ are finite almost everywhere.
Now, let $u\in \mathcal{F}$. Then we can find $j\in \{1,2,\dots,N \}$ such that $d_{L^0}(u,u_j) < \varepsilon^2/(12+3\varepsilon)$. For this $j$, consider the set
\[
	E(u) := (X\setminus A_{K}) \cup F_{\varepsilon} \cup Y,
\]
where $Y:=\{x\in X: |u(x) - u_j(x)| > \varepsilon/4 \}$. 

We shall now show that the set $E(u)$ satisfies conditions $i)$-$iii)$ in the statement of this theorem. First, note that since the map $t\mapsto t/(1+t)$ is increasing, we have
\begin{align*}
\mu(Y)\leq \frac{4 + \varepsilon}{\varepsilon} \int_{Y} \frac{|u(x)-u_j(x)|}{1+|u(x)-u_j(x)|}\,d\mu(x) \leq \frac{4 + \varepsilon}{\varepsilon} d_{L^0}(u,u_{j}) < \varepsilon/3. 
\end{align*}
From this, and the construction of $E(u)$, we have that statements $i)$ and $iii)$ in this theorem are fulfilled. Now, we check the condition $ii)$. Let $x,y\in X \setminus E(u)$ satisfy $d(x,y)<\delta$. By density of $D$  we can find $z\in D$ such that $d(x,z) < \delta$ and therefore (keeping in mind the properties of $\rho$ as given by Lemma~\ref{DST1})
\[
\rho(y,z)\leq C_{\rho}\max\big\{\rho(y,x),\rho(x,z)\big\} \leq C_{\rho}\max\big\{\widetilde{C}_d d(x,y),\widetilde{C}_d d(x,z)\big\}
<C_{\rho} \widetilde{C}_d \delta\leq C_{d} \widetilde{C}_d \delta=1/K.
\]  
Granted this, it follows from the definition of $E(u)$ that
\begin{align*}
|u(x) - u(y)| \leq |u(x) - u_{j}(x)| + |u_{j}(x) - u_{j}(z)| + |u_{j}(z) - u_{j}(y)| + |u_{j}(y) - u(y)| < \varepsilon.
\end{align*}
This finishes the verification of condition $ii)$ and this completes the proof of the theorem.
\end{proof}

\section{Compact embeddings of $M^{\alpha}_{p,q}$ and $N^{\alpha}_{p,q}$ into $L^{\tilde{p}}$ spaces}
\label{sect:cptembedd}

Throughout this section, we will use the following observation which is an immediate consequence of Proposition~\ref{embs}: Let $(X, d, \mu)$ be a quasi-metric space equipped with a nonnegative Borel measure. Then, for any $\alpha, p\in(0,\infty)$ and $q\in(0,\infty]$
\begin{equation}
\label{Haj-incl}
M^{\alpha}_{p,q}(X,d,\mu)\hookrightarrow M^{\alpha, p}(X,d,\mu)\quad\mbox{and}\quad
N^{\alpha}_{p,q}(X,d,\mu)\hookrightarrow M^{\alpha/2, p}(X,d,\mu).
\end{equation}
Hence, in order to study compact embeddings of $M^{\alpha}_{p,q}$ and $N^{\alpha}_{p,q}$ it is enough to show compactness of embeddings for $M^{\alpha,p}$-spaces.

\subsection{Embeddings into $L^0$}
\label{subsect:0}

We have the following theorem about compact inclusion in $L^0$. 
\begin{tw} \label{emb_in_L0}
Let $(X,d,\mu)$ be a separable  quasi-metric space and $\nu$ be a  measure satisfying $\nu(X)<\infty$ and $\nu << \mu$. Then, for any $\alpha,p\in(0,\infty)$ and $q\in(0,\infty]$ the embeddings
\[
M^{\alpha}_{p,q}(X,d,\mu)\hookrightarrow L^0(X,\nu)\quad\mbox{and}\quad N^{\alpha}_{p,q}(X,d,\mu)\hookrightarrow L^0(X,\nu)
\] 
are compact.
\end{tw}
Let us remark that in the case of quasi-metric-measure space separability follows from Proposition~\ref{sepequiv} and the fact that, by definition, all balls in quasi-metric-measure spaces are assumed to have positive and finite measure. 
\begin{proof}
Let $\alpha, p\in(0,\infty)$. We will show that the embedding $M^{\alpha,p}(X,d,\mu)\hookrightarrow L^0(X,\nu)$ is compact. Let $\mathcal{F} \subseteq M^{\alpha,p}(X,d,\mu)$ be a nonempty bounded set and let $M := \sup_{u\in \mathcal{F}} \|u\|_{M^{\alpha,p}(X,d,\mu)}$. Note that we can assume $M>0$.  It is enough to check that conditions in Theorem~\ref{compL0} hold (with $\nu$ in place of $\mu$). Now, let us fix $\varepsilon \in(0,\infty)$. Note that since $\nu(X)<\infty$ and $\nu << \mu$, we can find $\eta\in(0,\infty)$ such that, for any measurable set $E\subseteq X$, we have
\begin{align*}
\mu(E)<\eta \implies \nu(E)<\varepsilon.
\end{align*}

Let us define
\[
	\lambda :=  M2^{2+1/p}/\eta^{1/p}\quad\mbox{and}\quad\delta :=\big(\varepsilon/(2 \lambda)\big)^{1/\alpha},
\]
and fix $u \in \mathcal{F}$. Then there exist a measurable set $A_u \subseteq X$, with $\mu(A_u)=0$, and a Haj\l{}asz--gradient $g\in\mathcal{D}^\alpha_d(u)$ satisfying $\|g\|_{L^p(X, \mu)}\leq \|u\|_{{M}^{\alpha,p}(X,d,\mu)}$, and
\[
	|u(x)-u(y)|\leq [d(x,y)]^\alpha\big(g(x)+g(y)\big)\quad\mbox{for all}\,\, x, y \in X \setminus A_u.
\]
Next, we define the following set 
\begin{equation*}
E(u) := \big\{x \in X : \max\big\{|u(x)|,g(x)\big\} > \lambda \big\}\cup A_u.
\end{equation*}
Note that $\nu(A_u)=0$. Given our choice of $\lambda$, an application of Chebyshev's inequality implies that $\mu(E(u)) < \eta$ and so, $\nu(E(u))<\varepsilon$.  Moreover, condition $iii)$ in Theorem~\ref{compL0} is obviously satisfied. Therefore, we only need to check if $ii)$ holds.
Let $x,y \in X\setminus  E(u)$ such that $d(x,y) <\delta$. Then,
\begin{align*}
|u(x)-u(y)|\leq [d(x,y)]^\alpha\big(g(x)+g(y)\big) \leq [d(x,y)]^{\alpha} 2 \lambda < \varepsilon,
\end{align*}
and we get that condition $ii)$ in Theorem~\ref{compL0} is fulfilled. Thus, the embedding $M^{\alpha,p}(X,d,\mu)\hookrightarrow L^0(X,\nu)$ is compact and, given that $\alpha, p\in(0,\infty)$ were arbitrary, the conclusion of this theorem now follows from the observation in \eqref{Haj-incl}. 
\end{proof}

\subsection{Embeddings into $L^p$}
\label{subsect:p}
In this subsection we investigate necessary and sufficient conditions guaranteeing that  $M^{\alpha}_{p,q}$ and $N^{\alpha}_{p,q}$ compactly embed  into $L^{p}$. 

\subsubsection{Sufficient Conditions}
\label{subsubsect:suff}
The following theorem is the first main result of this subsection and it generalizes and extends \cite[Theorem~4.1]{GS}, \cite[Proposition~1.2]{BK}, and \cite[Theorem~2]{K} by relaxing the assumptions on the underlying space and measure (see also Corollary~\ref{remarkoint} in this regard).
\begin{tw} \label{drugie}
Let $(X,d,\mu)$ be a quasi-metric-measure space. Let $\nu$ be a measure such that $\nu\leq C \mu$ for some constant $C\in(0,\infty)$. Assume $\rho$ is a symmetric quasi-metric  on $X$ which is equivalent to $d$ and continuous on $X \times X$.\footnote{If $d$ is a metric then $\rho$ could be taken to be $d$.} Suppose that there exist $r_0\in(0,\infty)$ and $x_0 \in X$ such that for all $r\in(0,r_0]$, the following conditions are satisfied:
\smallskip
\begin{enumerate}[i)]
\item the family $\left\{\displaystyle\frac{\chi_{B_\rho(y,r)}(\cdot)}{\mu(B_\rho(\cdot,r))} \right\}_{y\in X}$ is equi-integrable with respect to measure $\nu$, and
\smallskip
\item $\sup\limits_{y \in X\setminus B_\rho(x_0,R)}\displaystyle \int_{B_\rho(y,r)} \frac{1}{\mu(B_\rho(x,r))}\, d\nu(x) \to 0$ with $R\to \infty$.
\end{enumerate}
Then for any $\alpha, p\in(0,\infty)$ and $q\in(0,\infty]$, the embeddings 
\[
	M^{\alpha}_{p,q}(X,d,\mu)\hookrightarrow L^p(X,\nu)\quad\mbox{and}\quad	N^{\alpha}_{p,q}(X,d,\mu)\hookrightarrow L^p(X,\nu)
\]
are compact.
\end{tw}

\begin{rem}
If $\nu(X)<\infty$, then  condition ii) follows from i).
\end{rem}
\begin{rem}
If there exist $c\in(1,\infty)$ and $\delta_0\in(0,\infty]$ such that the measure $\mu$ is $(c,\delta)$-doubling for every finite $\delta\in(0,\delta_0]$, then we can relax the assumptions that $\rho$ is a symmetric and continuous on $X \times X$ to only assuming that for each fixed $r\in(0,r_0]$, the mapping $x\mapsto\mu(B_\rho(x,r))$ is well defined and measurable. Indeed, in this case $\rho$ can be replaced by the equivalent regularized quasi-metric given by Lemma~\ref{DST1}.
\end{rem}
\begin{rem}
A typical example of measures $\nu$ and $\mu$ satisfying the condition ``$\nu\leq C\mu$" is as follows: Given a measure space $(X,\mathfrak{M},\mu)$ and a set $\Omega\in\mathfrak{M}$, the measure $\nu:=\mu\!\!\measurerestr\!\Omega\leq\mu$, where $\mu\!\!\measurerestr\!\Omega$ denotes the measure which is given by $(\mu\!\!\measurerestr\!\Omega)(A):=\mu(A\cap\Omega)$ for all $A\in\mathfrak{M}$. One could also consider the measure $\nu$ which is given by $d\nu:=f\,d\mu$, where $f\in L^\infty(X,\mu)$.
\end{rem}
\begin{proof}[Proof of Theorem~\ref{drugie}]
First of all we shall prove the two lemmas.
%\begin{lem}
%Let $(X,d,\mu)$ be a quasi metric measure-space such that $\mu(X)<\infty$. Let $\nu$ be any Borel measure such  that $\nu\leq C \mu$ for some constant $C>0$ (here we don't assume $\nu(B(x,r))>0$ ). Moreover, assume that for each $r>0$ we can find function $\psi_{r}:[0,\infty)\to[0,\infty)$ such that
%
%\begin{enumerate}[i)]
%\item $\psi_{r}(0) = 0$
%\item $\psi_{r}$ is non-decreasing
%\item $\lim\limits_{t\to\infty}\psi_{r}(t)/t = \infty$
%
%\item
%\begin{align*}
% \sup\limits_{y\in X} \int_{B(y,r)} \psi_{r}\bigg(\frac{1}{\mu(B(x,r))}\bigg) d\nu(x) < \infty
%\end{align*}
%
%\end{enumerate}

%Now we fix $r^{\alpha p} = \varepsilon/(2 \cdot 4^p M^p )$. We  rewrite condition iv) in the following way:
%\begin{align*}
%\sup\limits_{y\in X} \int_{X} \psi_{r}\bigg(\frac{\chi_{B(y,r)}(x)}{\mu(B(x,r))}\bigg) d\nu(x) < \infty
%\end{align*}
% From de la Vallée Poussin theorem it follows that the family $\bigg\{\frac{\chi_{B(y,r)}(\cdot)}{\mu(B(\cdot,r))}\bigg\}_{y\in X}$ is equi-integrable. We can find a $\delta>0$ s.t. whenever $\mu(D)<\delta$ then for any $y\in X$ we have
% 
% \begin{align*}
% \int_{D} \frac{\chi_{B(y,r)}(x)}{\mu(B(x,r))}d\nu(x)  <  \varepsilon/(2\cdot 4^p M^p).
% \end{align*}

\begin{lem}
\label{ew-4}
Let $(X,d,\mu)$ be a quasi-metric-measure space and let $\nu$ be a measure such that $\nu<<\mu$. Assume $\rho$ is a symmetric quasi-metric  on $X$ which is equivalent to $d$ and continuous on $X \times X$. Then, there exists $\kappa\in(0,\infty)$ (depending only on $d$ and $\rho$)  such that for any $\alpha,p, r \in(0,\infty)$, $u \in M^{\alpha,p}(X,d,\mu)$, $g\in \mathcal{D}^{\alpha}_d(u)$, and any measurable set $D\subseteq X$, one has
\begin{align} \label{key_ineq}
\int_{D}|u(x)|^p \, d\nu(x) &\leq  4^p\kappa^{\alpha p} r^{\alpha p } \|g\|_{L^p(X,\nu)}^p + \int_{X}\big(4^p\kappa^{\alpha p} r^{\alpha p} g^p(y) + 2^p |u(y)|^p\big)\int_{D} \frac{\chi_{B_\rho(y,r)}(x)}{\mu(B_\rho(x,r))}\,  d\nu(x) d\mu(y).  
\end{align}

\end{lem}

\begin{proof}
Let $\alpha,p, r\in(0,\infty)$, $u \in M^{\alpha,p}(X,d,\mu)$, $g\in \mathcal{D}^{\alpha}_d(u)$, and $D\subseteq X$ be a measurable set. Since $\rho$ is equivalent to $d$, there $\kappa\in[1,\infty)$ such that, for all $x,y\in X$,
$$
\kappa^{-1}\rho(x,y)\leq d(x,y)\leq \kappa\rho(x,y).
$$
Therefore, for almost every $x,y\in X$, we have
\begin{align*}\label{new-ineq}
|u(x)|^p &\leq 2^p\Big(|u(x) - u(y)|^p + |u(y)|^p\Big)\\
&\leq 2^p \Big([d(x,y)]^{\alpha p}\big(g(x)+g(y)\big)^p+ |u(y)|^p\Big)\\
&
\leq 2^p \Big(\kappa^{\alpha p}[\rho(x,y)]^{\alpha p}\big(g(x)+g(y)\big)^p+ |u(y)|^p\Big).
\end{align*}
Averaging the above inequality with respect to $y \in B_\rho(x,r)$, we get
\begin{align*}
|u(x)|^p %&\leq 2^p\bigg(\dashint_{B_\rho(x,r)} |u(x)-u(y)|^p \, d\mu(y) + \dashint_{B_\rho(x,r)} |u(y)|^p \, d\mu(y) \bigg) \\
&\leq 2^p\bigg( \kappa^{\alpha p}2^p \dashint_{B_\rho(x,r)} r^{\alpha p }\big(g^p(x)+g^p(y)\big)\, d\mu(y) + \dashint_{B_\rho(x,r)} |u(y)|^p \, d\mu(y) \bigg) \\
&= 4^p\kappa^{\alpha p} r^{\alpha p } g^p(x) + \dashint_{B_\rho(x,r)}\big( 4^p\kappa^{\alpha p} r^{\alpha p} g^p(y) + 2^p |u(y)|^p \big)\,d\mu(y).
\end{align*}
We integrate the above inequality over $x\in D$ with respect to measure $\nu$ and apply the Fubini Theorem to write
\begin{align*}
\int_{D}|u(x)|^p \, d\nu(x) &\leq  4^p\kappa^{\alpha p} r^{\alpha p } \|g\|_{L^p(X,\nu)}^p + \int_{D}\int_{X} \frac{4^p\kappa^{\alpha p} r^{\alpha p} g^p(y) + 2^p |u(y)|^p}{\mu(B_{\rho}(x,r))} \chi_{B_\rho(x,r)}(y)\,  d\mu(y) d\nu(x) \\
&\leq 4^p\kappa^{\alpha p} r^{\alpha p } \|g\|_{L^p(X,\nu)}^p  + \int_{X} (4^p\kappa^{\alpha p} r^{\alpha p} g^p(y) + 2^p |u(y)|^p) \int_{D} \frac{\chi_{B_\rho(y,r)}(x)}{\mu(B_\rho(x,r))}\,d\nu(x)   d\mu(y), 
\end{align*}
and this finishes the proof of Lemma~\ref{ew-4}.
\end{proof}

\begin{rem}
The symmetry and continuity of the quasi-metric $\rho$ is used in the proof of Lemma~\ref{ew-4} to ensure that our use of the Fubini Theorem is valid and that $\chi_{B_\rho(x,r)}(y)=\chi_{B_\rho(y,r)}(x)$. This is the only place where we use these assumptions.
\end{rem}

\begin{lem}\label{cor_i)}
Let $(X,d,\mu)$ be a quasi-metric-measure space. Let $\nu$ be a measure such that $\nu\leq C \mu$ for some constant $C\in(0,\infty)$. Assume $\rho$ is a symmetric quasi-metric  on $X$ which is equivalent to $d$ and continuous on $X \times X$. Suppose that there exist $r_0\in(0,\infty)$ such that for all $r\in(0,r_0]$, the  family $\Big\{ {\chi_{B_\rho(y,r)}(\cdot)}/{\mu(B_\rho(\cdot,r))} \Big\}_{y\in X}$ is equi-integrable with respect to  $\nu$. Then, for any $\alpha,p\in(0,\infty)$, bounded subsets of $M^{\alpha,p}(X,d,\mu)$ are $p$-equi-integrable with respect to $\nu$.
\end{lem}

\begin{proof}
We fix $\varepsilon,\alpha,p\in(0,\infty)$ and let $\mathcal{F}\subseteq M^{\alpha,p}(X,d,\mu)$ be a nonempty bounded set. By Lemma~\ref{equivspaces}, $\mathcal{F}$ is also a bounded subset of $M^{\alpha,p}(X,\rho,\mu)$. We define $M := \sup_{u\in \mathcal{F}} \|u\|_{M^{\alpha,p}(X,\rho,\mu)}$ and let $u\in \mathcal{F}$. 
Then we can choose $g\in \mathcal{D}^{\alpha}_\rho(u)$ such that $\|g\|_{L^p(X,\mu)} \leq \|u\|_{M^{\alpha,p}(X,\rho,\mu)}$. Let 
\[
v:= (\kappa^{\alpha p}4^p r^{\alpha p}_{0} g^p + 2^p |u|^p)^{1/p}.
\]
Since $\|g\|_{L^p(X,\nu)}^p \leq C \|g\|_{L^p(X,\mu)}^p$, where $C$ is as in the statement of the lemma, \eqref{key_ineq} implies
\begin{align}\label{int_D}
\int_{D}|u(x)|^p \, d\nu(x) &\leq C 4^p \kappa^{\alpha p} r^{\alpha p } M^p + \int_{X}|v(y)|^p\int_{D} \frac{\chi_{B_\rho(y,r)}(x)}{\mu(B_\rho(x,r))} \, d\nu(x) d\mu(y)  
\end{align}
for any measurable set $D\subseteq X$ and $r\leq r_0$.

Let us observe that
\begin{equation}\label{we}
\begin{split}
\int_X |v|^p\, d\mu & \leq  4^p \max\big\{1,\kappa^{\alpha p} r_0^{\alpha p}\big\}\left[\|g\|^p_{L^p(X,\mu)} +\|u\|^p_{L^p(X,\mu)}\right]\\
& \leq  2 (4M)^p \max\big\{1,\kappa^{\alpha p} r_0^{\alpha p}\big\}.
\end{split}
\end{equation}
We choose $r\in(0,r_0]$ such that 
\begin{align*}
C 4^p \kappa^{\alpha p} r^{\alpha p } M^p < \varepsilon/2.
\end{align*}
Note that we can assume $M>0$.  Then, under the current assumptions, there is $\delta\in(0,\infty)$, such that if $D\subseteq X$ is measurable set satisfying $\nu(D)<\delta$, then
\begin{align*}
\int_{D} \frac{\chi_{B_\rho(y,r)}(x)}{\mu(B_\rho(x,r))}\, d\nu(x) < \frac{\varepsilon}{4(4 M)^p\max\big\{1,\kappa^{\alpha p} r_0^{\alpha p}\big\}}.
\end{align*}
Let $D\subseteq X$ be any measurable set such that $\nu(D)<\delta$. Then \eqref{int_D} with \eqref{we} yield
\begin{align*}
 \int_{D}|u(x)|^p \, d\nu(x) &< \varepsilon/2+ \int_{X}|v(y)|^p\int_{D} \frac{\chi_{B_\rho(y,r)}(x)}{\mu(B_\rho(x,r))}  \,d\nu(x)  d\mu(y) <\varepsilon.
\end{align*}
This completes the proof of Lemma~\ref{cor_i)}.
\end{proof}
Now, we can prove Theorem~\ref{drugie}. Fix $\alpha,p\in(0,\infty)$. We will show that the embedding $M^{\alpha,p}(X,d,\mu)\hookrightarrow L^p(X,\nu)$ is compact. To this end, let $\mathcal{F}\subseteq M^{\alpha,p}(X,d,\mu)$ be a nonempty bounded set. By Lemma~\ref{equivspaces}, $\mathcal{F}$ is also a bounded subset of $M^{\alpha,p}(X,\rho,\mu)$. Define $M := \sup_{u\in \mathcal{F}} \|u\|_{M^{\alpha,p}(X,\rho,\mu)}$ and note that we can assume $M>0$. For $\varepsilon \in(0,\infty)$ we choose $r\in(0,r_0]$ such that 
\begin{align}\label{we-2}
C 4^p \kappa^{\alpha p} r^{\alpha p } M^p < \varepsilon/4.
\end{align}
From condition $ii)$ in the statement of Theorem~\ref{drugie}, we can find $R\in(C_\rho r,\infty)$ such that 
\begin{align}\label{we-3}
\sup\limits_{y \in X\setminus B_\rho(x_0,R/C_\rho - r)} \int_{B_\rho(y,r)} \frac{1}{\mu(B_\rho(x,r))}\, d\nu(x) < \frac{\varepsilon}{8(4 M)^p\max\big\{1,\kappa^{\alpha p}r_0^{\alpha p}\big\}}.
\end{align}
For any $y \in B_\rho(x_0,R/C_\rho - r)$ we have $[X \setminus B_\rho(x_0,R)] \cap B_\rho(y,r) = \emptyset$. Therefore, if we apply \eqref{int_D} with $D := X\setminus B_\rho(x_0,R)$ and $u\in \mathcal{F}$ then by \eqref{we}, \eqref{we-2}, and \eqref{we-3}, we get
\begin{align*}
\int_{X\setminus B_\rho(x_0,R)}|u(x)|^p \, d\nu(x) &< \varepsilon/4 + \int_{X}|v(y)|^p\int_{ [X\setminus B_\rho(x_0,R)] \cap B_\rho(y,r)} \frac{1}{\mu(B_\rho(x,r))}\,  d\nu(x)d\mu(y) \\
&\leq  \varepsilon/4 + \int_{X}|v(y)|^p  \sup_{y\in X\setminus B_\rho(x_0,R/C_\rho - r)}\int_{B_\rho(y,r)} \frac{1}{\mu(B_\rho(x,r))}  \,d\nu(x) d\mu(y) \\
&<\varepsilon/2.
\end{align*}
Next, we define $\tilde{\nu} :=\nu\!\!\measurerestr\!B_\rho(x_0,R)$. Then $\tilde{\nu}$ is a finite Borel measure  satisfying $\tilde{\nu} \leq \nu \leq C \mu$.  
From assumption $i)$ in the statement of Theorem~\ref{drugie} and Lemma~\ref{cor_i)} the family $\mathcal{F}$ is $p$-equi-integrable with respect to $\nu$ and hence, with respect to $\tilde{\nu}$ as well. We apply Theorem~\ref{emb_in_L0} with $\tilde{\nu}$ in place of $\nu$ and $\mu$, and Theorem~\ref{compLp} with $\tilde{\nu}$ in place of $\mu$ (keeping in mind that $M^\alpha_{p,\infty}=M^{\alpha,p}$ by Proposition~\ref{embs}) to conclude that $\mathcal{F}$ is totally bounded in $L^p(X,\tilde{\nu})$. Thus, there is a finite set $\{w_{i}\}_{i=1}^{N}$ such that for any $u\in \mathcal{F}$ we can find $j\in \{1,\dots,N\}$ for which
\begin{align*}
\|u-w_{j} \|_{L^p(X,\tilde{\nu})}^p  < \varepsilon/2.
\end{align*}
Let $u_i := w_{i} \chi_{B_\rho(x_0,R)}$ for $i\in \{1,\dots,N \}$. Then, 
\begin{align*}
\|u - u_{j} \|_{L^p(X,\nu)}^p &= \int_{B_\rho(x_0,R)} |u(x)-u_j(x)|^p\, d\nu(x) + \int_{X\setminus B_\rho(x_0,R)}|u(x)|^p\, d\nu(x) \\
 &<\|u-w_{j} \|_{L^p(X,\tilde{\nu})}^p + \varepsilon/2 \\
 &<\varepsilon.
\end{align*}
Therefore,  $\mathcal{F}$ is totally bounded in $L^{p}(X,\nu)$ and so, the embedding $M^{\alpha,p}(X,d,\mu)\hookrightarrow L^p(X,\nu)$ is compact. Given that $\alpha, p\in(0,\infty)$ were arbitrary, the conclusion of this theorem now follows from the observation in \eqref{Haj-incl}. This completes the proof of Theorem~\ref{drugie}.
\end{proof}

\begin{corollary}\label{remarkoint}
Let $(X,d,\mu)$ be a quasi-metric-measure space. Let $\nu$ be a measure such that $\nu\leq C \mu$ for some constant $C\in(0,\infty)$. If $\mu$ is $\nu$-integrable, i.e., if there exists a quasi-metric $\varrho$ on $X$ which is equivalent to $d$ and satisfies the following condition: for each fixed $r\in(0,\infty)$, the mapping $x\mapsto\mu(B_\varrho(x,r))$ is well defined, measurable, and
\begin{align*}
\int_{X} \frac{1}{\mu(B_{\varrho}(x,r))} \, d\nu(x)<\infty,
\end{align*}
then, for any $\alpha,p\in(0,\infty)$ and $q\in(0,\infty]$, the embeddings 
\begin{align*}
	M^{\alpha}_{p,q}(X,d,\mu)\hookrightarrow L^p(X,\nu)\quad\mbox{and}\quad
	N^{\alpha}_{p,q}(X,d,\mu)\hookrightarrow L^p(X,\nu)
\end{align*}
are compact. In particular, these embeddings are compact if $(X,d)$ is totally bounded.
\end{corollary}
\begin{proof}
By Lemma~\ref{DST1}, there exists a symmetric quasi-metric $\rho$ on $X$ which is equivalent to $d$ and continuous on $X\times X$. Since $\rho\approx d$, we also have $\rho\approx \varrho$ and so we can find a constant $\kappa\in[1,\infty)$ such that $B_{\varrho}(x,r/\kappa)\subseteq B_{\rho}(x,r)$, for any $x\in X$ and $r\in(0,\infty)$. Therefore, for any measurable set $D\subseteq X$, $y\in X$ and $r\in(0,\infty)$, we have
\begin{align*}
\int_{D}\dfrac{\chi_{B_{\rho}(y,r)}(x)}{\mu(B_{\rho}(x,r))}\,  d\nu(x) \leq  \int_{D} \frac{1}{\mu(B_{\varrho}(x,r/\kappa))}\, d\nu(x).
\end{align*}
Hence, $\nu$-integrability of $\mu$ immediately implies that $\rho$ satisfies the conditions $i)$ and $ii)$ in the statement of Theorem~\ref{drugie} and so, the desired conclusion follows. Finally, note that if $(X,d)$ is totally bounded then $\mu$ is $\nu$-integrable by Proposition~\ref{tot} (see proof of Proposition~\ref{TB-int}), and this completes the proof the corollary.
\end{proof}
\begin{ex}
Let $(\mathbb{R}^n, |\cdot-\cdot|, \lambda_n)$, where $|\cdot-\cdot|$ is the standard Euclidean distance and $\lambda_n$ is the $n$-dimensional Lebesgue measure. Suppose that $S\subseteq\mathbb{R}^n$ is any finite measurable set and consider the measure $\nu :=\lambda_n\!\!\measurerestr\!S\leq\lambda_n$. Using the Ahlfors-regularity of $\lambda_n$, it is easy to see that  $\lambda_n$ is $\nu$-integrable and so, by Corollary~\ref{remarkoint}, we have that the embeddings
\begin{align*}
	M^{\alpha}_{p,q}(\mathbb{R}^n, |\cdot-\cdot|, \lambda_n)\hookrightarrow L^p(\mathbb{R}^n,\lambda_n\!\!\measurerestr\!S)\quad\mbox{and}\quad
	N^{\alpha}_{p,q}(\mathbb{R}^n, |\cdot-\cdot|, \lambda_n)\hookrightarrow L^p(\mathbb{R}^n,\lambda_n\!\!\measurerestr\!S)
\end{align*}
are compact for all  $\alpha,p\in(0,\infty)$ and $q\in(0,\infty]$. Consequently, these embeddings hold for the classical Besov and Triebel--Lizorkin spaces (and so, in particular, for the classical Sobolev and fractional Sobolev spaces) for certain ranges of parameters; see Subsection~\ref{subsect:fnctspaces} for a relationship between these function spaces.
\end{ex}

\subsubsection{Necessary Conditions}
\label{subsubsect:nec}
Now, we identify necessary conditions for compactness of the embeddings 	$M^{\alpha}_{p,q}(X,d,\mu)\hookrightarrow L^p(X,\mu)$ and $N^{\alpha}_{p,q}(X,d,\mu)\hookrightarrow L^p(X,\mu)$. We begin with the following lemma.
\begin{lem} \label{trzecie}
Let $(X,d,\mu)$ be a quasi-metric-measure space such that $\mu$ is $(C_d,\delta)$-doubling for some fixed $\delta\in(0,\infty)$, where $C_d\in[1,\infty)$ is as in \eqref{C-d}. Fix $\alpha,p\in(0,\infty)$ and $q\in(0,\infty]$, and  suppose that $\alpha\leq(\log_{2}C_d)^{-1}$, where $\alpha=(\log_{2}C_d)^{-1}$ is only permitted when $q=\infty$. If either of the embeddings
\begin{equation}
\label{ni4-24}
M^{\alpha}_{p,q}(X,d,\mu)\hookrightarrow L^p(X,\mu)\quad\mbox{or}\quad
N^{\alpha}_{p,q}(X,d,\mu)\hookrightarrow L^p(X,\mu)
\end{equation}
is compact, then every $\widetilde{C}_dC_d^2\delta$-separated set in $X$ (with respect to $d$) is finite. Here, $\widetilde{C}_d\in[1,\infty)$ is as in \eqref{C-d-tilde}.
\end{lem}
\begin{proof}
Suppose to the contrary that $\{x_{j}\}_{j\in\mathbb{N}}$ is an infinite $\widetilde{C}_dC_d^2\delta$-separated set in $X$ with respect to $d$. Then, for any $k,\ell\in\mathbb{N}$ with $k\neq \ell$,  $B_d(x_{k},\delta) \cap B_d(x_{\ell},C_d\delta) = \emptyset$. Indeed, if $y\in B_d(x_k,\delta)\cap B_d(x_\ell,C_d\delta)$ for some $k,\ell\in\mathbb{N}$ with $k\neq \ell$, then we have
\begin{align*}
\widetilde{C}_dC_d^2\delta\leq d(x_k,x_\ell)&\leq C_d\max\big\{d(x_k,y),\widetilde{C}_dd(x_\ell,y)\big\}
\\
&<C_d\max\big\{\delta,\widetilde{C}_dC_d\delta\big\}=\widetilde{C}_dC_d^2 \delta,
\end{align*}
which is a clear contradiction and so, the claim follows. 

Since ${\rm dist}_d (X\setminus B_d(x_j,C_d\delta),B_d(x_j,\delta) ) \geq \delta/\widetilde{C}_d>0$ for each $j\in\mathbb{N}$, we can appeal to Lemma~\ref{GVa2} to obtain a measurable function $\Phi_j: X\to[0,1]$ such that $\Phi_j\equiv 1$  on $B_d(x_j,\delta)$ and $\Phi_j\equiv 0$ on $X\setminus B_d(x_j,C_d\delta)$, and we define a sequence $\{f_j\}_{j\in\mathbb{N}}$ of functions by setting
\[
f_j(x) := \frac{1}{[\mu(B_d(x_j,\delta)]^{1/p}}\,\Phi_j(x)\quad\mbox{for all $j\in\mathbb{N}$ and $x\in X$.}
\]
Note that we have $f_j\in M^{\alpha}_{p,q}(X,d,\mu)\bigcap N^{\alpha}_{p,q}(X,d,\mu)$ for every $j\in\mathbb{N}$. In fact, since each $f_j$ vanishes pointwise outside of $B_d(x_j,C_d\delta)$, it follows from the $(C_d,\delta)$-doubling condition, the observation that ${\rm dist}_d (X\setminus B_d(x_j,C_d\delta),B_d(x_j,\delta) ) \geq \delta/\widetilde{C}_d$, and the  estimates in \eqref{gradest1}-\eqref{gradest2} that each $f_j$ is bounded in $M^s_{p,q}(X,d,\mu)$ and $N^s_{p,q}(X,d,\mu)$ independent of $j$.
However, $\{f_j\}_{j\in\mathbb{N}}$ has no convergent subsequence in $L^p(X,\mu)$ since for any $k,\ell\in\mathbb{N}$ with $k\neq \ell$ we have 
\begin{align*}
\|f_k - f_\ell \|_{L^p(X,\mu)}^p \geq  \int\limits_{B_d(x_{k},\delta)\cup B_d(x_{\ell},\delta)} |f_k(x) - f_\ell(x)|^p\, d\mu(x) = 2.
\end{align*}
Hence, both of the embeddings in \eqref{ni4-24} are not compact which is a contradiction, and this completes the proof of the lemma.
\end{proof}
Before stating our next main result, we introduce the following notional convention.

\begin{conv}\label{convention}
Given a quasi-metric space $(X,d)$ and fixed numbers $\alpha\in(0,\infty)$ and $q\in(0,\infty]$, we will understand by $\alpha\preceq_q{\rm ind}(X,d)$ that $\alpha\leq{\rm ind}(X,d)$ and that the value $\alpha={\rm ind}(X,d)$ is only permissible when $q=\infty$ and the supremum defining ${\rm ind}(X,d)$ in \eqref{index} is attained.
\end{conv}

\begin{tw} \label{czwarte}
Let $(X,d,\mu)$ be a quasi-metric-measure space and suppose there exists $\delta_0\in(0,\infty)$ such that
$\mu$ is $(C_d,\delta)$-doubling for every $\delta\in(0,\delta_0]$, where $C_d\in[1,\infty)$ is as in \eqref{C-d}.
Fix $\alpha,p\in(0,\infty)$ and $q\in(0,\infty]$, and suppose that $\alpha\preceq_q{\rm ind}(X,d)$. If either of the embeddings
\[
	M^{\alpha}_{p,q}(X,d,\mu)\hookrightarrow L^p(X,\mu)\quad\mbox{or}\quad
	N^{\alpha}_{p,q}(X,d,\mu)\hookrightarrow L^p(X,\mu)
\]
is compact, then $(X,d)$ is totally bounded.
\end{tw}
Let us remark that in view of Proposition~\ref{doub-diff}, the conclusion of Theorem~\ref{czwarte} holds whenever there exist $c\in(1,\infty)$ and $\delta_0\in(0,\infty)$ such that $\mu$ is $(c,\delta)$-doubling for every $\delta\in(0,\delta_0]$.
\begin{proof} 
First, we show that there exist $\delta_1\in(0,\infty)$ and a quasi-metric $\rho$ on $X$ which is equivalent to $d$ and has the following properties: All balls with respect to $\rho$ are measurable, $\alpha\leq(\log_{2}C_\rho)^{-1}$, where $\alpha=(\log_{2}C_\rho)^{-1}$ can only occur when $q=\infty$, and $\mu$ is $(C_\rho,\delta)$-doubling with respect to $\rho$ for all $\delta\in(0,\delta_1]$. To this end, if $C_d=1$ then we can simply take $\rho:=d$ and $\delta_1:=\delta_0$. Otherwise, given that $\alpha\preceq_q{\rm ind}(X,d)$, we can choose a quasi-metric $d'$ on $X$ which is equivalent to $d$ and satisfies $\alpha\leq(\log_{2}C_{d'})^{-1}$, where $\alpha=(\log_{2}C_{d'})^{-1}$ can only occur when $q=\infty$. Since balls with respect to $d'$ may not be measurable, we appeal to Lemma~\ref{DST1} to obtain a quasi-metric $\rho$ on $X$ which is equivalent to $d'$ (hence, is also equivalent to $d$) and has the following properties: All balls with respect to $\rho$ are measurable and $C_\rho\leq C_{d'}$. Note that this latter fact implies $\alpha\leq(\log_{2}C_{\rho})^{-1}$, where $\alpha=(\log_{2}C_{\rho})^{-1}$ can only occur when $q=\infty$. It remains to show that there exists $\delta_1\in(0,\infty)$ such that $\mu$ is $(C_\rho,\delta)$-doubling with respect to $\rho$ for all $\delta\in(0,\delta_1]$. Since $\rho$ is equivalent to $d$, there exists a constant $\kappa\in[1,\infty)$ such that $B_{d}(x,r/\kappa)\subseteq B_\rho(x,r)\subseteq B_{d}(x,\kappa r)$, for any $x\in X$ and $r\in(0,\infty)$. By Proposition~\ref{doub-diff}, we have that $\mu$ is $(\kappa^2C_\rho,\delta)$-doubling for every $\delta\in(0,C_d\delta_0/(\kappa^2C_\rho)]$. As such, if we let $\delta_1:=C_d\delta_0/(\kappa C_\rho)\in(0,\infty)$, then for every $\delta\in(0,\delta_1]$, we have
\begin{align*}
\mu(B_\rho(x,C_\rho\delta))&\leq\mu(B_d(x,\kappa C_\rho\delta))
=\mu(B_d(x,\kappa^2 C_\rho\delta/\kappa)) \\ 
&\leq\Delta_{\kappa^2 C_\rho}(\delta/\kappa) \mu(B_d(x,\delta/\kappa))
\leq \Delta_{\kappa^2 C_\rho}(\delta/\kappa) \mu(B_\rho(x,\delta)).
\end{align*}
Hence, $\mu$ is $(C_\rho,\delta)$-doubling with respect to $\rho$ for all $\delta\in(0,\delta_1]$. 

Moving on, to show $(X,d)$ is totally bounded, we fix $\varepsilon\in(0,\widetilde{C}_\rho C_\rho^2\delta_1]$ and let $A$ be any maximal $\varepsilon$-separated set with respect to $\rho$. Then $A$ is an $\varepsilon$-net for $X$ (with respect to $\rho$). Recall that Lemma~\ref{equivspaces} implies $M^{\alpha}_{p,q}(X,\rho,\mu)=M^{\alpha}_{p,q}(X,d,\mu)$ and $N^{\alpha}_{p,q}(X,\rho,\mu)=N^{\alpha}_{p,q}(X,d,\mu)$. 
Moreover, we note that $A$ is a $\widetilde{C}_\rho C_\rho^2\delta$-separated set (with respect to $\rho$) for some $\delta\in(0,\delta_1]$. Combining these observations with the fact that $\mu$ is $(C_\rho,\delta)$-doubling, it follows from Lemma~\ref{trzecie} (applied to the quasi-metric-measure space $(X,\rho,\mu)$) that the set $A$ is finite, and the claim now follows granted that $\rho$ is equivalent to $d$.
\end{proof}

As the following example illustrates, Theorem~\ref{czwarte} does not hold without some additional assumption on the quasi-metric-measure space $(X,d,\mu)$.

\begin{ex}{\bf (Infinite Comb)}\label{infinitecomb}
There is a bounded metric-measure space $(G,d,\mu)$ such that every ball is not totally bounded (in particular, $(G,d)$ is not totally bounded or geometrically doubling), and for any $\alpha,p \in (0,\infty)$ and $q\in(0,\infty]$, the embeddings $M^{\alpha}_{p,q}(G,d,\mu) \hookrightarrow L^p(G,\mu)$ and $N^{\alpha}_{p,q}(G,d,\mu) \hookrightarrow L^p(G,\mu)$  are compact.
\end{ex}

\begin{figure}[h]
\centering
\includegraphics[width=3.3in]{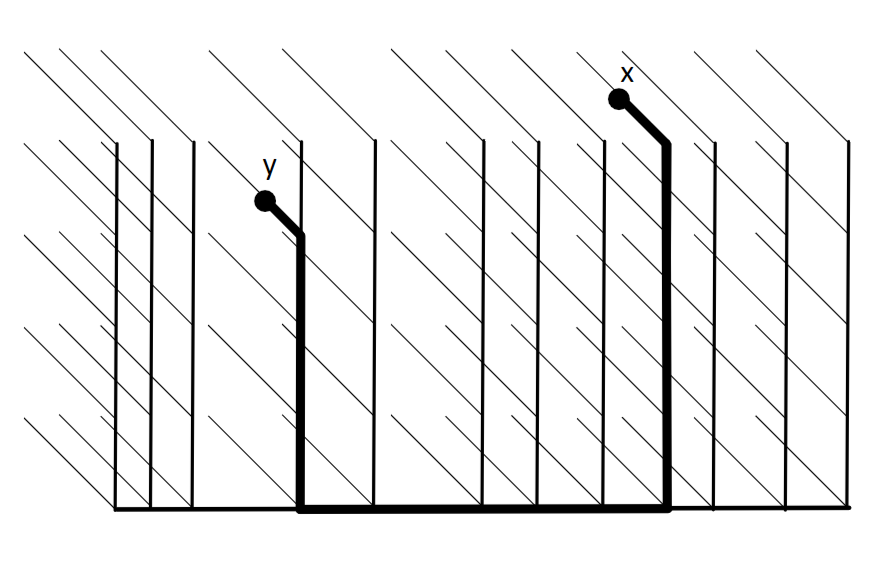}
\vspace*{-8mm}
\caption{The distance between two points in the infinite comb}
\label{fig:comb}
\end{figure}
\begin{proof}
Let $D $ be a dense and countable subset of the interval $(0,1/4)$. We construct $G\subseteq (0,1)^{\mathbb{N}}$ in the following way. Let $I = \{(t,0,0,\dots): t\in  (0,1/4)\}$. For each $n\in\mathbb{N}$ and each $\vv{d} := (d_{1},d_{2},\dots,d_{n}) \in D^{n}$, consider the set
\begin{align*}
I_{\vv{d}} := \big\{(d_{1},d_{2}/2,\dots,d_{n}/2^{n-1},t,0,0,\dots ): t\in (0,2^{-n-2}) \big\},
\end{align*}
and define $G\subseteq(0,1)^\mathbb{N}$ by setting
\begin{align*}
G := I \cup \bigcup_{n=1}^{\infty} \bigcup_{\vv{d} \in D^{n}} I_{\vv{d}}
\end{align*}

We endow $G$ with the following metric: Let $x,y \in G$, $x = (x_{n})_{n=1}^{\infty}, y = (y_{n})_{n=1}^{\infty}$. If $x=y$ then we set $d(x,y) := 0$. When $x\neq y$, we define  $\kappa := \kappa(x,y) := \min\{n\in\mathbb{N}: x_{n} \neq y_n \}$ and set

\begin{align*}
d(x,y) := |x_{\kappa} - y_{\kappa}| + \sum_{n = \kappa + 1}^{\infty} (x_{n} + y_{n}). 
\end{align*}
It is straightforward to check that $(G,d)$ is a bounded metric space and no ball is totally bounded.

In order to equip $G$ with an appropriate measure we need to properly enumerate the elements of the set $\bigcup_{n=1}^{\infty} D^{n}$. 

Let $\phi:\mathbb{P} \to D$ be any bijection between $D$ and set of prime numbers $\mathbb{P}$. We have $D = \{\phi(p): p\in \mathbb{P} \}$. Moreover, for any $n\in\mathbb{N}$ and any $\vv{p}  = (p_{1},\dots,p_{n})\in  \mathbb{P}^{n}$ we define $\Phi(\vv{p}) := (\phi(p_{1}),\dots, \phi(p_{n}))$. Note that $\Phi$ is bijection between $\bigcup_{n=1}^{\infty} \mathbb{P}^{n}$ and $\bigcup_{n=1}^{\infty} D^{n}$.

Now we will define bijection $\Psi$ between $\bigcup_{n=1}^{\infty} \mathbb{P}^{n}$ and set of positive integers by induction. We will walk through the few first steps to emphasize the idea of construction.
For 2 and 3 we take $\Psi(2) = 1,\Psi(3) = 2$. Since $4 = 2\cdot 2$ we define $\Psi((2,2)) = 3$. 5 is prime so we can take $\Psi(5) = 4$. Now since $6 = 2\cdot 3 = 3 \cdot 2$ we define $\Psi((2,3))$ and $\Psi((3,2))$ using 5 and 6.

For $m\in\mathbb{N}$ with $m\geq 2$, consider 
\begin{align*}
G_{m} := \big\{(p_{1},p_{2},\dots,p_{N_{k}}) \in \mathbb{P}^{N_{k}}: \Pi_{i=1}^{N_{k}} p_{i} = k \text{ for } 2\leq k \leq m \big\}.
\end{align*}

Suppose that for some $m\geq 2$ we have defined a bijection $\Psi:G_{m} \to \{1,2,\dots,\#G_{m} \}$ such that for any $2\leq \ell < k \leq m$ if $\vv{p} = (p_{1},\dots,p_{N_{\ell}})$ and $\vv{q} = (q_{1},\dots,q_{N_{k}})$  are decompositions of $\ell$ and $k$ respectively, then $\Psi(\vv{p}) < \Psi(\vv{q})$.

Let $\tilde{\Psi}$ be any bijection between $G_{m+1} \setminus G_{m}$ and $\{n\in \mathbb{N}: \#G_{m} +1 \leq n \leq \#G_{m+1}  \}$.

We extend $\Psi$ onto $G_{m+1}$ by $\Psi\restriction _{G_{m+1}\setminus G_{m} }=\tilde{\Psi}$.

Let $b_{k} :=\Phi \circ \Psi^{-1} (k)$. We define $J_0 = I$ and for $k\geq 1$  $J_{k} = I_{b_k}$. Now $G = \bigcup_{k=0}^{\infty} J_{k}$. Let $a_{k} := \diam_d(J_{k})$.

For $k\in\mathbb{N}$, consider $f_{k}: (0,a_{k}) \to J_{k}$ given by $f_{k}(t) := (d_{1},d_{2}/2,\dots,d_{n}/2^{n-1},t,0,0,\dots)$, where $(d_{1},d_{2},\dots,d_{n})=b_k$,  
and $f_0: (0,1/4) \to J_0$ $f_0(t)=(t,0,0,\dots)$. We denote by $L_{k}$ the $\sigma$-algebra of Lebesgue measurable subsets of interval $(0,a_{k})$. Let us define a function $g_{k}: (0,a_{k}) \to \mathbb{R}_{+} $ by setting $g_k(y) := \exp( -(y+k)^{2})$, for every $y\in(0,a_{k})$.

On $L_{k}$ we consider measures $\nu_{k}$ given by $d\nu_{k} := g_{k} dl_1$. We pushforward $(L_{k},\nu_{k})$ with $f_{k}$ and obtain sequence of $\sigma$-algebras $\mathfrak{M}_{k} := f_{k} \# L_{k}$ 
and sequence of measures defined on $\mathfrak{M}_{k}$  $\mu_{k} := f_{k}\# \nu_{k}$.
Finally, we define 
\begin{align*}
\mathfrak{M} := \Big\{A \subseteq G: A = \bigcup_{k=0}^{\infty} A_{k} \text{ for some } A_{k} \in \mathfrak{M}_{k} \Big\} 
\end{align*}
and set
\begin{align*}
\mu(A) := \sum_{k=0}^{\infty} \mu_{k}(A_{k})\quad\mbox{ for all $A\in\mathfrak{M}$.}
\end{align*}

We will show that $\mu$ is integrable. We can assume that $0<r<1/4$.
Let $x := (t,0,\dots) \in J_{0}$. Then 

\begin{align*}
B_d(x,r) \cap J_0 = \big\{(y,0,\dots): \max\{0,t-r\} < y < \min\{1/4,t+r\} \big\}.
\end{align*}
Since 
\begin{align*}
\min\{1/4,t+r\} - \max\{0,t-r\} \geq r 
\end{align*}
we have
\begin{align*}
\mu(B_d(x,r)\cap J_0) = \int\limits_{\max\{0,t-r\}}^{\min\{1/4,t+r\}} \exp(-y^2)dy \geq r \exp(-1/16).
\end{align*}
Hence,
\begin{align}
\label{qq-33}
\int_{J_0} \frac{1}{\mu(B_d(x,r))}\, d\mu(x) \leq \frac{e^{1/16}}{r}.
\end{align}

Let $k\geq 1$ and $x\in J_{k} = I_{b_{k}}$, $b_{k}= (d_{1},\dots,d_{n})$. 
We write  $x = (x_{i})_{i=1}^{\infty}$, $x_{i} = d_{i}/2^{i-1}$ for $i=1,\dots,n$, $x_{n+1} = t$ for  $t\in (0,a_{k})$, $a_{k} = 2^{-n-2}$  and for $i \geq n+2$ we have $x_{i} = 0$.

There is a finite, increasing sequence of nonnegative integers $i_0,i_1, \dots , i_{n-1},i_{n}$ such that $i_0 =0$, $i_n = k$ and for any $1\leq j \leq n$ $(d_{1},\dots,d_{j}) = b_{i_{j}}$.

If $t\geq r/2$ then
\begin{align*}
\mu(B_d(x,r) \cap J_{k}) \geq \int_{t-r/2}^{t-r/4} \exp(-(y+k)^{2})\, dy \geq r/4 \exp(-(t-r/4 + k)^2).
\end{align*}
In case $t<r/2$ we consider
\begin{align*}
L = \min \{j \in \{0,1,\dots,n\}: B_d(x,r) \cap J_{i_j} \neq \emptyset \}.
\end{align*}
Clearly $L<n$.
Let $\tilde{x} = (x_{1},\dots,x_{L+1},0,0,\dots)$. We now consider two cases.
\smallskip

{\tt 1st case:} Assume that $d(x,\tilde{x}) \geq r/2$. Since 
\begin{align*}
d(x,\tilde{x}) = \sum_{j=L+1}^{n} x_{j+1}
\end{align*}
we have
\begin{align*}
\mu(B_d(x,r) \cap \bigcup_{j=L+1}^{n} J_{i_{j}}) &\geq  \sum_{j=L+1}^{n}\int_{0}^{x_{j+1}} \exp(-(y+ i_{j})^2) \, dy \\
&=  \sum_{j=L+1}^{n}\int_{i_{j}}^{i_{j}+x_{j+1}} \exp(-y^2) \, dy \\
&\geq \int_{k - \sum_{j=L+1}^{n-1} x_{j+1}}^{k+t} \exp(-y^{2}) dy \\
&\geq \int_{t - r/2}^{t- r/4} \exp(-(y + k)^2) dy .
\end{align*}
\smallskip

{\tt 2nd case:}  We assume that $d(x,\tilde{x}) < r/2$. From the triangle inequality, we have that 
\begin{align*}
B_d(\tilde{x},r/2) \subseteq B_d(x,r).
\end{align*}
Moreover, due to the minimality we have
\begin{align*}
\mu(B_d(x,r))&\geq \mu( B_d(\tilde{x},r/2) \cap J_{i_{L}}) \geq \min \left( \int_{x_{L+1} - r/2}^{x_{L+1}} \exp(-(y+i_{L})^2) \,dy , \int_{x_{L+1}}^{x_{L+1}+r/2} \exp(-(y+i_{L})^2) \,dy\right) \\
&= \min \left(\int_{t-r3/4}^{t-r/4} \exp(-(z+x_{L+1} - t + r/4 + i_{L})^2) dz, \int_{t-r/2}^{t} \exp(-(z+x_{L+1} - t + r/2 + i_{L})^2) dz \right) \\
&\geq \int_{t-r/2}^{t - r/4} \exp(-(z+k)^2) dz
\end{align*}
since $i_L \leq k - 1$.

Finally note that for $k\geq 1$ we always have the inequality

\begin{align*}
\mu(B_d(x,r)) \geq \int_{t-r/2}^{t-r/4} \exp(-(y+k)^{2})\, dy \geq r/4 \exp(-(t-r/4 + k)^2).
\end{align*}
Now
\begin{align*}
\sum_{k=1}^{\infty} \int_{J_{k}} \frac{1}{\mu(B_d(x,r))}\, d\mu(x) &\leq 4/r \sum_{k=1}^{\infty}\int_{0}^{a_k}  \exp((t-r/4 + k)^2) \cdot \exp(-(t+k)^2) dt \\
&\leq 4/r \sum_{k=1}^{\infty} \int_{k}^{k+a_{k}} \exp(-rt/2 + r^2/16) dt < \infty.
\end{align*}
By combining this with \eqref{qq-33} and the fact that $G = \bigcup_{k=0}^{\infty} J_{k}$, we conclude that $\mu$ is integrable. Therefore, the remaining claim now follows from Corollary~\ref{remarkoint}.
\end{proof}

\subsubsection{Main Characterization}
\label{subsubsect:charac}
Collecting Theorem~\ref{czwarte}, Corollary~\ref{remarkoint} with $C=1$, $\mu=\nu$, and Proposition~\ref{TB-int}, we get the following characterization for the compactness of the embeddings $M^{\alpha}_{p,q}\hookrightarrow L^{p}$ and $N^{\alpha}_{p,q}\hookrightarrow  L^{p}$. The reader is reminded that the notation ``$\alpha\preceq_q{\rm ind}(X,d)$" introduced in Convention~\ref{convention}.

\begin{tw} \label{piate}
Let $(X,d,\mu)$ be a quasi-metric-measure space and suppose there exists $\delta_0\in(0,\infty)$ such that
$\mu$ is $(C_d,\delta)$-doubling for every $\delta\in(0,\delta_0]$, where $C_d\in[1,\infty)$ is as in \eqref{C-d}. Then the following statements are equivalent.
\begin{enumerate}
\item $(X,d)$ is totally bounded.
\item $\mu$ is integrable.
\item For all exponents $\alpha,p\in(0,\infty)$ and $q\in(0,\infty]$ with $\alpha\preceq_q{\rm ind}(X,d)$, the embedding $M^{\alpha}_{p,q}(X,d,\mu)\hookrightarrow L^p(X,\mu)$ is compact.
\item For all exponents $\alpha,p\in(0,\infty)$ and $q\in(0,\infty]$ with $\alpha\preceq_q{\rm ind}(X,d)$,  the embedding $N^{\alpha}_{p,q}(X,d,\mu)\hookrightarrow L^p(X,\mu)$ is compact.
\item There exist exponents $\alpha,p\in(0,\infty)$ and $q\in(0,\infty]$ with $\alpha\preceq_q{\rm ind}(X,d)$ such that the embedding $M^{\alpha}_{p,q}(X,d,\mu)\hookrightarrow L^p(X,\mu)$ is compact.
\item There exist exponents $\alpha,p\in(0,\infty)$ and $q\in(0,\infty]$ with $\alpha\preceq_q{\rm ind}(X,d)$ such that the embedding $N^{\alpha}_{p,q}(X,d,\mu)\hookrightarrow L^p(X,\mu)$ is compact.
\end{enumerate}
\end{tw}

If $(X,d,\mu)$ is an ultrametric-measure space then $C_d=1$ and so, ${\rm ind}(X,d)=\infty$ and  $\mu$ is trivially $(C_d,\delta)$-doubling for all $\delta\in(0,\infty)$. From these observations we have the following immediate consequence of Theorem~\ref{piate}.

\begin{tw} \label{piate-ult}
Let $(X,d,\mu)$ be an ultrametric-measure space. Then the following statements are equivalent.
\begin{enumerate}
\item $(X,d)$ is totally bounded.
\item $\mu$ is integrable.
\item For all  $\alpha,p\in(0,\infty)$ and $q\in(0,\infty]$, the embedding $M^{\alpha}_{p,q}(X,d,\mu)\hookrightarrow L^p(X,\mu)$ is compact.
\item For all  $\alpha,p\in(0,\infty)$ and $q\in(0,\infty]$,  the embedding $N^{\alpha}_{p,q}(X,d,\mu)\hookrightarrow L^p(X,\mu)$ is compact.
\item There exist  $\alpha,p\in(0,\infty)$ and $q\in(0,\infty]$ such that the embedding $M^{\alpha}_{p,q}(X,d,\mu)\hookrightarrow L^p(X,\mu)$ is compact.
\item There exist  $\alpha,p\in(0,\infty)$ and $q\in(0,\infty]$ such that the embedding $N^{\alpha}_{p,q}(X,d,\mu)\hookrightarrow L^p(X,\mu)$ is compact.
\end{enumerate}
\end{tw}

In the following example we see that it is also possible for $M^{\alpha}_{p,q}$ and $N^{\alpha}_{p,q}$ to embed compactly into $L^p$ in unbounded quasi-metric-measure spaces.
\begin{ex}
Let $(\mathbb{R}_{+}, d, \mu)$, where $d:=|\cdot-\cdot|$ is the standard Euclidean distance and $\mu$ is defined by
$d\mu:=e^{x^{\beta}}dx$ with $\beta\in(0,\infty)$. If $\alpha\in(0,1]$, $p\in(0,\infty)$, and $q\in(0,\infty]$, where the value $\alpha=1$ is only permitted when $q=\infty$, then the embeddings
\[
	M^{\alpha}_{p,q}(\mathbb{R}_{+}, d, \mu) \hookrightarrow L^p(\mathbb{R}_{+},\mu)\quad\mbox{and}\quad	N^{\alpha}_{p,q}(\mathbb{R}_{+}, d, \mu)\hookrightarrow L^p(\mathbb{R}_{+},\mu)
\]
are compact if and only if $\beta\in(1,\infty)$.
\end{ex}
\begin{proof}
If $\beta > 1$, then by Example~\ref{expbeta} in Section~\ref{sect:meas} we have that $\mu$ is integrable. Thus Corollary~\ref{remarkoint} yields compactness of the embeddings. On the other hand, let $\beta\in (0,1]$ and suppose that 
$M^{\alpha}_{p,q}(\mathbb{R}_{+}, d, \mu) \hookrightarrow L^p(X,\mu)$ or $ N^{\alpha}_{p,q}(\mathbb{R}_{+}, d, \mu)\hookrightarrow L^p(X,\mu)$ is compact. By virtue of Example~\ref{exp0} in Section~\ref{sect:meas} we have that $\mu$ is $(2,\delta)$-doubling for every $\delta\in(0,\infty)$. Thus Theorem~\ref{czwarte} implies that $\mathbb{R}_{+}$ is totally bounded which is a clear contradiction.
\end{proof}

In the next theorem, we  show that under certain assumptions on the measure $\mu$, the spaces $M^{\alpha}_{p,q}$ and $N^{\alpha}_{p,q}$ never embed compactly into $L^p$.
\begin{tw}\label{thm-doubinf}
Let $(X,d,\mu)$ be an unbounded quasi-metric-measure space with $\mu(X) <\infty$ and assume that $(X,d,\mu)$ is doubling at infinity (in the sense of Definition~\ref{doub-inf}). Fix exponents $\alpha,p\in(0,\infty)$ and $q\in(0,\infty]$ such that  $\alpha\leq(\log_{2}C_d)^{-1}$, where $\alpha=(\log_{2}C_d)^{-1}$  is only permitted when $q=\infty$. Then the embeddings
\begin{equation}
\label{ni4-24-2}
	M^{\alpha}_{p,q}(X,d,\mu)\hookrightarrow L^p(X,\mu)\quad\mbox{and}\quad
	N^{\alpha}_{p,q}(X,d,\mu)\hookrightarrow L^p(X,\mu)
\end{equation}
are not compact.
\end{tw}

\begin{proof}
Let $x_0\in X$ be as in the definition of doubling at infinity (Definition~\ref{doub-inf}) and let $\{ R_{N} \}_{N \in \mathbb{N}}$ be an increasing sequence such that $R_N \geq 1$, $R_{N} \to \infty$ and 
\begin{align*}
\lim_{N \to \infty}\dfrac{\mu(X\setminus  B_d(x_0,R_{N}))}{\mu(X\setminus  B_d(x_0,C_d R_{N}))} =  \liminf_{R \to \infty}\dfrac{\mu(X\setminus  B_d(x_0,R))}{\mu(X\setminus  B_d(x_0,C_d R))} < \infty.
\end{align*}
Observe that ${\rm dist}_d(B_d(x_0,R_N),X\setminus B_d(x_0,C_d R_N))\geq R_N\geq1$ and therefore, we can define a sequence of functions by setting
\[
f_N(x):= \dfrac{1}{[\mu(X\setminus  B_d(x_0,C_dR_N))]^{1/p}}\,\Phi_N(x)\quad\mbox{for all $N\in\mathbb{N}$ and $x\in X$,}
\] 
where $\Phi_N: X\to[0,1]$ is the measurable function given by Lemma~\ref{GVa2} that satisfies $\Phi_{N} \equiv 0$ on $B_d(x_0,R_N)$ and $\Phi_{N}\equiv 1$ outside of $B_d(x_0,C_dR_N)$. As an immediate consequence of the doubling at infinity condition, the estimates in \eqref{gradest1}-\eqref{gradest2}, and that fact that each $f_N$ vanishes pointwise in $B_d(x_0,R_N)$, we have (keeping in mind ${\rm dist}_d(B_d(x_0,R_N),X\setminus B_d(x_0,C_d R_N))\geq R_N\geq1$) that the sequence $\{f_N\}_{N\in\mathbb{N}}$ is bounded in $M^\alpha_{p,q}(X,d,\mu)$ and $N^\alpha_{p,q}(X,d,\mu)$. Now, suppose to the contrary that $\{f_N\}_{N\in\mathbb{N}}$ is precompact in $L^p(X,\mu)$. Then, by the $p$-equi-integrability condition in the Lebesgue--Vitali theorem (Theorem~\ref{compLp}) we have
\begin{align*}
\int\limits_{X\setminus  B_d(x_0,R)} |f_{N}|^p\, d\mu \to 0\quad \textrm{ uniformly in $N$ as } R\to \infty.
\end{align*}
However, this is impossible since
\begin{align*}
\int\limits_{X\setminus  B_d(x_0,C_d R_N)} |f_N|^p\, d\mu = 1\quad \textrm{ for all } N\in\mathbb{N}.
\end{align*} 
Therefore the embeddings in \eqref{ni4-24-2} cannot be compact and the proof of the theorem is now complete.
\end{proof}

\subsection{Embeddings into $L^{\tilde{p}}$ with $\tilde{p}<p$}
\label{subsect:p-less}
In this subsection we investigate necessary and sufficient conditions guaranteeing that  $M^{\alpha}_{p,q}$ and $N^{\alpha}_{p,q}$ compactly embed into $L^{\tilde{p}}$ with $\tilde{p}<p$.
More specifically, we will establish the following result (the reader is reminded that the notation ``$\alpha\preceq_q{\rm ind}(X,d)$" was introduced in  Convention~\ref{convention}).

\begin{tw} \label{drugie_plus}
Let $(X,d,\mu)$ be a quasi-metric-measure space. If $\mu(X) <\infty$, then for any $\alpha, p\in (0,\infty)$, $q\in (0,\infty]$, and  $\tilde{p}\in(0,p)$, the embeddings
\begin{align} \label{emb_in_lower}
M^{\alpha}_{p,q}(X,d,\mu)\hookrightarrow L^{\tilde{p}}(X,\mu)\quad\mbox{and}\quad	N^{\alpha}_{p,q}(X,d,\mu)\hookrightarrow L^{\tilde{p}}(X,\mu)
\end{align}
are compact. On the other hand, if either one of the embeddings in \eqref{emb_in_lower} is continuous for some $p\in (0,\infty)$, $\tilde{p}\in(0,p)$, $q\in (0,\infty]$, and  $\alpha\in (0,\infty)$ with $\alpha\preceq_q{\rm ind}(X,d)$, then  $\mu(X)<\infty$.
\end{tw}

\begin{rem}
Given a sigma-finite measure space\footnote{Every  quasi-metric-measure space is sigma-finite.}  $(X, \mu)$, if $L^p(X,\mu) \hookrightarrow L^{\tilde{p}}(X,\mu)$ for some fixed $\tilde{p}\in(0,p)$, then $\mu(X)<\infty$. Indeed, let $A_j \subseteq X$ be such that $X=\bigcup_{j=1}^{\infty}A_j$, $A_j \subseteq A_{j+1}$, and $\mu(A_j) <\infty$.  We can suppose that $\mu(X) >0$.  In this case for some $C\in(0,\infty)$ and for all sufficiently large $j \in \mathbb{N}$, we have $\mu(A_j)>0$ and 
\begin{align*}
[\mu(A_j)]^{1/\tilde{p}} = \|\chi_{A_j}\|_{L^{\tilde{p}}(X,\mu)} \leq C \|\chi_{A_j}\|_{L^p(X,\mu)} =  C[\mu(A_j)]^{1/p},
\end{align*}
which turns into
\begin{align*}
[\mu(A_j)]^{1/\tilde{p} - 1/p} \leq C.
\end{align*}
After passing to the limit as $j\to\infty$, we obtain that $\mu(X)<\infty$. 
\end{rem}

Before proving Theorem~\ref{drugie_plus}, we shall establish the following result.

\begin{tw} \label{comp_p/theta}
Let $(X,d,\mu)$ be a quasi-metric-measure space. Let $\nu$ be a measure such that $\nu<<\mu$ and $\frac{d\nu}{d\mu} \in L^{\theta'}(X,\mu)$ for some $\theta\in(1,\infty)$, where $\theta':=\frac{\theta}{\theta-1}$. Then for any $\alpha,p \in (0,\infty)$ and $q\in (0,\infty]$  the embeddings
\begin{align} \label{emb_in_p/theta}
M^{\alpha}_{p,q}(X,d,\mu)\hookrightarrow L^{p/\theta}(X,\nu)\quad\mbox{and}\quad	N^{\alpha}_{p,q}(X,d,\mu)\hookrightarrow L^{p/\theta}(X,\nu)
\end{align}
are compact.
\end{tw}
\begin{proof}
As was the case in the proofs of Theorems~\ref{emb_in_L0} and \ref{drugie}, it is enough to to show that the embedding $M^{\alpha,p}(X,d,\mu)\hookrightarrow L^{p/\theta}(X,\nu)$ is compact for every $\alpha,p \in (0,\infty)$. Fix $\alpha,p \in (0,\infty)$ and let $\mathcal{F}\subseteq M^{\alpha,p}(X,d,\mu)$ be a nonempty bounded set. In view of the H\"older inequality, note that for any $u\in \mathcal{F}$ and any measurable set $D\subseteq X$, we have
\begin{align}
\label{p/theta-est}
\int_{D} |u(x)|^{p/\theta}\, d\nu(x) &= \int_{D} |u(x)|^{p/\theta} \frac{d\nu}{d\mu} (x) \, d\mu(x) \nonumber \\
&\leq \|u\|_{L^p(X,\mu)}^{p/\theta} \bigg(\int_{D}\Big| \frac{d\nu}{d\mu}\Big|^{\theta'}(x)\, d\mu(x)\bigg)^{1/\theta'}.
\end{align} 
Let $x_0 \in X$ and fix $\varepsilon\in(0,\infty)$. Since $\frac{d\nu}{d\mu} \in L^{\theta'}(X,\mu)$ we can use \eqref{p/theta-est} to find sufficiently large $\lambda,R\in(0,\infty)$ such that, for any $u\in \mathcal{F}$,
\begin{align}
\label{Aqp-3}
\int_{X\setminus B_{d}(x_0,R)} |u(x)|^{p/\theta}\, d\nu(x) < \frac{\varepsilon^{p/\theta} }{3}\quad \mbox{ and } \quad \int_{E_{\lambda}} |u(x)|^{p/\theta}\, d\nu(x) < \frac{\varepsilon^{p/\theta} }{3},
\end{align}
where $E_{\lambda} := \big\{x\in X: \frac{d\nu}{d\mu}(x) > \lambda\big\}$. We define measure $\tilde{\nu} :=\nu\!\!\measurerestr\!\big(B_d(x_0,R)\cap (X\setminus E_{\lambda})\big)$. 
Let us observe that from  \eqref{p/theta-est} for any measurable set $D\subseteq X$ we have
\begin{align*}
\int_{D} |u(x)|^{p/\theta}\, d\tilde{\nu}(x) \leq \lambda^{1/\theta} \|u\|_{L^p(X,\mu)}^{p/\theta}\big[\tilde{\nu}(D)\big]^{1/\theta'}.
\end{align*} 
Thus, we get $p/\theta$-equi integrability of $\mathcal{F}$ with respect to $\tilde{\nu}$. Therefore, since 
\begin{align*}
\tilde{\nu}(X) = \nu(B_d(x_0,R)\cap (X\setminus E_{\lambda})) = \int_{B_d(x_0,R)\cap (X\setminus E_{\lambda})} \frac{d\nu}{d\mu}(x)\, d\mu(x) \leq \lambda \mu(B_{d}(x_0,R)) <\infty,
\end{align*}
and  $\tilde{\nu}<<\mu$, we use Theorem~\ref{emb_in_L0} and the Lebesgue--Vitali Theorem (Theorem~\ref{compLp}) to obtain \eqref{emb_in_p/theta} for $\tilde{\nu}$ in place of $\nu$.

To prove \eqref{emb_in_p/theta} as stated, let $\{\tilde{v}_{i} \}_{i=1}^{N}$ be a finite $\frac{\varepsilon}{3^{\theta/p}}$-net of $\mathcal{F}$ in $L^{p/\theta}(X,\tilde{\nu})$. Then, for any $u\in \mathcal{F}$ we can find $j\in \{1,\dots,N\}$ for which
\begin{align*}
\|u-\tilde{v}_{j} \|_{L^{p/\theta}(X,\tilde{\nu})} < \frac{\varepsilon}{3^{\theta/p}}.
\end{align*}
For $i=1,\dots,N$ let $v_i := \tilde{v}_{i} \chi_{B_d(x_0,R)\cap (X\setminus E_{\lambda})}$. Then this, along with \eqref{Aqp-3}, yields
\begin{align*}
\|u-v_{j} \|_{L^{p/\theta}(X,\nu)}^{p/\theta} = \int_{B_d(x_0,R)\cap (X\setminus E_{\lambda})} |u(x)-v_{j}(x)|^{p/\theta}\, d\nu(x)  + \int_{(X\setminus B_d(x_0,R))\cup E_{\lambda}} |u(x)|^{p/\theta}\, d\nu(x) <\varepsilon^{p/\theta}.
\end{align*}
Hence, $\{v_{i} \}_{i=1}^{N}$ is a finite $\varepsilon$-net of $\mathcal{F}$ in $L^{p/\theta}(X,\nu)$, and this completes the proof of the theorem.
\end{proof}
For finite target measures we can strengthen the conclusion of Theorem~\ref{comp_p/theta} in the following manner.
\begin{cor} \label{Idontknow}
Let $(X,d,\mu)$ be a quasi-metric-measure space. Let $\nu$ be a finite  measure such that $\nu<<\mu$ and $\frac{d\nu}{d\mu} \in L^{\theta'}(X,\mu)$ for some $\theta\in(1,\infty)$, where $\theta':=\frac{\theta}{\theta-1}$. Then for any $\alpha,p \in (0,\infty)$, $q\in (0,\infty]$, and $ \tilde{p} \in (0, p/\theta]$, the embeddings
\begin{align} \label{emb_below_p/theta}
M^{\alpha}_{p,q}(X,d,\mu)\hookrightarrow L^{\tilde{p}}(X,\nu)\quad\mbox{and}\quad	N^{\alpha}_{p,q}(X,d,\mu)\hookrightarrow L^{\tilde{p}}(X,\nu)
\end{align}
are compact. Moreover, if $\nu \leq C \mu$  for some $C\in (0,\infty)$, then the embeddings above are compact for any $\tilde{p} \in (0, p)$.
\end{cor}
\begin{proof}
The first part is an immediate consequence of Theorem~\ref{comp_p/theta} and the H\"older inequality, where as the second part follows from the first one, since $\frac{d\nu}{d\mu} \in L^{1}(X,\mu)\cap L^{\infty}(X,\mu)$ implies $\frac{d\nu}{d\mu} \in L^{1}(X,\mu)\cap L^{\theta'}(X,\mu)$ with $\theta := p/ \tilde{p}$.
\end{proof}

We are now ready to  prove Theorem~\ref{drugie_plus}.

%\begin{tw} \label{drugie_plus}
%Let $(X,d,\mu)$ be a quasi-metric-measure space. If $\mu(X) <\infty$, then for any $\alpha, p\in (0,\infty)$, $q\in (0,\infty]$, and  $\tilde{p}\in(0,p)$, the embeddings
%\begin{align} \label{emb_in_lower}
%M^{\alpha}_{p,q}(X,d,\mu)\hookrightarrow L^{\tilde{p}}(X,\mu)\quad\mbox{and}\quad	N^{\alpha}_{p,q}(X,d,\mu)\hookrightarrow L^{\tilde{p}}(X,\mu)
%\end{align}
%are compact. Moreover, if either one of the embeddings in \eqref{emb_in_lower} is continuous for some $p\in (0,\infty)$, $\tilde{p}\in(0,p)$, $q\in (0,\infty]$, and  $\alpha\in (0,\infty)$ with $\alpha\preceq_q{\rm ind}(X,d)$, then  $\mu(X)<\infty$.
%\end{tw}
\begin{proof}[Proof of Theorem~\ref{drugie_plus}]
The first claim in the statement of Theorem~\ref{drugie_plus} follows from Corollary~\ref{Idontknow} with $C=1$ and $\nu=\mu$.

Next, assume that either one of the embeddings in \eqref{emb_in_lower} is continuous. Since $\alpha\preceq_q{\rm ind}(X,d)$, we can find a quasi-metric $d'$ on $X$ which is equivalent to $d$ and satisfies $\alpha\leq(\log_{2}C_{d'})^{-1}$, where the value $\alpha=(\log_{2}C_{d'})^{-1}$ can only occur when $q=\infty$. Fix finite numbers $r\geq1$ and $\beta\in[\alpha,(\log_{2}C_{d'})^{-1}]$, where $\beta\neq \alpha$
unless $\alpha=(\log_{2}C_{d'})^{-1}$.  Set $r_0:=r$ and for $k\in\mathbb{N}_0$ 
let $r_{k+1}:=\big[(r_k)^\beta+2^{-k}r^\beta\big]^{1/\beta}$. Observe that $(r_{k+1})^\beta-(r_k)^\beta=2^{-k}r^\beta$, from which it is easy to see that 
$r\leq r_k<r_{k+1}$ for all $k\in\mathbb{N}_0$.
%
%$$B_{\varrho_{\#}}\big(x,r_j)\subseteq B_{\varrho_{\#}}\big(x,r_1)\subseteq B_{\rho}(x,r).$$
%
Moreover, we have
$$
r_{k+1}=\bigg[r^\beta+\sum_{j=0}^{k}2^{-j}r^\beta\bigg]^{1/\beta}
<\bigg[r^\beta+\sum_{j=0}^{\infty}2^{-j}r^\beta\bigg]^{1/\beta}
=3^{1/\beta}r\quad\mbox{for any $k\in\mathbb{N}_0$.}
$$
In light of Lemma~\ref{DST1}, we can find a quasi-metric $\rho$ on $X$ which is equivalent to ${d'}$ (hence, is also equivalent to $d$) and such that all $\rho$-balls are open, $C_\rho\leq C_{d'}$, and $\rho^\beta$ is a genuine metric. Fix any point $z\in X$ and let $B_k:=B_{\rho}(z,r_k)$ for every $k\in\mathbb{N}_0$. We can assume that $X\setminus B_{k}\neq\emptyset$ for all $k\in\mathbb{N}_0$; otherwise, $X$ is bounded and we immediately have that $\mu(X)<\infty$. Since $\rho^\beta$ is a genuine metric, it follows from the triangle inequality that
$$
{\rm dist}_\rho(B_k,X\setminus B_{k+1})
\geq\big[(r_{k+1})^\beta-(r_k)^\beta\big]^{1/\beta}
=2^{-k/\beta}r\geq2^{-k/\beta}>0.
$$ 
Therefore, we can apply Lemma~\ref{GVa2} (for the quasi-metric-measure space $(X,\rho,\mu)$) to obtain sequence of  functions $\phi_{k}: X\to[0,1]$ such that $\phi_{k}= 1$ on $B_k$ and $\phi_{k} = 0$ on $X \setminus B_{k+1}$, and for which the following estimates hold for some constant $C\in(0,\infty)$ which is independent of $k$ and $r$: 
$$
\Vert \phi_{k}\Vert_{{M}^\alpha_{p,q}(X,d,\mu)}\leq C\Vert \phi_{k}\Vert_{{M}^\alpha_{p,q}(X,\rho,\mu)} \leq C\big(1+2^{k\alpha/\beta}\big)
[\mu(B_{k+1})]^{1/p}\leq 2C2^{k\alpha/\beta}
[\mu(B_{k+1})]^{1/p},
$$
and
$$
\Vert \phi_{k}\Vert_{{N}^\alpha_{p,q}(X,d,\mu)}\leq C\Vert \phi_{k}\Vert_{{N}^\alpha_{p,q}(X,\rho,\mu)}\leq C\big(1+2^{k\alpha/\beta}\big)
[\mu(B_{k+1})]^{1/p}\leq 2 C 2^{k\alpha/\beta}
[\mu(B_{k+1})]^{1/p}.
$$
Note that we have made use Lemma~\ref{equivspaces} in obtaining the first two inequalities in each of the estimates above. On the other hand, we have 
\begin{align*}
\|\phi_{k}\|_{L^{\tilde{p}}(X,\mu)}\geq [\mu(B_{k})]^{1/\tilde{p}}.
\end{align*}
Since $\alpha\leq\beta$ implies $2^{k\alpha/\beta}\leq2^{k}$, by combining these estimates with the embedding \eqref{emb_in_lower}, we obtain an inequality of the form
\begin{align*}
\frac{1}{[\mu(B_{k+1})]^{1/p}}  \leq 2C\widetilde{C} 2^{k} \frac{1}{[\mu(B_{k})]^{1/\tilde{p}}},
\end{align*}
where $\widetilde{C}\in(0,\infty)$ is the constant from the embeddings in \eqref{emb_in_lower}.
Employing \cite[Lemma 16]{agh20} with the bounded sequence $a_{k} := 1/\mu(B_{k})$ gives
\begin{align*}
\bigg(\frac{1}{\mu(B_0)}\bigg)^{1 - \tilde{p}/p} \big(2C\widetilde{C}\big)^{\tilde{p}} 2^{\frac{p\tilde{p}}{p-\tilde{p}}} \geq 1
\end{align*}
which leads to
\begin{align*}
[\mu(B_\rho(z,r))]^{1 - \tilde{p}/p} \leq \big(2C\widetilde{C}\big)^{\tilde{p}} 2^{\frac{p\tilde{p}}{p-\tilde{p}}}.
\end{align*}
Passing to the limit as $r\to\infty$ completes the proof of Theorem~\ref{drugie_plus}.
\end{proof}

Theorem~\ref{drugie_plus} concerns embeddings of $M^{\alpha}_{p,q}$ and $N^{\alpha}_{p,q}$ into 
$L^{\tilde{p}}$ spaces with the \textit{same} measure when $\tilde{p}<p$. We conclude this subsection with the following example which shows that if we have two different measures then such embedding may be compact even though both of the measures are infinite.
\begin{ex}
Fix $n\in \mathbb{N}$ and let $\nu$ be the measure given by $d\nu := \frac{dx}{1+|x|^n}$. Then $\nu(\mathbb{R}^n)=\infty$ and for any $p\in[1,\infty)$ and  $\tilde{p}\in(0,p)$, the embedding $W^{1,p}(\mathbb{R}^n) \hookrightarrow L^{\tilde{p}}(\mathbb{R}^n,\nu)$ is compact, where $W^{1,p}(\mathbb{R}^n)$ denotes the classical Sobolev space in $\mathbb{R}^n$ (equipped with the standard Euclidean distance and Lebesgue measure). Consequently, $M^{1,p}(\mathbb{R}^n) \hookrightarrow L^{\tilde{p}}(\mathbb{R}^n,\nu)$ is also compact.
\end{ex}
\begin{proof}
From the H\"{o}lder inequality, for any measurable set $E\subseteq\mathbb{R}^n$ we have 

\begin{align} \label{weight_Hol}
\int_{E} |f(x)|^{\tilde{p}}\, d\nu(x) \leq \bigg(\int_{\mathbb{R}^{n}} |f(x)|^{p}\, dx\bigg)^{\tilde{p}/p} \bigg( \int_{E} \Big( {1+|x|^n}\Big)^{-p/(p-\tilde{p})}\, dx\bigg)^{1 - \tilde{p}/p}.
\end{align}
Since for any $R\in(0,\infty)$ the embedding $W^{1,p}(\mathbb{R}^n) \hookrightarrow L^{p}(B(0,R))$ is compact and the embedding  $L^{p}(B(0,R)) \hookrightarrow L^{\tilde{p}}(B(0,R), \nu)$ is continuous, we have compactness of the embedding $W^{1,p}(\mathbb{R}^n) \hookrightarrow L^{\tilde{p}}(B(0,R), \nu)$. Furthermore, by  \eqref{weight_Hol} for any bounded family $\mathcal{F} \subseteq W^{1,p}(\mathbb{R}^n)$ we have
\begin{align*}
\lim_{R\to \infty} \sup_{f\in \mathcal{F}} \int_{\mathbb{R}^n \setminus B(0,R)} |f|^{\tilde{p}}\, d\nu = 0.
\end{align*}
Therefore, by standard considerations we get that $W^{1,p}(\mathbb{R}^n) \hookrightarrow L^{\tilde{p}}(\mathbb{R}^n,\nu)$ is compact. 

Finally, since $M^{1,p}(\mathbb{R}^n) \hookrightarrow W^{1,p}(\mathbb{R}^n)$ (see, e.g., \cite[Theorem~7]{H}), we have that $M^{1,p}(\mathbb{R}^n) \hookrightarrow L^{\tilde{p}}(\mathbb{R}^n,\nu)$ is also compact.
\end{proof}

\subsection{Embeddings into $L^{\tilde{p}}$ with $\tilde{p}>p$ and the Rellich--Kondrachov compactness theorem}
\label{subsect:p-greater}
In this subsection we prove a Rellich--Kondrachov-type compactness theorem (Theorem~\ref{snr-3}).

We being by introducing some terminology.
Let $(X,d,\mu)$ be a quasi-metric-measure space and fix $s\in(0,\infty)$. Recall that the measure $\mu$ is said to be \textit{locally lower Ahlfors $s$-regular} if there exists a constant $b\in(0,\infty)$ such that
\begin{equation}
\label{llAr}
\mu(B_d(x,r))\geq br^s\,\mbox{ for all $x\in X$ and $r\in(0,1]$.}
\end{equation}
Note that by adjusting the constant $b$, the upper bound 1 for the radii $r$ in \eqref{llAr} can be replaced by any strictly positive and finite number.

\begin{tw}\label{snr-3}
Let $(X,d,\mu)$ be a quasi-metric-measure space such that $\mu(X)<\infty$ and the measure $\mu$ is locally lower $s$-Ahlfors-regular for some $s\in(0,\infty)$. Then the following statements are valid for all $\alpha,p\in(0,\infty)$ and $q\in(0,\infty]$:
\begin{enumerate}
\item If $\alpha p<s$, then for each fixed ${\tilde{p}}\in(0,sp/(s-\alpha p))$ the embeddings $M^{\alpha}_{p,q}(X,d,\mu)\hookrightarrow L^{\tilde{p}}(X,\mu)$ and $N^{\alpha}_{p,q}(X,d,\mu)\hookrightarrow L^{\tilde{p}}(X,\mu)$ are compact.
\item If $\alpha p>s$, then for each fixed $\gamma\in(0,\alpha-s/p)$ the embeddings $M^{\alpha}_{p,q}(X,d,\mu)\hookrightarrow \mathcal{C}^{\gamma}(X, d)$ and $N^{\alpha}_{p,q}(X,d,\mu)\hookrightarrow \mathcal{C}^{\gamma}(X, d)$ are compact. In particular, the embeddings $M^{\alpha}_{p,q}(X,d,\mu)\hookrightarrow L^{\tilde{p}}(X,\mu)$ and $N^{\alpha}_{p,q}(X,d,\mu)\hookrightarrow L^{\tilde{p}}(X,\mu)$ are compact for all $\tilde{p}\in(0,\infty)$.
\end{enumerate}
\end{tw}

\begin{rem}
By \cite[Theorem~4.13]{AYY21}, a local lower Ahlfors regular condition on the measure is necessary for the embeddings $M^{\alpha}_{p,q}(X,d,\mu)\hookrightarrow L^{\tilde{p}}(X,\mu)$ and $N^{\alpha}_{p,q}(X,d,\mu)\hookrightarrow L^{\tilde{p}}(X,\mu)$ to be continuous when $\tilde{p}>p$.
\end{rem}
\begin{proof}
As a preamble to the proof, observe that since $\mu$ is locally lower $s$-Ahlfors-regular, the assumption $\mu(X)<\infty$ ensures that $(X,d)$ is totally bounded, granted Proposition~\ref{tot}. 

Moving on, we first prove \textit{(1)} for ${M}^\alpha_{p,q}$-spaces. Assume $\alpha p<s$ and fix ${\tilde{p}}\in(0,sp/(s-\alpha p))$. Since $X$ is a bounded set, we can appeal to \cite[Theorem~3.5]{AYY21} to conclude that there exists a constant $C\in(0,\infty)$ such that
\begin{equation}
\label{rjq-23}
\Vert u\Vert_{L^{p^\ast}(X,\mu)}\leq C \Vert u\Vert_{{M}^\alpha_{p,q}(X,d,\mu)},
\end{equation}
whenever $u\in {M}^\alpha_{p,q}(X,d,\mu)$. Here $p^\ast:=sp/(s-\alpha p)\in(p,\infty)$. Given that $\mu(X)<\infty$ and $\tilde{p}<p^\ast$, it is easy to see from
\eqref{rjq-23} that ${M}^\alpha_{p,q}(X,d,\mu)\hookrightarrow L^{\tilde{p}}(X,\mu)$. We will now show that this continuous embedding is compact. To this end, suppose that $\mathcal{F}$ is a nonempty bounded subset of ${M}^\alpha_{p,q}(X,d,\mu)$ and let $M:=\sup_{f\in\mathcal{F}}\Vert f\Vert_{{M}^\alpha_{p,q}(X,d,\mu)}$. In light of Corollary~\ref{remarkoint}, there exists a sequence $\{f_n\}_{n\in\mathbb{N}}$ of functions in $\mathcal{F}$ which converges in $L^p(X,\mu)$. As such, the proof of \textit{(1)} will be complete once we show that $\{f_n\}_{n\in\mathbb{N}}$ converges in $L^{\tilde{p}}(X,\mu)$. If $\tilde{p}\in(p,p^\ast)$ then by choosing $\theta\in(0,1)$ such that
$1/\tilde{p}=\theta/p+(1-\theta)/p^\ast$, it follows from \eqref{rjq-23} and the interpolation inequality for the $L^r$-norm that for every $n,m\in\mathbb{N}$,
\begin{align}
\Vert f_n-f_m\Vert_{L^{\tilde{p}}(X,\mu)}
&\leq\Vert f_n-f_m\Vert_{L^{{p}}(X,\mu)}^\theta\Vert f_n-f_m\Vert_{L^{p^\ast}(X,\mu)}^{1-\theta}\\
&\leq\big(2^{1+1/p^\ast}CM\big)^{1-\theta}\Vert f_n-f_m\Vert_{L^{{p}}(X,\mu)}^\theta.
\end{align}
On the other hand, if $\tilde{p}\in(0,p]$ then by the H\"older inequality, we have
$$
\Vert f_n-f_m\Vert_{L^{\tilde{p}}(X,\mu)}\leq[\mu(X)]^{1/\tilde{p}-1/p}\Vert f_n-f_m\Vert_{L^{{p}}(X,\mu)}.
$$
Therefore, in either case we immediately obtain that the sequence $\{f_n\}_{n\in\mathbb{N}}$ is Cauchy, hence convergent, in $L^{\tilde{p}}(X,\mu)$.

We now prove \textit{(2)}  for ${M}^\alpha_{p,q}$-spaces. Assume $\alpha p>s$ and fix $\gamma\in(0,\alpha-s/p)$. By \cite[Theorem~3.5 and Remark~3.6]{AYY21} we have ${M}^\alpha_{p,q}(X,d,\mu)\hookrightarrow\mathcal{C}^{\alpha-s/p}(X,d)$. Since $(X,d)$ is totally bounded, it follows from Lemma~\ref{holdcpt} that the embedding $\mathcal{C}^{\alpha-s/p}(X,d)\hookrightarrow\mathcal{C}^{\gamma}(X,d)$ is well defined and compact. By combining these two embeddings, we obtain that  ${M}^\alpha_{p,q}(X,d,\mu)\hookrightarrow\mathcal{C}^{\gamma}(X,d)$ is compact.

The proof of \textit{(1)} and \textit{(2)} for ${N}^\alpha_{p,q}$-spaces is similar; however, we need to slightly modify the argument. We will point out the important changes. To prove \textit{(1)}, assume $\alpha p<s$ and ${\tilde{p}}\in(0,sp/(s-\alpha p))$, and choose $\varepsilon\in(0,\alpha)$ so that ${\tilde{p}}<sp/(s-\varepsilon p)$. Then by \cite[Theorem~3.9 and Remark~3.10]{AYY21}, we have that
\begin{equation}
\label{rjq-23-2}
\Vert u\Vert_{L^{p^\ast_\varepsilon}(X,\mu)}\leq C \Vert u\Vert_{{N}^\alpha_{p,q}(X,d,\mu)},
\end{equation}
whenever $u\in {N}^\alpha_{p,q}(X,d,\mu)$. Here $p^\ast_\varepsilon:=sp/(s-\varepsilon p)\in(p,\infty)$. At this stage, we can proceed as in the proof of \textit{(1)} for ${M}^\alpha_{p,q}$-spaces where we use \eqref{rjq-23-2} in place of \eqref{rjq-23}.

To prove \textit{(2)}, assume $\alpha p>s$ and $\gamma\in(0,\alpha-s/p)$. Then \cite[Theorem~3.9 and Remark~3.10]{AYY21} ensure that  ${N}^\alpha_{p,q}(X,d,\mu)\hookrightarrow\mathcal{C}^{\alpha-s/p}(X,d)$. The rest of the proof is now as in the proof of \textit{(2)} for ${M}^\alpha_{p,q}$-spaces.
\end{proof}

\subsection{Self-improvement of certain compact embeddings}
\label{subsect:selfimprove}
In this subsection we prove that certain compact embeddings of the spaces $M^{\alpha}_{p,q}$ and $N^{\alpha}_{p,q}$ self-improve in quasi-metric-measure spaces where the measure is doubling (as in \eqref{doub}). 

Note that if $(X,d,\mu)$ is a quasi-metric-measure space and $\mu$ is a doubling measure then by iterating the doubling condition \eqref{doub} (in a manner similar to the proof of \cite[Lemma~4.7]{Hajlasz2}) we have that 
\begin{equation}
\label{doubdim}
\mu(B_d(x,\lambda r))\leq 4^s\lambda^s\mu(B_d(x,r))
\end{equation} 
for all $x\in X$, $r\in(0,\infty)$, $\lambda\in[1,\infty)$, and $s\in[\log_2C_\mu,\infty)$, where $C_\mu$ is the doubling constant in \eqref{doub}. The exponent $s$ in \eqref{doubdim} typically plays the role of ``dimension" in this setting. With this in mind, we state and prove the main result in this section (keeping in mind Convention~\ref{convention}).

\begin{tw}\label{improve}
Let $(X,d,\mu)$ be a quasi-metric-measure space where the measure $\mu$ is doubling. Fix $\alpha,p\in(0,\infty)$, $\tilde{p}\in(0,p]$, and $q\in(0,\infty]$ with $\alpha\preceq_q{\rm ind}(X,d)$. Suppose that either of the embeddings
\[
	M^{\alpha}_{p,q}(X,d,\mu)\hookrightarrow L^{\tilde{p}}(X,\mu)\quad\mbox{or}\quad
	N^{\alpha}_{p,q}(X,d,\mu)\hookrightarrow L^{\tilde{p}}(X,\mu)
\]
is compact. Then, the following statements are valid for all $\beta,u\in(0,\infty)$, and $v\in(0,\infty]$ with $s\in(0,\infty)$ as in \eqref{doubdim}:
 
\begin{enumerate}
\item If $\beta u<s$, then for each fixed ${\tilde{q}}\in(0,su/(s-\beta u))$ the embeddings $M^{\beta}_{u,v}(X,d,\mu)\hookrightarrow L^{\tilde{q}}(X,\mu)$ and $N^{\beta}_{u,v}(X,d,\mu)\hookrightarrow L^{\tilde{q}}(X,\mu)$ are compact.
\item If $\beta u>s$, then for each fixed $\gamma\in(0,\beta-s/u)$ the embeddings $M^{\beta}_{u,v}(X,d,\mu)\hookrightarrow \mathcal{C}^{\gamma}(X, d)$ and $N^{\beta}_{u,v}(X,d,\mu)\hookrightarrow \mathcal{C}^{\gamma}(X, d)$ are compact. Moreover, the embeddings $M^{\beta}_{u,v}(X,d,\mu)\hookrightarrow L^{\tilde{q}}(X,\mu)$ and $N^{\beta}_{u,v}(X,d,\mu)\hookrightarrow L^{\tilde{q}}(X,\mu)$ are compact for all $\tilde{q}\in(0,\infty)$.
\end{enumerate}
\end{tw}
Let us remark that the doubling condition in Theorem~\ref{improve} cannot be relaxed to a locally doubling condition, in general. Indeed, for the metric-measure space $(\mathbb{N},d,\mu)$ in Example~\ref{exdis}, we have ${\rm ind}(\mathbb{N},d)=\infty$ (since $d$ is an ultrametric) and so, Theorems~\ref{drugie_plus} and \ref{czwarte}, imply that $M^{\alpha,p}(\mathbb{N},d,\mu)$ embeds compactly into $L^{\tilde{p}}(\mathbb{N},\mu)$ for all $\alpha,p\in(0,\infty)$ and $\tilde{p}\in(0,p)$ and yet the embedding $M^{\alpha,p}(\mathbb{N},d,\mu)\hookrightarrow L^{p}(\mathbb{N},\mu)$ fails to be compact. Moreover, since $\lambda\geq1$ in \eqref{doubdim}, it is always possible to choose an exponent $s\in(0,\infty)$ satisfying \eqref{doubdim} and $s>\beta u$, for any fixed $\beta,u\in(0,\infty)$.

\begin{proof}
Fix $\beta,u\in(0,\infty)$, and $v\in(0,\infty]$, and suppose that $s\in(0,\infty)$ is as in \eqref{doubdim}. First we claim that $(X,d)$ is totally bounded. Indeed, if $\tilde{p}=p$ then total boundedness of $(X,d)$ follows from Theorem~\ref{czwarte}. If $\tilde{p}<p$ then $\mu(X)<\infty$ by Theorem~\ref{drugie_plus} which, together with Lemma~\ref{doubbdd} and Proposition~\ref{tot1}, imply that $(X,d)$ is totally bounded. Since $(X,d)$ is totally bounded, $\mu(X)<\infty$ and by Proposition~\ref{tot}, we have $h(1):=\inf_{x\in X} \mu (B_d(x,1))>0$.
%On the other hand, by increasing the value of $s$, if necessary, we can assume that $s>\alpha p$. 
As such, for all $x\in X$ and $r\in(0,1]$, we can take $\lambda:=1/r\in[1,\infty)$ in \eqref{doubdim} to estimate
$$
h(1)4^{-s}\,r^s\leq 4^{-s}\,r^s\mu(B_d(x,1))\leq\mu(B_d(x,r)).
$$
Thus, $\mu$ is locally lower $s$-Ahlfors-regular with exponent $s$ and the desired conclusions now follow from Theorem~\ref{snr-3} and the fact that  $\mu(X)<\infty$.
\end{proof}

\section{Compact embeddings between $M^{\alpha}_{p,q}$ and $N^{\alpha}_{p,q}$ spaces}
\label{sect:betweenbesov}
We now record some results concerning compact embeddings between $M^{\alpha}_{p,q}$ and $N^{\alpha}_{p,q}$ spaces with different choices of exponents.

\begin{tw}\label{Triebel_comp}
Let $(X,d,\mu)$ be a quasi-metric space equipped with a nonnegative Borel measure, and let $\alpha,p \in (0,\infty)$ and $q\in (0,\infty]$. Then the following are equivalent:
\begin{enumerate}
\item The embedding $M^{\alpha}_{p,q}(X,d,\mu) \hookrightarrow L^p(X,\mu)$ is compact.
\item For all $\beta \in (0,\alpha)$, $r\in (0,\infty]$ the embeddings $M^{\alpha}_{p,q}(X,d,\mu) \hookrightarrow M^{\beta}_{p,r}(X,d,\mu)$ and $M^{\alpha}_{p,q}(X,d,\mu) \hookrightarrow N^{\beta}_{p,r}(X,d,\mu)$  are compact.
\item There exist $\beta \in (0,\alpha)$ and $r\in (0,\infty]$ such that at least one of the embeddings $M^{\alpha}_{p,q}(X,d,\mu) \hookrightarrow M^{\beta}_{p,r}(X,d,\mu)$ or $M^{\alpha}_{p,q}(X,d,\mu) \hookrightarrow N^{\beta}_{p,r}(X,d,\mu)$ is compact.
\end{enumerate}
\end{tw}
\begin{rem}
For $p,\alpha\in(0,\infty)$, $\beta\in(0,\alpha)$, and $q=r=\infty$, Theorem~\ref{Triebel_comp} gives that the embedding $M^{\alpha,p}(X,d,\mu) \hookrightarrow L^p(X,\mu)$ is compact if and only if  the embedding $M^{\alpha,p}(X,d,\mu) \hookrightarrow M^{\beta,p}(X,d,\mu)$ is compact.
\end{rem}
\begin{tw} \label{Bes_comp}
Let $(X,d,\mu)$ be a quasi-metric space equipped with a nonnegative Borel measure, and let $\alpha,p \in (0,\infty)$ and $q\in (0,\infty]$. Then the following are equivalent:
\begin{enumerate}
\item The embedding $N^{\alpha}_{p,q}(X,d,\mu) \hookrightarrow L^p(X,\mu)$ is compact.
\item For all $\beta \in (0,\alpha)$, $r\in (0,\infty]$ the embeddings $N^{\alpha}_{p,q}(X,d,\mu) \hookrightarrow M^{\beta}_{p,r}(X,d,\mu)$ and $N^{\alpha}_{p,q}(X,d,\mu) \hookrightarrow N^{\beta}_{p,r}(X,d,\mu)$  are compact.
\item There exist $\beta \in (0,\alpha)$ and $r\in (0,\infty]$ such that at least one of the embeddings $N^{\alpha}_{p,q}(X,d,\mu) \hookrightarrow M^{\beta}_{p,r}(X,d,\mu)$ or $N^{\alpha}_{p,q}(X,d,\mu) \hookrightarrow N^{\beta}_{p,r}(X,d,\mu)$ is compact.
\end{enumerate}
\end{tw}
\begin{proof}[Proof of Theorems~\ref{Triebel_comp} and \ref{Bes_comp}]
Implications ${\it (2)}\implies {\it (3)}, {\it (3)}\implies {\it (1)}$ are trivial and the implication ${\it (1)} \implies {\it (2)}$ follows from the following lemma.

%\begin{lem} \label{Lem_impr_conv}
%If a sequence $\{u_{n}\}_{n\in \mathbb{N}}$ is bounded in $M^{\alpha}_{p,q}(X,d,\mu)$ or $N^{\alpha}_{p,q}(X,d,\mu)$ and is convergent in $L^p(X,\mu)$, then it is convergent in 
%$M^{\beta}_{p,r}(X,d,\mu)$ and $N^{\beta}_{p,r}(X,d,\mu)$, for any $\beta \in (0,\alpha)$, $r\in (0,\infty]$.
%\end{lem}

\begin{lem} \label{Lem_impr_conv}
Suppose that $\{u_{n}\}_{n\in \mathbb{N}}$ is a sequence which is bounded in $M^{\alpha}_{p,q}(X,d,\mu)$ or $N^{\alpha}_{p,q}(X,d,\mu)$ and is convergent in $L^{\tilde{p}}(X,\mu)$ for some $\tilde{p}\in(0,p]$. Then the following two statements are valid for all $\beta \in (0,\alpha)$, $r\in (0,\infty]$.
\begin{enumerate}[(i)]
\item If $\tilde{p}=p$ then $\{u_{n}\}_{n\in \mathbb{N}}$ converges in 
$M^{\beta}_{p,r}(X,d,\mu)$ and $N^{\beta}_{p,r}(X,d,\mu)$.
\item If $\tilde{p}<p$ and $\mu(X)<\infty$ then $\{u_{n}\}_{n\in \mathbb{N}}$ converges in 
$M^{\beta}_{\tilde{p},r}(X,d,\mu)$ and $N^{\beta}_{\tilde{p},r}(X,d,\mu)$.
\end{enumerate} 
\end{lem}
\begin{proof}
Fix $\beta \in (0,\alpha)$ and $r\in (0,\infty]$, and let $A\in \{M,N\}$. Since $\{u_{n}\}_{n\in\mathbb{N}}$ is bounded in $A^{\alpha}_{p,q}(X,d,\mu)$, by assumption, we have that
$C:=\sup_{n\in \mathbb{N}} \|u_n\|_{A^{\alpha}_{p,q}(X,d,\mu)}<\infty$. In light of Proposition~\ref{embs} it suffices to prove the conclusions of this lemma for spaces of the same type with $r=q$. Given that Haj\l{}asz--Besov and Haj\l{}asz--Triebel--Lizorkin spaces are always complete, we will show that for any $\beta \in (0,\alpha)$, the sequence $\{u_{n}\}_{n\in\mathbb{N}}$ is Cauchy in $A^{\beta}_{\tilde{p},q}(X,d,\mu)$.
Fix $\beta \in (0,\alpha)$ and note that for every $n\in\mathbb{N}$, we can find $\overrightarrow{g_{n}}:=\{g_{n,k}\}_{k\in\mathbb{Z}}\in \mathbb{D}^{\alpha}_{d}(u_{n})$ such that
\begin{align*}
\|\overrightarrow{g_{n}}\| \leq \|u_{n}\|_{{A}^{\alpha}_{p,q}(X,d,\mu)} \leq C,
\end{align*}
where $\| \cdot \|$ is $L^p(l^q)$ norm if $A=M$ or $l^q(L^p)$ norm if $A=N$ (see  Subsection~\ref{subsect:fnctspaces} for definitions of these norms). By the definition of $\alpha$-fractional gradient, for each $n\in\mathbb{N}$, we can find a measurable set $E_n \subseteq X$  such that $\mu(E_n) = 0$ and
\begin{align*}
|u_n(x) - u_n(y)| \leq [d(x,y)]^{\alpha}\big(g_{n,k}(x) + g_{n,k}(y)\big),
\end{align*}
for all $k\in \mathbb{Z}$ and $x,y\in X\setminus  E_n$ satisfying $2^{-k-1}\leq d(x,y)<2^{-k}$.

In order to prove that $\{u_{n}\}_{n\in\mathbb{N}}$ is Cauchy in $A^{\beta}_{\tilde{p},q}(X,d,\mu)$, we will construct a particular family of $\beta$-fractional gradients for the functions $u_n-u_m$, with $n,m\in\mathbb{N}$. To this end, we let $E:= \bigcup_{n\in \mathbb{N}}E_{n}$ and define $\overrightarrow{g_{n,m}}:=\{g_{n,m,k}\}_{k\in\mathbb{Z}}:=\overrightarrow{g_{n}} + \overrightarrow{g_{m}}\in \mathbb{D}^{\alpha}_{d}(u_{n}-u_{m})$. Note that
\begin{align*}
\|\overrightarrow{g_{n,m}}\| \leq C \kappa(p) \kappa(q),
\end{align*}
where $\kappa:(0,\infty] \to [1,\infty)$ $\kappa(t) := 1$ for $t\in[1,\infty]$ and $\kappa(t) := 2^{1/t - 1}$ for $t\in (0,1)$.
We fix $\varepsilon \in (0,\infty)$ and choose $K\in \mathbb{Z}$ such that $2^{-K}\leq \varepsilon < 2^{-K+1}$. 

We claim that for all $n,m\in\mathbb{N}$, the sequence $\overrightarrow{h_{n,m}}:=\{h_{n,m,k}\}_{k\in\mathbb{Z}}$, which is by
\begin{center}
$h_{n,m,k}(x) := \begin{cases}
2^{\beta}2^{k\beta}|u_{n}(x) - u_{m}(x)| &\mbox{ if } k<K, \\
\varepsilon^{\alpha-\beta} g_{n,m,k}(x) &\mbox{ if } k\geq K,
\end{cases}$
\end{center}
is an $\alpha$-fractional gradient of the function $u_n-u_m$. Fix $n,m\in\mathbb{N}$ and $k\in\mathbb{Z}$, and suppose that $x,y \in X\setminus E$ satisfy $2^{-k-1} \leq d(x,y) < 2^{-k}$.  If $k\geq K$ then we have
\begin{align*}
|(u_{n} - u_{m})(x) - (u_{n} - u_{m})(y)| &\leq [d(x,y)]^{\alpha}\big(g_{n,m,k}(x) + g_{n,m,k}(y)\big) \\
&\leq [d(x,y)]^{\beta}\big( \varepsilon^{\alpha-\beta} g_{n,m,k}(x) + \varepsilon^{\alpha-\beta}g_{n,m,k}(y)\big).
\end{align*}
On the other hand, if $k<K$ then
\begin{align*}
|(u_{n} - u_{m})(x) - (u_{n} - u_{m})(y)| &\leq [d(x,y)]^{\beta}\big(2^{\beta}2^{k\beta}|(u_{n} - u_{m})(x)| + 2^{\beta}2^{k\beta}|(u_{n} - u_{m})(y)|\big).
\end{align*}
Thus, $\overrightarrow{h_{n,m}} \in \mathbb{D}^{\beta}_{d}(u_{n}-u_{m})$, as wanted.

We now estimate the $L^p(l^q)$ and $l^q(L^p)$ norms of $\overrightarrow{h_{n,m}}$.

\noindent{\tt 1st case:} $A=M$

We start by estimating the $l^{q}(\mathbb{Z})$ norm of $ \{h_{n,m,k}\}_{k\in \mathbb{Z}}$. If $q=\infty$, then for $x\in X\setminus E$
\begin{align*}
\|\{h_{n,m,k}(x) \}_{k\in \mathbb{Z}}\|_{l^{\infty}(\mathbb{Z})} \leq  \frac{2^{\beta}}{\varepsilon^{\beta}}|u_{n}(x) - u_{m}(x)| + \varepsilon^{\alpha-\beta} \|\{g_{n,m,k}(x) \}_{k\in \mathbb{Z}}\|_{l^{\infty}(\mathbb{Z})}.
\end{align*}
Similarly, if $q\in(0,\infty)$ and $x\in X\setminus E$ we have
\begin{align*}
\|\{h_{n,m,k}(x) \}_{k\in \mathbb{Z}}\|_{l^{q}(\mathbb{Z})}^{q} &= \sum_{k = -\infty}^{K-1} 2^{q\beta}2^{k\beta q}|u_{n}(x) - u_{m}(x)|^{q} +  \sum_{k = K}^{\infty}\varepsilon^{(\alpha-\beta)q} g_{n,m,k}^q(x) \\
&\leq  \frac{2^{q\beta}/\varepsilon^{\beta q}}{1-2^{-\beta q}} |u_{n}(x) - u_{m}(x)|^{q} + \varepsilon^{(\alpha-\beta)q} \|\{g_{n,m,k}(x) \}_{k\in \mathbb{Z}}\|_{l^{q}(\mathbb{Z})}^{q}.
\end{align*}
Thus, for all $q\in(0,\infty]$ we have
\begin{align*}
\|\{h_{n,m,k}(x) \}_{k\in \mathbb{Z}}\|_{l^{q}(\mathbb{Z})} \leq \gamma(\beta,q)\kappa(q) \frac{2^{\beta}}{\varepsilon^{\beta}}|u_{n}(x) - u_{m}(x)| + \kappa(q)\varepsilon^{\alpha-\beta} \|\{g_{n,m,k}(x) \}_{k\in \mathbb{Z}}\|_{l^{q}(\mathbb{Z})}.
\end{align*}
where 
\begin{eqnarray}
\gamma(\beta,q) :=
\left\{
\begin{array}{cl}
\frac{1}{(1-2^{-\beta q})^{1/q}} &\mbox{ if } q\in(0,\infty), \\
1 &\mbox{ if } q = \infty.
\end{array}\right.
\end{eqnarray}
Now, taking $L^{\tilde{p}}(X,\mu)$ norm and using the H\"older inequality leads to
\begin{align*}
\|u_{n} - u_{m}\|_{\dot{M}^{\beta}_{{\tilde{p}},q}(X,d,\mu)} &\leq \gamma(\beta,q)\kappa(q)\kappa(\tilde{p}) \frac{2^{\beta}}{\varepsilon^{\beta}}\|u_{n} - u_{m} \|_{L^{\tilde{p}}(X,\mu)} + \kappa(q) \kappa(\tilde{p})\varepsilon^{\alpha-\beta}\|\overrightarrow{g_{n,m}}\|_{L^{\tilde{p}}(X;\ell^{q}(\mathbb{Z}))}\\
&\leq \gamma(\beta,q)\kappa(q)\kappa(\tilde{p}) \frac{2^{\beta}}{\varepsilon^{\beta}}\|u_{n} - u_{m} \|_{L^{\tilde{p}}(X,\mu)}
+\kappa(q)^2 \kappa(p)\kappa(\tilde{p})\lambda(\tilde{p})\varepsilon^{\alpha-\beta}C,
\end{align*}
where 
\begin{eqnarray}
\lambda(\tilde{p}) :=
\left\{
\begin{array}{cl}
1 &\mbox{ if } \tilde{p}=p, \\
\mu(X)^{\frac{1}{\tilde{p}}-\frac{1}{p}} &\mbox{ if } \tilde{p}<p\mbox{ and }\mu(X)<\infty.
\end{array}\right.
\end{eqnarray}
This claim in the case that $A=M$  now follows.

\noindent{\tt 2nd case:} $A=N$. 

We have
\begin{center}
$\|h_{n,m,k}\|_{L^{\tilde{p}}(X,\mu)} = \begin{cases}
2^{\beta}2^{k\beta}\|u_{n} - u_{m} \|_{L^{\tilde{p}}(X,\mu)} &\mbox{ if } k<K, \\
\varepsilon^{\alpha-\beta} \|g_{n,m,k}\|_{L^{\tilde{p}}(X,\mu)} &\mbox{ if } k\geq K.
\end{cases}$
\end{center}
Now if $q=\infty$, then
\begin{align*}
\|\{\|h_{n,m,k}\|_{L^{\tilde{p}}(X,\mu)}\}_{k\in \mathbb{Z}}\|_{l^{\infty}(\mathbb{Z})} \leq \frac{2^{\beta}}{\varepsilon^{\beta}}\|u_{n} - u_{m}\|_{L^{\tilde{p}}(X,\mu)} + \varepsilon^{\alpha-\beta}  \|\{\|g_{n,m,k}\|_{L^{\tilde{p}}(X,\mu)}\}_{k\in \mathbb{Z}}\|_{l^{\infty}(\mathbb{Z})}
\end{align*}
And if $q\in (0,\infty)$
\begin{align*}
\|\{\|h_{n,m,k}\|_{L^p(X,\mu)}\}_{k\in \mathbb{Z}}\|_{l^{q}(\mathbb{Z})}^q \leq \sum_{k = -\infty}^{K-1} 2^{q\beta}2^{k\beta q}\|u_{n} - u_{m}\|_{L^{\tilde{p}}(X,\mu)}^{q} +  \sum_{k = K}^{\infty}\varepsilon^{(\alpha-\beta)q} \|g_{n,m,k}\|_{L^{\tilde{p}}(X,\mu)}^q
\end{align*}
Therefore,  by using the H\"older inequality, we have
\begin{align*}
 \|u_{n} - u_{m}\|_{\dot{N}^{\beta}_{{\tilde{p}},q}(X,d,\mu)} \leq \gamma(\beta,q)\kappa(q) \frac{2^{\beta}}{\varepsilon^{\beta}}\|u_{n} - u_{m} \|_{L^{\tilde{p}}(X,\mu)} + \kappa(q)^2 \kappa(p)\lambda(\tilde{p})\varepsilon^{\alpha-\beta}C.
\end{align*}
This claim in the case that $A=N$  now follows and so the proof of Lemma~\ref{Lem_impr_conv} is now complete.
\end{proof}
This completes the proofs of Theorems~\ref{Triebel_comp} and \ref{Bes_comp}.
\end{proof}
\begin{cor}\label{beta_p/theta}
Let $(X,d,\mu)$ be a quasi-metric-measure space and $\nu$ be a measure such that $\nu<<\mu$ and $\frac{d\nu}{d\mu} \in L^{\theta'}(X,\mu)$ for some $\theta \in (1,\infty)$, where $\theta' = \theta/(\theta - 1)$. Then for any  $\alpha,p \in (0,\infty)$, $q,r\in (0,\infty]$, $\beta \in (0,\alpha)$  the embeddings 
\begin{align*}
M^{\alpha}_{p,q}(X,d,\mu) &\hookrightarrow M^{\beta}_{p/\theta,r}(X,d,\nu), \quad M^{\alpha}_{p,q}(X,d,\mu) \hookrightarrow N^{\beta}_{p/\theta,r}(X,d,\nu), \\ 
N^{\alpha}_{p,q}(X,d,\mu) &\hookrightarrow M^{\beta}_{p/\theta,r}(X,d,\nu),\quad N^{\alpha}_{p,q}(X,d,\mu) \hookrightarrow N^{\beta}_{p/\theta,r}(X,d,\nu)
\end{align*}
are compact.
\end{cor}
\begin{proof}

Let $A\in\{M,N \}$ and let $\{u_{n}\}_{n\in\mathbb{N}}$ be a bounded sequence in $A^{\alpha}_{p,q}(X,d,\mu)$. By Theorem \ref{comp_p/theta} there exists a subsequence $\{u_{n_{j}}\}_{j\in\mathbb{N}}$ converging in $L^{p/\theta}(X,\nu)$. Since $\tfrac{d\nu}{d\mu} \in L^{\theta'}(X,\mu)$,  the H\"older inequality yields
\begin{align*}
L^p(X,\mu) \hookrightarrow L^{p/\theta}(X,\nu).
\end{align*}
Therefore, by Lemma~\ref{Fakt}, sequence $\{u_{n_{j}}\}_{j\in\mathbb{N}}$ is bounded in $A^{\alpha}_{p/\theta,q}(X,d,\nu)$. Thus, by Lemma~\ref{Lem_impr_conv} the sequence 
$\{u_{n_{j}}\}_{j\in\mathbb{N}}$ is converging in $M^{\beta}_{p/\theta,r}(X,d,\nu)$ and $N^{\beta}_{p/\theta,r}(X,d,\nu)$.
\end{proof}

Proceeding similarly to the proof of the previous corollary, we obtain.
\begin{cor}\label{cor1}
Let $(X,d,\mu)$ be a quasi-metric-measure space and $\nu$ be a measure such that $\nu \leq C \mu$ for some $C \in (0, \infty)$ satisfying the conditions listed in Theorem~\ref{drugie}. Then for any $\alpha,p \in (0,\infty)$,  $q,r\in (0,\infty]$ and $\beta\in (0,\alpha)$ the embeddings are compact:
\begin{align*}
M^{\alpha}_{p,q}(X,d,\mu) \hookrightarrow M^{\beta}_{p,r}(X,d,\nu) \mbox{ and } M^{\alpha}_{p,q}(X,d,\mu) \hookrightarrow N^{\beta}_{p,r}(X,d,\nu) \\ 
N^{\alpha}_{p,q}(X,d,\mu) \hookrightarrow M^{\beta}_{p,r}(X,d,\nu) \mbox{ and } N^{\alpha}_{p,q}(X,d,\mu) \hookrightarrow N^{\beta}_{p,r}(X,d,\nu).
\end{align*}
In particular, the above embeddings are compact whenever $(X, d)$ is totally bounded.

On the other hand, if $\mu =\nu$ and if there exists $\delta_0\in(0,\infty)$ such that
$\mu$ is $(C_d,\delta)$-doubling for every $\delta\in(0,\delta_0]$, where $C_d\in[1,\infty)$ is as in \eqref{C-d}, and any one of the above embeddings is compact  for some $p,\beta\in (0,\infty)$,  $q\in (0,\infty]$, and  $\alpha\in (0,\infty)$ with $\alpha\preceq_q{\rm ind}(X,d)$, then $(X,d)$ is totally bounded.
\end{cor}

\begin{cor} \label{beta}
Let $(X,d,\mu)$ be a quasi-metric-measure space and $\nu$ be a finite measure such that $\nu \leq C \mu$ for some constant $C\in (0,\infty)$. Then for any $\alpha,p \in (0,\infty)$, $q,r\in (0,\infty]$, $\beta\in (0,\alpha)$ and $\tilde{p} \in (0,p)$, the embeddings
\begin{align*}
M^{\alpha}_{p,q}(X,d,\mu) &\hookrightarrow M^{\beta}_{\tilde{p},r}(X,d,\nu), \quad M^{\alpha}_{p,q}(X,d,\mu) \hookrightarrow N^{\beta}_{\tilde{p},r}(X,d,\nu), \\ 
N^{\alpha}_{p,q}(X,d,\mu) &\hookrightarrow M^{\beta}_{\tilde{p},r}(X,d,\nu), \quad N^{\alpha}_{p,q}(X,d,\mu) \hookrightarrow N^{\beta}_{\tilde{p},r}(X,d,\nu)
\end{align*} 
are compact. In particular, if  $\mu(X)<\infty$, then the embeddings
\begin{align*}
M^{\alpha}_{p,q}(X,d,\mu) &\hookrightarrow M^{\beta}_{\tilde{p},r}(X,d,\mu), \quad M^{\alpha}_{p,q}(X,d,\mu) \hookrightarrow N^{\beta}_{\tilde{p},r}(X,d,\mu), \\ 
N^{\alpha}_{p,q}(X,d,\mu) &\hookrightarrow M^{\beta}_{\tilde{p},r}(X,d,\mu), \quad N^{\alpha}_{p,q}(X,d,\mu) \hookrightarrow N^{\beta}_{\tilde{p},r}(X,d,\mu)
\end{align*} 
are compact.

On the other hand,  if $\mu=\nu$ and if any one of the above embeddings is continuous for some $p,\beta\in (0,\infty)$, $\tilde{p}\in(0,p)$, $q\in (0,\infty]$, and  $\alpha\in (0,\infty)$ with $\alpha\preceq_q{\rm ind}(X,d)$, then  $\mu(X)<\infty$.
\end{cor}
\begin{proof}
Since $\nu(X)<\infty$ and $\nu\leq C\mu$ we have $\frac{d\nu}{d\mu} \in L^{\theta'}(X,\mu)$ for every $\theta \in [1,\infty)$, where $\theta'=\theta/(\theta-1)$. Therefore, by Corollary~\ref{beta_p/theta} the proof of compactness follows. The second part of the corollary follows from Theorem \ref{drugie_plus}.
\end{proof}
\medskip

{\bf Funding} 
 PG and AS have been supported by (POB Cybersecurity and data analysis) of Warsaw University of Technology within the Excellence Initiative: Research University (IDUB) programme.
\medskip

{\bf Conflict of interest} All authors declare that they have no Conflict of interest.
\medskip

{\bf  Ethical approval} This article does not contain any studies with human participants or animals performed by
the authors.
\medskip

{\bf Data Availability Statement} Data sharing not applicable to this article as no datasets were generated or
analyzed during the current study.

\bigskip
{\small Ryan Alvarado}\\
\small{Department of Mathematics and Statistics,}\\
\small{Amherst College,}\\
\small{Amherst, MA, USA} \\
{\tt rjalvarado@amherst.edu}\\
\\
{\small Przemys{\l}aw  G\'orka}\\
\small{Faculty of Mathematics and Information Sciences,}\\
\small{Warsaw University of Technology,}\\
\small{Pl. Politechniki 1, 00-661 Warsaw, Poland} \\
{\tt przemyslaw.gorka@pw.edu.pl}\\
\\
{\small Artur S{\l}abuszewski}\\
\small{Faculty Mathematics and Information Sciences,}\\
\small{Warsaw University of Technology,}\\
\small{Pl. Politechniki 1, 00-661 Warsaw, Poland} \\
{\tt artur.slabuszewski@pw.edu.pl}\\
\end{document}